\documentclass{amsart}
\usepackage[utf8]{inputenc}
\usepackage[T1]{fontenc}
\usepackage{lmodern}
\usepackage{amsmath}
\usepackage{amssymb}
\usepackage{mathtools}
\usepackage{latexsym}
\usepackage[lite]{amsrefs}
\usepackage{nicefrac}
\usepackage{microtype}
\usepackage{color}
\newcommand{\llp}{\mathrel{\ooalign{\hss$\square$\hss\cr$\diagup$}}}
\usepackage{tikz-cd}
\usepackage{MnSymbol}
\usepackage{enumitem} 
\setlist[enumerate,1]{label=\textup{(\arabic*)}}

\usepackage[all]{xy}
\newdir{ >}{{}*!/-5pt/@{>}}

\usepackage[pdftitle={Bornological K-theory},
pdfauthor={Jack Kelly},
pdfsubject={Mathematics}
    ]{hyperref}
    
    \BibSpec{book}{%
      +{}  {\PrintPrimary}                {transition}
      +{,} { \textit}                     {title}
      +{.} { }                            {part}
      +{:} { \textit}                     {subtitle}
      +{,} { \PrintEdition}               {edition}
      +{}  { \PrintEditorsB}              {editor}
      +{,} { \PrintTranslatorsC}          {translator}
      +{,} { \PrintContributions}         {contribution}
      +{,} { }                            {series}
      +{,} { \voltext}                    {volume}
      +{,} { }                            {publisher}
      +{,} { }                            {organization}
      +{,} { }                            {address}
      +{,} { \PrintDateB}                 {date}
      +{,} { }                            {status}
      +{}  { \parenthesize}               {language}
      +{}  { \PrintTranslation}           {translation}
      +{;} { \PrintReprint}               {reprint}
      +{.} { }                            {note}
      +{.} {}                             {transition}
      +{} { \PrintDOI}                   {doi}
      +{} { available at \url}            {eprint}
      +{}  {\SentenceSpace \PrintReviews} {review}
    }
    
    \renewcommand*{\PrintDOI}[1]{\href{http://dx.doi.org/\detokenize{#1}}{doi: \detokenize{#1}}}

    \newcommand{\comment}[1]{}  
    
    \newcommand{\adj}[4]{#1\negmedspace: #2\rightleftarrows #3:\negmedspace #4}
    
    \theoremstyle{plain}
    \newtheorem{theorem}{Theorem}[section]
    \newtheorem{lemma}[theorem]{Lemma}
    \newtheorem{corollary}[theorem]{Corollary}
        
         \newtheorem{construction}[theorem]{Construction}
    \newtheorem{proposition}[theorem]{Proposition}
    \theoremstyle{remark}
    \newtheorem{remark}[theorem]{Remark}

    \theoremstyle{definition}
    \newtheorem{definition}[theorem]{Definition}
    \newtheorem{example}[theorem]{Example}
    \numberwithin{theorem}{section}

    \newcommand{\heart}{\heartsuit}
    
    \newcommand{\defeq}{\mathrel{:=}} 

    \newcommand{\colim}{\mathrm{colim}}
    
    \DeclarePairedDelimiterX{\setgiven}[2]{\{}{\}}{#1\,{:}\,\mathopen{}#2}

    \DeclareMathAlphabet{\mathpzc}{OT1}{pzc}{m}{it}

\begin{document}
    \title{Proto-exact categories and injective Banach modules}

    \author{Jack Kelly}
    \email{jack.kelly@lincoln.ox.ac.uk}

        \begin{abstract}
We develop the basic theory of covers and envelopes in proto-exact categories. As an application, we prove the existence of enough injectives in categories of seminormed and normed modules over arbitrary seminormed monoids. Over seminormed rings, we also show that certain (proto-)exact structures on categories of seminormed, normed, and Banach modules have enough injectives. In categories of non-Archimedean modules, this includes the quasi-abelian exact structure. We then give a construction of derived functors between categories of seminormed, and Banach modules that also works in the Archimedean quasi-abelian setting, even when the existence of enough injectives is unclear. As an application, we suggest a definition of bounded cohomology with coefficients in arbitrary seminormed rings.
    \end{abstract}
    \maketitle

    \tableofcontents
    
    \section{Introduction}

Let $K$ be a spherically complete Banach field. The classical Hahn-Banach Theorem implies that $K$ is an injective object in the quasi-abelian category $\mathrm{Ban}_{K}$ of Banach spaces over $K$. In other words, any bounded monomorphism $K\rightarrow E$ (necessarily a homeomorphism onto its image) splits. In fact, one can use this fact to prove that $\mathrm{Ban}_{K}$ has enough injectives, with injective cogenerator given by the space $l^{\infty}$ of bounded sequences. 

This article was motivated by the question of whether, over any Banach ring $R$, the quasi-abelian category $\mathrm{Ban}_{R}$ of Banach $R$-modules, with bounded morphisms between them, has enough injectives. The category $\mathrm{Ban}_{R}$ only has finite limits and colimits, so there is no hope of establishing the existence of enough injectives using standard techniques from Grothendieck abelian categories, or exact categories of Grothendieck type \cite{vst2012exact}. However, if we instead consider the category $\mathrm{Ban}_{R}^{\le1}$, where we restrict all maps to have norm at most $1$, then this category is locally presentable.
Unfortunately, in doing this we lose additivity. However,  the category is still very close to being quasi-abelian - it is a \textit{parabelian} pointed category. In particular, it has a canonical proto-exact structure, in the sense of \cite{MR3970975}. This led us to consider the question of the existence of enough injectives in proto-exact categories.


\subsection{Strategies for finding injectives}
Let us summarise the various strategies in this paper for establishing the existence of enough injectives.
\subsubsection{Injectives by deconstructibility}\label{subsubsec:injbydec}
The proto-exact structure on $\mathrm{Ban}_{R}^{\le1}$ is extremely rich - it satisfies a version of the obscure axiom (cf. \cite{Buehler}*{Proposition 2.16}), that is not even satisfied by the category of pointed sets. It turns out that locally presentable proto-exact categories with some version of the obscure axiom, and in which certain filtered colimits are exact, admit a reasonable theory of \textit{deconstructibility}, along the lines of Eklof-Trlifaj and successors:  (\cite{MR1832549}, \cite{MR1798574}, \cite{estrada2014derived}, \cite{Gillespie2}, \cite{gillespie2006flat}, \cite{MR3649814}, \cite{saorin2011exact}, \cite{vst2012exact}). Under an additional assumption called \textit{tameness} we can show that such categories admit enough injectives. The class of tame categories includes, for example, one-sided exact categories. When applied to exact categories, this result is a very moderate simultaneous generalisation of \cite{positselski2023locally} (allowing $\kappa$-presentability rather than finite presentability) and \cite{MR3649814} Lemma 4.14 (allowing exact, rather than just abelian, categories. The proof is similar in spirit to \cite{MR3649814} Lemma 4.14, but the main technical ingredient more closely follows \cite{Gillespie2} Proposition 4.9. Our main application of the Eklof-Trlifaj type argument is to deduce the existence of enough injective non-Archimedean Banach spaces.
\begin{theorem}[Theorem \ref{thm:nainj}]
   Let $R$ be a non-Archimedean Banach ring. Then the quasi-abelian category $\mathrm{Ban}^{nA}_{R}$ of non-Archimedean Banach $R$-modules has enough injectives. 
\end{theorem}
It is worth remarking that the non-Archimedean result happens completely in the additive (exact) setting.
\subsubsection{Approximate injectives and derived functors}
In the additive Archimedean case, the situation is more subtle. This requires a variant strategy to prove the existence of enough injectives in certain (proto-)exact structures. Roughly speaking, we consider one-sided proto-exact structures in which the admissible monomorphisms are generated as a weakly saturated class by a set of admissible monomorphisms. The small object argument then automatically implies the existence of enough injectives.
With this in hand, for each cardinal $\kappa$, we define an exact structure $\mathcal{Q}_{\kappa}$ on $\mathrm{Ban}_{R}$, and show that the resulting exact category has enough injectives. The construction of these injectives happens in the proto-exact world. Precisely, we define a one-sided proto-exact category $(\mathrm{Ban}_{R}^{\le1},\mathcal{Q}_{\kappa}^{\le1})$ (this is really just a category equipped with a QE-mono class in the sense of \cite{positselski2025pure}), and show that this category has enough injectives. As in the proof for the non-Archimedean case, one deduces that $(\mathrm{Ban}_{R},\mathcal{Q}_{\kappa})$ then also has enough injectives. 
The existence of enough injectives then allows us to show that, for $\gamma\ge \kappa$, the natural functor 
$$\mathrm{D}_{-}(\mathrm{Ban}_{R},\mathcal{Q}_{\kappa})\rightarrow \mathrm{D}_{-}(\mathrm{Ban}_{R},\mathcal{Q}_{\gamma}),$$
induced from the exactness of the identity functor
$(\mathrm{Ban}_{R},\mathcal{Q}_{\kappa})\rightarrow (\mathrm{Ban}_{R},\mathcal{Q}_{\gamma})$, has a fully faithful right adjoint. 
We may then regard each category $\mathrm{D}_{-}(\mathrm{Ban}_{R},\mathcal{Q}_{\gamma})$ as a full subcategory of $\mathrm{D}_{-}(\mathrm{Ban}_{R},\mathcal{Q}_{\kappa})$.
Next, we show that for certain cardinals $\kappa$, whenever $X$ is a bounded above complex of $\kappa$-small Banach modules and $Y$ is an arbitrary bounded above complex of Banach modules, we have an isomorphism on hom groups in the derived category:
$$\mathrm{Hom}_{\mathrm{D}(\mathrm{Ban}_{R},\mathcal{Q}_{\kappa})}(X,Y)\cong\mathrm{Hom}_{\mathrm{D}(\mathrm{Ban}_{R},\mathcal{Q})}(X,Y).$$
The collection of such cardinals $\kappa$ forms a cofinal subclass of the class of all cardinals. This also allows us to prove the following.
\begin{theorem}[Theorem \ref{thm:bancolimcard}]
  The category $\mathrm{D}_{-}(\mathrm{XNrm}_{R},\mathcal{Q}_{\max})$ admits a presentation
    $$\colim_{\kappa\in\mathrm{Card}}\mathrm{D}_{-}(\mathrm{XNrm}_{R},\mathcal{Q}_{\kappa})\cong\mathrm{D}_{-}(\mathrm{XNrm}_{R},\mathcal{Q}_{\max}).$$
\end{theorem}
A complex `in each stage' may be resolved by injectives in that stage. In future work we will show that, under certain circumstances, an adjunction
$$\adj{F}{\mathrm{Ban}_{R}}{\mathrm{Ban}_{S}}{G}$$
may still be derived to an adjunction
$$\adj{\mathbb{L}F}{\mathrm{D}_{-}(\mathrm{Ban}_{R})}{\mathrm{D}_{-}(\mathrm{Ban}_{S})}{\mathbb{R}G}$$
of bounded above derived categories. This includes, for example, an adjunction of the form 
$$\adj{\mathrm{triv}}{\mathrm{Ban}_{R}}{\mathrm{Ban}_{R[G]}}{(-)^{G}}$$
where $G$ is a discrete (or normed) group, and $(-)^{G}$ is the functor of taking invariants. This generalises bounded cohomology to arbitrary Banach rings.

\subsubsection{Injectives by effective unions}
Finally, we establish the existence of enough injectives more generally in sufficiently structured proto-exact categories $(\mathcal{C},\mathcal{Q})$ that have so-called $\mathcal{Q}$-\textit{effective unions}. Essentially, this means that if $X\rightarrow Z$ and $Y\rightarrow Z$ are admissible monomorphisms, then the maps from the intersections $X\cap_{Z} Y\rightarrow X$, $X\cap_{Z} Y\rightarrow Y$, and the map from the union $X\cup_{Z} Y\rightarrow Z$ are all admissible monomorphisms. This follows the work of \cite{barr2006categories} and \cite{MR1780498} very closely. 
Our main application of the effective union approach is to pointed seminormed sets.
A pointed seminormed set is a pair $(X,\rho)$ where $X$ is a pointed set, and $\rho:X\rightarrow\mathbb{R}_{\ge0}$ is a map such that $\rho(0_{X})=0$. One defines bounded morphisms of seminormed sets in the obvious way. A seminormed set is said to be \textit{normed}
 if $\rho(x)=0\Rightarrow x=0_{X}$. One may define bounded morphisms of seminormed sets, this gives rise to the categories $\mathrm{SNrm}_{(\mathbb{R}_{\ge0})_{*}}$ and $\mathrm{Nrm}_{(\mathbb{R}_{\ge0})_{*}}$.
 \begin{theorem}[Theorem \ref{thm:injpointedsets}]
    The parabelian categories $\mathrm{SNrm}_{(\mathbb{R}_{\ge0})_{*}}$ and $\mathrm{Nrm}_{(\mathbb{R}_{\ge0})_{*}}$ of seminormed and normed sets, respectively, have enough injectives.
\end{theorem}
\comment{
\subsubsection{Tying the non-Archimedean threads}
Consider the category $\mathrm{SNrm}^{\le1}_{(\mathbb{R}_{\ge0})_{*}}$uniformise notation for pointed normed sets. The category $\mathrm{Ab}(\mathrm{SNrm}^{\le1}_{(\mathbb{R}_{\ge0})_{*}})$ is precisely the category $\mathrm{SNrm}_{\mathbb{Z}}^{\le1,nA}$ of non-Archimedean seminormed abelian groups. In fact if $(\mathcal{C},\mathcal{Q})$ is a finitely complete proto-exact category, then $\mathrm{Ab}(\mathcal{C})$ inherits a canonical deflation proto-exact structure use this terminology later such that the forgetful functor $\mathrm{Ab}(\mathcal{C})\rightarrow\mathcal{C}$ is exact.
}

\subsection{Structure}
The article is structured as follows. In Section \ref{sec:defs} we begin by introducing one-sided generalisations of proto-exact and parabelian categories, and establishing some of their basic properties. We also consider non-pointed versions, in order to connect explicitly with work of Positselski \cite{positselski2025pure}. We introduce the obscure axiom, and consider various notions of generators in proto-exact categories. Finally, we give definitions of cotorsion pairs, covers, and envelopes in proto-exact categories. In Section \ref{sec:cohdec} (cf. also \cite{positselski2025pure}) we generalise the idea of deconstructibility from \cite{saorin2011exact} and \cite{vst2012exact} to proto-exact categories, and explain how deconstructibility, together with tameness, implies the existence of certain covers and envelopes. After giving various exactness conditions on colimits in proto-exact categories, and considering pure monomorphisms in locally presentable proto-exact categories, comes the technical heart of the paper. Specifically, we generalise the notion of coherence of an exact category, introduced in \cite{positselski2023locally}, to proto-exact categories. Our main application here is to give conditions on a locally presentable exact category such that its class of admissible monomorphisms is generated as a weakly saturated class by a set, a result from which deconstructibility can be deduced. We also define injective resolutions in proto-exact categories, and establish some of their rudimentary properties. In Section \ref{sec:bancoh}, we show that various exact categories of seminormed, normed, and Banach modules satisfy the technical assumptions required by the statements in Section \ref{sec:cohdec}. To conclude, in Section \ref{sec:injban}  we deduce various results concerning the existence of enough injectives. We use this to define non-additive derived functors that are well-defined up to isometric isomorphism.

\begin{remark}
    A previous version of this article claimed that the categories of seminormed, normed, and Banach modules over an arbitrary seminormed ring $R$ all have enough injectives. The proof contained a gap, and only works in the additive setting of non-Archimedean modules. We still do not know whether the Archimedean categories have enough injectives. This is partially why the method of constructing enough injectives explained in Subsubsection \ref{subsubsec:injbydec} still only has an exact category as its main example. The construction of Theorem \ref{thm:bancolimcard} is an attempt to partially remedy this.
\end{remark}

\subsection{Statement on the use of Generative AI}
I used a non-pro model to check for grammatical errors and typos, and for literature search. I do not use it to generate text or ideas.
\subsection{Funding declaration}
No external funding to declare.

\subsection{Acknowledgements} 

The author would like to thank Federico Bambozzi, Kobi Kremnizer, and Sven-Ake Wegner for useful comments on earlier versions of this paper.

\section{Primitive-exact categories}\label{sec:defs}

Proto-exact categories, introduced in \cite{MR3970975}, are non-additive versions of Quillen exact categories. Before recalling their basic theory, we here introduce a generalisation that also incorporates Positselski's QE-classes (see \cite{positselski2025pure}, Section 8).

\subsection{Categories with epimorphisms (or monomorphisms)}
Let $\mathcal{C}$ be a category with initial object $\emptyset$ and terminal object $\bullet$.
\begin{definition}
    A morphism $f:X\rightarrow Y$ in $\mathcal{C}$ is an \textit{effective epimorphism} if 
    \begin{enumerate}
        \item 
        the fibre product $X\times_{Y}X$ exists, and
        \item 
        the morphism $f:X\rightarrow Y$ is the coequaliser of the two projections $\pi_{1},\pi_{2}:X\times_{Y}X\rightarrow X$.
    \end{enumerate}
    Dually one defines effective monomorphisms. 
\end{definition}

 For $f:X\rightarrow Y$ a morphism in $\mathcal{C}$ we define $\mathrm{Coker}(f)$ by the following pushout diagram, provided it exists:
\begin{displaymath}
    \xymatrix{
    X\ar[r]^{f}\ar[d] & Y\ar[d]^{g}\\
  \bullet  \ar[r] & \mathrm{Coker}(f).
    }
\end{displaymath}
We define $\mathrm{Ker}(f)$ by the following pullback diagram, provided it exists:
\begin{displaymath}
    \xymatrix{
    \mathrm{Ker}(f)\ar[r]^{i}\ar[d] & X\ar[d]^{f}\\
    \emptyset\ar[r] & Y.
    }
\end{displaymath}

\begin{definition}
    Let $\mathcal{C}$ be a category with initial and terminal objects. By a \textit{class of regular epimorphic short exact sequences} $\mathcal{Q}$, we mean a collection of pairs $(f,g)$, where $g:Y\rightarrow Z$ is a regular epimorphism, and $f$ is isomorphic to the morphism $\mathrm{Ker}(g)\rightarrow Y$. Dually one defines a \textit{class of regular monomorphic short exact sequences}.
\end{definition}

Given a class $\mathcal{Q}$ of either regular epimorphic or regular monomorphic short exact sequences in a category $\mathcal{C}$, we write $\mathcal{M}_{\mathcal{Q}}=\{f:(f,g)\in\mathcal{Q}\}$ for the class of \textit{admissible monomorphisms} (relative to $\mathcal{Q}$), and $\mathcal{E}_{\mathcal{Q}}=\{g:(f,g)\in\mathcal{Q}\}$ for the class of \textit{admissible epimorphisms} (relative to $\mathcal{Q}$). If $(f:X\rightarrow Y,g:Y\rightarrow Z)\in\mathcal{Q}$, we will equivalently say that
\begin{displaymath}
          \xymatrix{
            \emptyset\ar[r] & X\ar[r]^{f} & Y\ar[r]^{g} & Z\ar[r] & \bullet
            }
        \end{displaymath}
is a short exact sequence relative to $\mathcal{Q}$. 

\begin{definition}
    A \textit{deflation primitive-exact category} is a pair $(\mathcal{C},\mathcal{Q})$, where $\mathcal{C}$ is a category with initial and terminal objects, and $\mathcal{Q}$ is a class of regular epimorphic short exact sequences such that
    \begin{enumerate}
        \item 
        $\mathcal{E}_{\mathcal{Q}}$ is closed under composition,
        \item 
        $\mathcal{E}_{\mathcal{Q}}$ contains all identity morphisms, and
        \item 
        $\mathcal{E}_{\mathcal{Q}}$ is stable by pullback along all morphisms in $\mathcal{M}_{\mathcal{Q}}$.
    \end{enumerate}
    One dually defines an \textit{inflation primitive-exact category}. By a \textit{primitive-exact category} we mean either a deflation primitive-exact category or an inflation primitive-exact category.
\end{definition}
\begin{definition}[\cite{positselski2025pure}*{After Lemma 6.1}]
    A \textit{category with QE-epi class} is a deflation primitive-exact category $(\mathcal{C},\mathcal{Q})$ such that pullbacks of morphisms in $\mathcal{E_{Q}}$ along arbitrary morphisms exist, and are in $\mathcal{E_{Q}}$. Dually one defines a \textit{category with QE-mono class}. 
\end{definition}

Since we will usually be working in pointed categories, from now on we will always write
$$A\vee B$$
for the coproduct of two objects, and 
$$A\times B$$
for their product.
\begin{proposition}
    \begin{enumerate}
        \item Let $(\mathcal{C},\mathcal{Q})$ be a category with QE-mono class. Then for any two objects $A$ and $B$, the sequence
        $$\emptyset\rightarrow A\rightarrow A\times B\rightarrow B\rightarrow\bullet$$
        is exact.
             \item  Let $(\mathcal{C},\mathcal{Q})$ be a category with QE-epi class. Then for any two objects $A$ and $B$, the sequence
        $$\emptyset\rightarrow A\rightarrow A\vee B\rightarrow B\rightarrow\bullet$$
        is exact.
    \end{enumerate}
\end{proposition}

\begin{proof}
    Consider the exact sequence
    $$\emptyset\rightarrow A\rightarrow A\rightarrow \bullet\rightarrow \bullet.$$
    Taking the pullback of the map $A\rightarrow \bullet$ along $B\rightarrow \bullet$ gives claim (1). Claim (2) is dual. 
    
\end{proof}

\subsubsection{(One-sided) proto-exact categories}

\begin{definition}
    A (deflation/inflation) proto-exact category is a (deflation/inflation) primitive-exact category $(\mathcal{C},\mathcal{Q})$ such that whenever $(f,g)\in\mathcal{Q}$ then $f$ is the kernel of $g$, and $g$ is the cokernel of $f$. By a \textit{one-sided proto-exact category} we mean either a deflation proto-exact category or an inflation proto-exact category.
\end{definition}
\begin{proposition}
    A one-sided proto-exact category is necessarily pointed.
\end{proposition}
\begin{proof}
    Suppose that $(\mathcal{C},\mathcal{Q})$ is a deflation primitive-exact category that is a one-sided proto-exact category. The case that $(\mathcal{C},\mathcal{Q})$ is an inflation primitive-exact category is similar. Let $A$ be any object. We consider the kernel of the identity $A\rightarrow A$. This is just the canonical morphism $\emptyset\rightarrow A$. The cokernel of this map is $A\rightarrow A\coprod\bullet$. By assumption, this map is an isomorphism. This is true for all $A$, and implies that $\bullet$ must be the initial object. 
\end{proof}

Many, but not all, of our examples of one-sided proto-exact categories come from proto-exact categories. Here we recall the definitions of proto-exact and parabelian categories, and their basic properties. We follow the setup in \cite{mozgovoy}.

Let $\mathcal{C}$ be a pointed category. A pair $(f,g)=(f:X\rightarrow Y,g:Y\rightarrow Z)$ of morphisms in $\mathcal{C}$ is said to be \textit{strict} if $f$ is isomorphic to $\mathrm{Ker}(g)\rightarrow Y$ and $g$ is isomorphic to $Y\rightarrow\mathrm{Coker}(f)$. We call the map $g$ appearing in such a pair a \textit{strict epimorphism}, and the map $f$ appearing in such a pair a \textit{strict monomorphism}. We have the following useful result (see \cite{mozgovoy}*{after Lemma 1.2}).

\begin{proposition}[\cite{mozgovoy}]\label{prop:strictnessequiv}
    Let $\mathcal{C}$ be a pointed category with all kernels and cokernels. Let $f:X\rightarrow Y$ be a morphism in $\mathcal{C}$.
    \begin{enumerate}
        \item 
        The following are equivalent.
        \begin{enumerate}
            \item 
            The morphism $f$ is a strict epimorphism.
            \item 
            The morphism $f$ is isomorphic to $X\rightarrow \mathrm{Coker}(\mathrm{Ker}(f)\rightarrow X).$
             \item 
            There is a morphism $g:K\rightarrow X$ such that $f\circ g=0$ and $f$ is isomorphic to $X\rightarrow \mathrm{Coker}(g)$.
        \end{enumerate}
\item 
The following are equivalent.
     \begin{enumerate}
            \item 
            The morphism $f$ is a strict monomorphism.
            \item 
            The morphism $f$ is isomorphic to $\mathrm{Ker}(Y\rightarrow\mathrm{Coker}(f))\rightarrow Y$.
            \item 
            There is a morphism $g:Y\rightarrow Z$ such that $g\circ f=0$ and $f$ is isomorphic to $\mathrm{Ker}(g)\rightarrow Y$.
        \end{enumerate}
    \end{enumerate}
\end{proposition}

By a \textit{sequence} in a pointed category $\mathcal{C}$ we will mean a diagram

    \begin{displaymath}
            \xymatrix{
         X\ar[r]^{f} & Y\ar[r]^{g} & Z
            }
        \end{displaymath}
        such that $g\circ f=0$. 

\begin{definition}[\cite{MR3970975} Definition 2.4.2]
    A \textit{proto-exact category} is a pair $(\mathcal{C},\mathcal{Q})$ where
    \begin{enumerate}
        \item $\mathcal{C}$ is a pointed category, and
        \item $\mathcal{Q}$ is a class of strict pairs
    \end{enumerate}
    such that 
    \begin{enumerate}
        \item 
        $\mathrm{Id}_{X}\in\mathcal{M}_{\mathcal{Q}}\cap\mathcal{E}_{\mathcal{Q}}$ for all $X\in\mathcal{C}$;
        \item 
        $\mathcal{M}_{\mathcal{Q}}$ and $\mathcal{E}_{\mathcal{Q}}$ are closed under composition;
        \item 
        $\mathcal{E}_{\mathcal{Q}}$ is closed under pullbacks along $\mathcal{M}_{\mathcal{Q}}$, and
        \item 
        $\mathcal{M}_{\mathcal{Q}}$ is closed under pushouts along $\mathcal{E}_{\mathcal{Q}}$.
    \end{enumerate}
    A proto-exact category $(\mathcal{C},\mathcal{Q})$ is said to be \textit{right totally proto-exact} if $\mathcal{E}_{\mathcal{Q}}$ is closed under pullbacks along all morphisms, \textit{left totally proto-exact} if $\mathcal{M}_{\mathcal{Q}}$ is closed under pushouts along all morphisms, and \textit{totally proto-exact} if it is both left totally proto-exact and right totally proto-exact.
\end{definition}

\begin{example}
    If $(\mathcal{C},\mathcal{Q})$ is a left totally proto-exact category, then the class $\mathbf{AdMon}$ of admissible monomorphisms is a QE-mono class.
\end{example}

The next result lends further evidence to our sense that the only sensible notion of a `two-sided' proto-exact structure (potentially on an unpointed category), is just a proto-exact structure on a pointed category in the usual sense.
\begin{proposition}
  Let $\mathcal{C}$ be a category with initial object $\emptyset$ and terminal object $\bullet$. Suppose that there exists the structure of an inflation primitive-exact category $(\mathcal{C},\mathcal{Q}_{m})$, and a deflation primitive-exact category $(\mathcal{C},\mathcal{Q}_{e})$. Suppose further that $\mathcal{E}_{\mathcal{Q}_{m}}=\mathcal{E}_{\mathcal{Q}_{e}}$ and $\mathcal{M}_{\mathcal{Q}_{m}}=\mathcal{M}_{\mathcal{Q}_{e}}.$ Then the following hold.
  \begin{enumerate}
      \item 
      The category $\mathcal{C}$ is pointed.
      \item 
      We have $\mathcal{Q}_{m}=\mathcal{Q}_{e}=\mathcal{Q}$, and $(\mathcal{C},\mathcal{Q})$ is proto-exact.
  \end{enumerate}
\end{proposition}
\begin{proof}
    Consider the morphism $\emptyset\rightarrow\bullet$. This is the kernel of $\bullet\rightarrow \bullet$. Since identity morphisms are in $\mathcal{E}_{\mathcal{Q}_{e}}$, $\emptyset\rightarrow\bullet$ is in $\mathcal{M}_{\mathcal{Q}_{e}}=\mathcal{M}_{\mathcal{Q}_{m}}$. It is therefore a regular monomorphism. Dually, it is a (regular) epimorphism. Consequently, it is an isomorphism. Claim (2) is clear.
\end{proof}

\begin{remark}
Note that a Quillen exact category is precisely an additive totally proto-exact category. Totality is important here!
    \end{remark}

\smallskip

\begin{proposition}\label{prop:pullker}
    Let $\mathcal{C}$ be a pointed category, and let $f:X\rightarrow Z$, $g:Y\rightarrow Z$ be morphisms such that $\mathrm{Ker}(f)$, $\mathrm{Ker}(g)$, and the pullback diagram
       \begin{displaymath}
             \xymatrix{
             X\times_{Z}Y\ar[d]^{g'}\ar[r]^{f'} & Y\ar[d]^{g}\\
             X\ar[r]^{f} & Z.
             }
         \end{displaymath}
         all exist. Then the natural map $\mathrm{Ker}(f')\rightarrow\mathrm{Ker}(f)$ is an isomorphism. An obvious dual claim holds for pushouts and cokernels.
\end{proposition}

\begin{proof}
By using the Yoneda Lemma, we may assume that $\mathcal{C}$ is the category of pointed sets. Then $X\times_{Z}Y=\{(x,y):f(x)=g(y)\}$. The map $f':X\times_{Z}Y\rightarrow Y$ is the projection to $Y$. Then $f'(x,y)=0_{Y}$ if and only if $y=0_{Y}$. Then we must have $f(x)=g(y)=g(0_{Y})=0_{Z}$, so $x\in\mathrm{Ker}(f)$. Conversely, any element of the form $(x,0_{Y})$ with $x\in\mathrm{Ker}(f)$ is clearly in $\mathrm{Ker}(f')$. Thus $\mathrm{Ker}(f')=\{(x,0_{Y}):x\in\mathrm{Ker}(f)\}$. By projection to the $X$ factor, this is clearly in bijection with $\mathrm{Ker}(f)$.
\end{proof}

We will make use of the following result.

     \begin{lemma}\label{lem:pullbacker}
         Let $(\mathcal{C},\mathcal{Q})$ be a proto-exact category. Let $f:B'\rightarrow B$ be an admissible epimorphism, and $i:A\rightarrow B$ an admissible monomorphism. Set $A'=\mathrm{Ker}(B'\rightarrow\mathrm{Coker}(i))$. Then the following diagram is bicartesian, $i'$ is an admissible monomorphism, and $f'$ is an admissible epimorphism:
         \begin{displaymath}
             \xymatrix{
             A'\ar[d]^{i'}\ar[r]^{f'} & A\ar[d]^{i}\\
             B'\ar[r]^{f} & B.
             }
         \end{displaymath}
     \end{lemma}

     \begin{proof}
     This is really a rewording of the fact that the pullback square of an admissible monomorphism along an admissible epimorphism is bicartesian.
     
     Note that $A'$ exists since, as the composition of the admissible epimorphisms $B'\rightarrow B$ and $B\rightarrow\mathrm{Coker}(i)$, $B'\rightarrow\mathrm{Coker}(i)$ is an admissible epimorphism.
        Consider the pullback square
         \begin{displaymath}
             \xymatrix{
             P\ar[d]^{i''}\ar[r]^{f''} & A\ar[d]^{i}\\
             B'\ar[r]^{f} & B.
             }
         \end{displaymath}
         Then $i''$ is an admissible monomorphism, and $f''$ is an admissible epimorphism. Moreover by \cite{mozgovoy}*{Remark 1.7 (3)}, the square is bicartesian. Let $C=\mathrm{Coker}(i'')$. By Proposition \ref{prop:pullker}, we then get a commutative diagram with exact columns:
           \begin{displaymath}
             \xymatrix{
             P\ar[d]^{i''}\ar[r]^{f''} & A\ar[d]^{i}\\
             B'\ar[r]^{f}\ar[d] & B\ar[d]\\
             C\ar@{=}[r] & C.
             }
         \end{displaymath}
         Hence, $C\cong\mathrm{Coker}(i)$, and we have $P\cong\mathrm{Ker}(B'\rightarrow C)\cong\mathrm{Ker}(B'\rightarrow\mathrm{Coker}(i))$. 
     \end{proof}

     \begin{proposition}\label{prop:cokernelcomp}
         Let $(\mathcal{C},\mathcal{Q})$ be a proto-exact category, and let $f:X\rightarrow Y$ and $g:Y\rightarrow Z$ be admissible monomorphisms. Then we have an exact sequence
         $$0\rightarrow\mathrm{Coker}(f)\rightarrow\mathrm{Coker}(g\circ f)\rightarrow\mathrm{Coker}(g)\rightarrow0.$$
     \end{proposition}

\begin{proof}
    Consider the following diagram in which all squares and hence the entire rectangle are pushouts
    \begin{displaymath}
        \xymatrix{
        X\ar[d]\ar[r]^{f} & Y\ar[d]\ar[r]^{g} & Z\ar[d]\\
        0\ar[r] & \mathrm{Coker}(f)\ar[r] & \mathrm{Coker}(g\circ f).
        }
    \end{displaymath}

    In particular, we have that $\mathrm{Coker}(f)\rightarrow\mathrm{Coker}(g\circ f)$ is an admissible monomorphism, and the cokernel of this map is isomorphic to $\mathrm{Coker}(g)$.
\end{proof}

\begin{remark}\label{rem:pushoutcoker}
    The same pushout diagram as in the proof of Proposition \ref{prop:cokernelcomp} shows that, in any finitely cocomplete pointed category, and for any composable maps $f:X\rightarrow Y$ and $g:Y\rightarrow Z$, we have $\mathrm{Coker}(g)\cong\mathrm{Coker}(\mathrm{Coker}(f)\rightarrow\mathrm{Coker}(g\circ f))$.
\end{remark}

In the additive world, a category $\mathpzc{E}$ is called \textit{quasi-abelian} if the class of all kernel-cokernel pairs determines a Quillen exact structure on $\mathpzc{E}$. The analogue in the non-additive world is the notion of a \textit{parabelian category}.

\begin{definition}[\cite{mozgovoy}*{Definition 1.9}]
    A pointed category $\mathcal{C}$ with all kernels and cokernels is called \textit{parabelian} if the class of all strict pairs determines a proto-exact structure. 
\end{definition}

\begin{remark}
    Let $\mathcal{C}$ be parabelian, we will write $\mathcal{Q}_{max}$ for the proto-exact structure on $\mathcal{C}$ given by all strict pairs.
\end{remark}

When $\mathcal{C}$ is a parabelian category, unless we state otherwise, we will assume that it is equipped with its canonical parabelian proto-exact structure.

By \cite{mozgovoy}*{Theorem 1.10}, a pointed category $\mathcal{C}$ with kernels and cokernels is parabelian precisely if strict epimorphisms are closed under pullbacks along strict monomorphisms, and strict monomorphisms are closed under pushouts along strict epimorphisms.

\begin{definition}
    A parabelian category $\mathcal{C}$ is called \textit{(left/right) totally parabelian} if the parabelian proto-exact structure is (left/right) total.
\end{definition}

In particular, an additive totally parabelian category is precisely a quasi-abelian category.

\begin{example}
    Let $\mathrm{Set}_{*}$ denote the category of pointed sets. A monomorphism is just an injective base-point-preserving map of sets. All monomorphisms are strict. A strict epimorphism $f:X\rightarrow Y$ is precisely a surjection that is injective when restricted to $X\setminus f^{-1}(0_{Y})$. The category $\mathrm{Set}_{*}$ is a parabelian category \cite{mozgovoy}*{Section 2.1}. It is obviously left totally parabelian, but we claim that it is \textit{not totally parabelian}. Let $f:X\rightarrow Y$ be a strict epimorphism, and $g:Z\rightarrow Y$ any map. We have that $X\times_{Y}Z$ is the subset of $X\times Z$ consisting of those pairs $(x,z)$ such that $f(x)=g(z)$. The map $f':X\times_{Y}Z\rightarrow Z$ is the projection. Suppose there is $z\neq 0$ such that $g(z)=0_{Y}$. Suppose further there are distinct $x_{1}\neq x_{2}\in X$ with $f(x_{1})=f(x_{2})=0_{Y}$. Then $(x_{1},z)\neq (x_{2},z)$, they are both not in $f'^{-1}(0_{Z})$, but they both map to $z$ under $f'$. Hence $f'$ is not a strict epimorphism. For an explicit example, take $X=\{0_{X},x_{1},x_{2}\}$, $Z=\{0_{Z},z\}$, and $Y=\{0_{Y}\}$. The only choice for $f$ and $g$ is to send everything to zero. Then in fact both $f$ and $g$ are strict epimorphisms, but $f'$ is not. These issues vanish if $g$ is required to be injective. 
\end{example}

\begin{example}
    Let $\mathcal{B}$ be a belian category in the sense of \cite{deitmar2016belian}*{Definition 1.1.3}. If $\mathcal{B}$ has enough injectives and enough projectives in the sense of \cite{deitmar2016belian}*{Definition 1.3.1}, then $\mathcal{B}$ is parabelian by \cite{deitmar2016belian}*{Lemma 1.3.2}.
\end{example}

\subsubsection{Some useful facts}

Here we collect some useful facts concerning the detection of isomorphisms in proto-exact categories.

\begin{proposition}\label{prop:admonepiisio}
    Let $f:X\rightarrow Y$ be a morphism in a proto-exact category $(\mathcal{C},\mathcal{Q})$ that is an admissible monomorphism and an epimorphism. Then it is an isomorphism.
\end{proposition}

\begin{proof}
As an admissible monomorphism, $f$ is, in particular, a regular monomorphism. Since it is also an epimorphism, it is an isomorphism, by \cite{MR2377903}*{Proposition 16.16}.
\end{proof}

\begin{proposition}\label{prop:adepikerzero}
    Let $f:X\rightarrow Y$ be a morphism in a proto-exact category $(\mathcal{C},\mathcal{Q})$. If $f$ is an admissible epimorphism and $\mathrm{Ker}(f)\cong 0$, then $f$ is an isomorphism.
\end{proposition}

\begin{proof}
    Since $f$ is an admissible epimorphism, we have an exact sequence
    $$0\rightarrow 0\rightarrow X\rightarrow Y\rightarrow 0.$$
    Hence $Y\cong\mathrm{Coker}(0\rightarrow X)\cong X$.
   
\end{proof}

\subsubsection{Exact functors}

Here we give some definitions of exact functors between one-sided proto-exact categories.

  \begin{definition}
    Let $(\mathcal{C},\mathcal{Q})$ be a primitive-exact category.
        A sequence 
        \begin{displaymath}
            \xymatrix{
            \emptyset\ar[r] & X\ar[r]^{f} & Y\ar[r]^{g} & Z\ar[r] & \bullet
            }
        \end{displaymath}
        for which $\mathrm{Ker}(g)$, $\mathrm{Coker}(g)$, and $\mathrm{Im}(g)$ exist, is said to be 
        \begin{enumerate}
            \item 
                   \textit{left semi-exact (relative to $\mathcal{Q}$)} if 
        \begin{enumerate}
            \item 
            the map $g:Y\rightarrow Z$ is an effective epimorphism,
            \item 
            the map $f$ is an admissible monomorphism, 
            \item 
            and $f$ is isomorphic to the canonical map $\mathrm{Ker}(g)\rightarrow Y$.
        \end{enumerate}
        \item 
               \textit{right semi-exact (relative to $\mathcal{Q}$)} if 
        \begin{enumerate}
            \item
           the map $f:X\rightarrow Y$ is an effective monomorphism,
            \item 
            the map $g$ is an admissible epimorphism, 
            \item and $g$ is isomorphic to the canonical map $Y\rightarrow \mathrm{Coker}(f)$.
        \end{enumerate}
        \end{enumerate}
    \end{definition}

Note that a sequence is exact precisely if it is both left and right semi-exact.

\begin{definition}
Let $(\mathcal{C},\mathcal{Q})$, $(\mathcal{D},\mathcal{R})$ be proto-exact categories.
    A functor 
    $$F:\mathcal{C}\rightarrow\mathcal{D}$$
    is said to be 
    \begin{enumerate}
    \item 
        \textit{left semi-exact} (relative to $\mathcal{Q}$ and $\mathcal{R}$) if it sends exact sequences in $(\mathcal{C},\mathcal{Q})$ to left semi-exact sequences in $(\mathcal{D},\mathcal{R})$;
             \item 
        \textit{right semi-exact} (relative to $\mathcal{Q}$ and $\mathcal{R}$) if it sends exact sequences in $(\mathcal{C},\mathcal{Q})$ to right semi-exact sequences in $(\mathcal{D},\mathcal{R})$;
        \item
           \textit{exact} (relative to $\mathcal{Q}$ and $\mathcal{R}$) if it sends exact sequences in $(\mathcal{C},\mathcal{Q})$ to exact sequences in $(\mathcal{D},\mathcal{R})$.
    \end{enumerate}
\end{definition}
Note that $F$ is exact if and only if it is both left semi-exact and right semi-exact.

\subsubsection{Exact sub-categories}

\begin{definition}
    Let $(\mathcal{D},\mathcal{R})$ be a proto-exact category. A \textit{proto-exact subcategory of} $(\mathcal{D},\mathcal{R})$ is a proto-exact category $(\mathcal{C},\mathcal{Q})$ equipped with an exact functor
    $$i:(\mathcal{C},\mathcal{Q})\rightarrow(\mathcal{D},\mathcal{R})$$ where $i$ is the inclusion of a subcategory.
\end{definition}

\begin{definition}
A \textit{subexact category} is a proto-exact subcategory equipped with an exact functor
   $$i:(\mathcal{C},\mathcal{Q})\rightarrow(\mathcal{D},\mathcal{R})$$ 
   where $i$ is the inclusion of a subcategory, and $(\mathcal{D},\mathcal{R})$ is an exact category.
\end{definition}

\begin{example}
    Even in subexact, totally proto-exact categories, various useful lemmata from the theory of exact categories cannot be expected to hold. Let  $$i:(\mathcal{C},\mathcal{Q})\rightarrow(\mathcal{D},\mathcal{R})$$ be subexact, and let $A,B$ be objects of $\mathcal{C}$. Consider the natural map $c_{A,B}:A\vee B\rightarrow A\times B$ from the coproduct to the product. This map satisfies the property that $i(c_{A,B})$ is an isomorphism in $\mathpzc{D}$. Indeed, we have the following commutative diagram of exact sequences

    \begin{displaymath}
        \xymatrix{
        0\ar[r] & A\ar@{=}[d]\ar[r] & A\vee B\ar[d]\ar[r] & B\ar@{=}[d]\ar[r] & 0\\
        0\ar[r] & A\ar[r] & A\times B\ar[r] & B\ar[r] & 0.
        }
    \end{displaymath}
  Consequently, $c_{A,B}$ is both a monomorphism and an epimorphism in $\mathcal{C}$. It is not an isomorphism in $\mathcal{C}$ unless $\mathcal{C}$ is a pre-additive subcategory. In particular, the familiar $3\times 3$ lemma does not often hold. 
\end{example}

\subsection{The obscure axiom}

We shall need a version of the obscure axiom (\cite{Buehler}*{Proposition 2.16}) that, for proto-exact categories, generally needs to be imposed. It does not follow from the axioms - for example, it is not even satisfied for pointed sets.

\begin{definition}
    A primitive-exact category $(\mathcal{C},\mathcal{Q})$ is said to 
    \begin{enumerate}
        \item 
        \textit{satisfy the left obscure axiom if}
        whenever a morphism $i$ has a cokernel and there exists a morphism $j$ such that $j\circ i$ is an admissible monomorphism, then $i$ is an admissible monomorphism;
        \item
             \textit{satisfy the right obscure axiom if}
        whenever a morphism $e$ has a kernel and there exists a morphism $j$ such that $e\circ j$ is an admissible epimorphism, then $e$ is an admissible epimorphism;
              \item
             \textit{satisfy the obscure axiom if}
        it satisfies both the left and right obscure axioms.
    \end{enumerate}
\end{definition}

\begin{definition}
       A primitive-exact category is said to 
    \begin{enumerate}
        \item 
        \textit{satisfy the strong left obscure axiom if}
        whenever $j\circ i$ is an admissible monomorphism, then $i$ is an admissible monomorphism;
        \item
             \textit{satisfy the strong right obscure axiom if}
        whenever $e\circ j$ is an admissible epimorphism, then $e$ is an admissible epimorphism;
              \item
             \textit{satisfy the strong obscure axiom if}
        it satisfies both the left and right strong obscure axioms.
    \end{enumerate}
\end{definition}

\begin{remark}
    In the context of a category with QE-epi class $(\mathcal{C},\mathcal{Q})$, Positselski says that $\mathcal{E_{Q}}$ is a \textit{strong QE-epi class} (\cite{positselski2025pure}*{Section 10}) if it satisfies the strong right obscure axiom in the sense above.
\end{remark}

Note that if $\mathcal{C}$ is finitely cocomplete (resp. finitely complete), then $(\mathcal{C},\mathcal{Q})$ satisfies the strong left (resp. right) obscure axiom if and only if it satisfies the left (resp. right) obscure axiom.

The two-sided obscure axiom seems to be quite rarely satisfied, even in the complete and cocomplete setting. As we shall see later, it is satisfied in our example of interest, $\mathrm{Ban}_{R}^{\le1}$. Our impression is that this is somehow to do with the fact that $\mathrm{Ban}_{R}^{\le1}$ is very close to being additive. 

\begin{example}
    Consider the parabelian category $\mathrm{Set}_{*}$ of pointed sets. Note that this category is complete and cocomplete. Since the strict monomorphisms are precisely the injections, the parabelian category $\mathrm{Set}_{*}$ clearly satisfies the left obscure axiom. On the other hand, a morphism $h:A\rightarrow B$ of pointed sets is a strict epimorphism precisely if it is surjective, and injective away from $A\setminus h^{-1}(\{0_{B}\})$. Let $f:X\rightarrow Y$ and $g:Y\rightarrow Z$ be morphisms such that $g\circ f$ is a strict epimorphism. Then $g\circ f$ is surjective, and injective away from $(g\circ f)^{-1}(\{0_{Z}\})$. Now, $g$ will be surjective, but may not necessarily be injective away from $g^{-1}(\{0_{Z}\})$. Indeed, consider a pointed set $X$ with one non-zero element $x$ (so two elements in total), and put $X=Z$. Let $X\subset Y$ with $Y$ strictly larger than $X$. Let $f$ be the inclusion, and $g:Y\rightarrow X$ the map which sends all non-zero elements to $x$. Then $g\circ f$ is the identity, but $g$ is not a strict epimorphism. Hence $\mathrm{Set}_{*}$ \textit{does not satisfy the right obscure axiom}.
\end{example}

\subsubsection{Split morphisms}

In the theory of exact categories, there is always an `initial' exact structure consisting of all split exact sequences. In the non-additive setting, this becomes much more complicated.

\begin{definition}
    Let $\mathcal{C}$ be a category. A morphism $i:X\rightarrow Y$ is said to be \textit{a split monomorphism } if there is a morphism $j:Y\rightarrow X$ such that $j\circ i=\mathrm{Id}_{X}$. Dually, the morphism $j$ is then called a \textit{split epimorphism}.
 \end{definition}

\begin{remark}
    Let $(\mathcal{C},\mathcal{Q})$ be a primitive-exact category. Unlike the exact case, it is not automatic that split monomorphisms/split epimorphisms are admissible. Note, however, that if $(\mathcal{C},\mathcal{Q})$ satisfies the strong left (resp. right) obscure axiom, then split monomorphisms (resp. split epimorphisms) are admissible.
\end{remark}

\begin{example}
    Consider again the example $\mathrm{Set}_{*}$. Clearly, split monomorphisms are strict monomorphisms. On the other hand, any pointed epimorphism is split (assuming the Axiom of Choice). However, pointed epimorphisms need not be strict. 
\end{example}

\subsubsection{$\mathcal{E}$-admissible generators}

It will be important to have various notions of generators in primitive-exact categories. As we shall see, the obscure axiom features in this discussion. We fix a category $\mathcal{C}$, and a class $\mathcal{E}$ of epimorphisms in $\mathcal{C}$.

\begin{definition}
Let $\mathcal{C}$ be a category. A collection $\mathcal{G}$ of objects of $\mathcal{C}$ is said to be an $\mathcal{E}$-\textit{generating class}
    if, for any object $C\in\mathcal{C}$ there is a subset $\{G_{i}\}_{i\in\mathcal{I}}\subset\mathcal{G}$ and a morphism
    $$\coprod_{i\in\mathcal{I}}G_{i}\rightarrow C$$ 
    in $\mathcal{E}$.
\end{definition}

\begin{lemma}\label{lem:adepicheckgen}
    Suppose that a category $\mathcal{C}$ has a $\mathcal{E}$-generating class $\mathcal{G}$ and that $\mathcal{E}$ satisfies the strong right obscure axiom - that is, if $g\circ f\in\mathcal{E}$ then $g\in\mathcal{E}$. Then, if $\mathrm{Hom}(G,X)\rightarrow\mathrm{Hom}(G,Y)$ is an epimorphism of sets for all $G\in\mathcal{G}$, the map $X\rightarrow Y$ is in $\mathcal{E}$.
 \end{lemma}

 \begin{proof}
Suppose that $\mathrm{Hom}(G,X)\rightarrow\mathrm{Hom}(G,Y)$ is an epimorphism of sets for all $G\in\mathcal{G}$. Pick a morphism $\coprod_{i\in\mathcal{I}} G_{i}\rightarrow Y$ in $\mathcal{E}$, with all $G_{i}\in\mathcal{G}$. This must then factor through a map $\coprod G\rightarrow X$. Hence $X\rightarrow Y$ is in $\mathcal{E}$.
 \end{proof}

 \begin{proposition}\label{prop:canonicaladepi}
    Suppose that a category $\mathcal{C}$ has an $\mathcal{E}$-generating class $\mathcal{G}$ and that $\mathcal{E}$ satisfies the strong right obscure axiom - that is, if $g\circ f\in\mathcal{E}$ then $g\in\mathcal{E}$. Suppose further that $\mathcal{C}$ admits arbitrary coproducts of objects of $\mathcal{G}$. For any object $S$ of $\mathcal{C}$, the evaluation map 
    $$e:\coprod_{G\in\mathcal{G}}\coprod_{\mathrm{Hom}(G,S)}G\rightarrow S$$
    is in $\mathcal{E}$.
\end{proposition}

\begin{proof}
Note that when $\mathcal{E}$ is the class of all epimorphisms, this is standard; see, e.g., \cite{MR0340371}, and the proof works exactly the same. 
\end{proof}

When we have the strong right obscure axiom, we get a sensible notion of admissible generator. 

\begin{definition}
Let $(\mathcal{C},\mathcal{Q})$ be a deflation primitive-exact category. A collection $\mathcal{G}$ of objects of $\mathcal{C}$ is said to be an \textit{admissible generating class}
    if it is an $\mathcal{E_{Q}}$-generating class.
\end{definition}

\begin{corollary}\label{cor:adepicheckgen}
    Suppose that $(\mathcal{C},\mathcal{Q})$ is a deflation primitive-exact category, has an admissible generating class $\mathcal{G}$, and satisfies the strong right obscure axiom. If $\mathrm{Hom}(G,X)\rightarrow\mathrm{Hom}(G,Y)$ is an epimorphism of sets for all $G\in\mathcal{G}$, then $X\rightarrow Y$ is an admissible epimorphism.
 \end{corollary}

 \begin{example}
    Again, the example of $\mathrm{Set}_{*}$ is instructive in seeing that the strong right obscure axiom is really necessary in Lemma \ref{lem:adepicheckgen}. The pointed set $G$ with one non-zero element is an admissible, and hence epimorphic, generator. For any pointed set $X$, we have that $\mathrm{Hom}(G,X)=X$ as a pointed set. In particular, the fact that $\mathrm{Hom}(G,X)\rightarrow\mathrm{Hom}(G,Y)$ is an epimorphism \textit{does not imply} that $X\rightarrow Y$ is a strict epimorphism.
\end{example}

Although the notion of an admissible generator may seem more natural for a proto-exact category, in many cases having an epimorphic generator will be enough. In one of our primary applications, we will want to detect when a morphism $f:X\rightarrow Y$ in a proto-exact category $(\mathcal{C},\mathcal{Q})$ is an isomorphism. For separate reasons we will already know that it is an admissible monomorphism, and we will know that $\mathrm{Hom}(G,X)\rightarrow\mathrm{Hom}(G,Y)$ is an epimorphism for some epimorphic generator $G$. In particular, $f$ is also an epimorphism. Then we can conclude that $f$ is an isomorphism by Proposition \ref{prop:admonepiisio}.

Having an admissible generator will prove useful for establishing local presentability, however. Recall that a subcategory $\mathcal{G}\subset\mathcal{C}$ is a strong generator (\cite{MR3795415}*{Section 3}) precisely if the functor
$$F:\mathcal{C}\rightarrow\mathrm{Set}^{\mathcal{G}^{op}}, C\mapsto\mathrm{Hom}(-,C)$$
is faithful and conservative.

\begin{proposition}\label{prop:Eepi}
Let $\mathcal{C}$ be a category, and $\mathcal{E}$ a class of regular epimorphisms in $\mathcal{C}$ that satisfies the strong right obscure axiom - that is, if $g\circ f\in\mathcal{E}$ then $g\in\mathcal{E}$. If $\mathcal{C}$ has an $\mathcal{E}$-generating set $\mathcal{G}$ such that arbitrary coproducts of objects of $\mathcal{G}$ exist in $\mathcal{C}$ then, regarded as a full subcategory, $\mathcal{G}\subset\mathcal{C}$ is a strong generator.
\end{proposition}

\begin{proof}
The assumptions and Proposition \ref{prop:canonicaladepi} imply that, for any object $S$ of $\mathcal{C}$, the canonical morphism
 $$e:\coprod_{G\in\mathcal{G}}\coprod_{\mathrm{Hom}(G,S)}G\rightarrow S$$
 is a regular epimorphism. The result now follows from \cite{borger1991strong}*{Theorem 3.5}.

\end{proof}

\begin{proposition}
    Let $(\mathcal{C},\mathcal{Q})$ be a one-sided proto-exact category satisfying the strong right obscure axiom, and admitting all finite limits, and $\mathcal{G}$ an admissible generating class. Let 
    $$0\rightarrow X\rightarrow Y\rightarrow Z\rightarrow 0$$
    be a sequence such that, for any $G\in\mathcal{G}$, 
    $$0\rightarrow \mathrm{Hom}(G,X)\rightarrow \mathrm{Hom}(G,Y)\rightarrow \mathrm{Hom}(G,Z)\rightarrow 0$$
    is left semi-exact. Then the sequence    $$0\rightarrow X\rightarrow Y\rightarrow Z\rightarrow 0$$
    is exact.
\end{proposition}
\begin{proof}
    By the strong right obscure axiom, and Proposition \ref{prop:Eepi}, the morphism $Y\rightarrow Z$ is an admissible epimorphism. Now, there is a natural map $X\rightarrow\mathrm{Ker}(Y\rightarrow Z)$. By Yoneda, this is an isomorphism.
\end{proof}

By \cite{adamek1994locally}*{Theorem 1.20}, we immediately get the following.

\begin{theorem}\label{thm:locpresentgen}
    Let $(\mathcal{C},\mathcal{Q})$ be a primitive-exact category that satisfies the strong right obscure axiom and is cocomplete. Suppose it has a set of admissible generators $\mathcal{G}$, and that there exists a cardinal $\lambda$ such that for each $G\in\mathcal{G}$, $G$ is $\lambda$-presentable. Then, $\mathcal{C}$ is locally $\lambda$-presentable as a category. 
\end{theorem}

\begin{remark}\label{rem:lambdapresobjects}
   By the proof of \cite{adamek1994locally}*{Theorem 1.20} and \cite{adamek1994locally}*{Remark 1.30}, the $\lambda$-presentable objects of $\mathcal{C}$ are $\lambda$-split quotients of $\lambda$-small colimits of objects of $\mathcal{G}$.
\end{remark}

\subsubsection{Primitive-exact structures on monad algebras}

Here we explain when one can lift primitive-exact structures to categories of algebras over monads.

\begin{proposition}\label{prop:monadstr}
Let $\mathcal{C}$ be a category, and $T:\mathcal{C}\rightarrow\mathcal{C}$ a monad.
\begin{enumerate}
    \item 
   Suppose that $(\mathcal{C},\mathcal{Q})$ is a deflation primitive-exact category, and that $T$ commutes with coequalisers of kernel pairs. Then there is a canonical deflation primitive-exact category $(\mathrm{Alg}_{T}(\mathcal{C}),\mathcal{Q}_{T})$ on $\mathrm{Alg}_{T}(\mathcal{C})$, whereby a morphism $f:Y\rightarrow Z$ in $\mathrm{Alg}_{T}(\mathcal{C})$ is an admissible epimorphism precisely if $|f|:|Y|\rightarrow|Z|$ is an admissible epimorphism in $(\mathcal{C},\mathcal{Q})$, where $|-|:\mathrm{Alg}_{T}(\mathcal{C})\rightarrow\mathcal{C}$ denotes the forgetful functor. Moreover, $(\mathrm{Alg}_{T}(\mathcal{C}),\mathcal{Q}_{T})$ is a category with QE-epi class whenever $(\mathcal{C},\mathcal{Q})$ is. Finally, $(\mathrm{Alg}_{T}(\mathcal{C}),\mathcal{Q}_{T})$ satisfies the strong right obscure axiom whenever $(\mathcal{C},\mathcal{Q})$ does.
    \item 
       Suppose that $(\mathcal{C},\mathcal{Q})$ is an inflation primitive-exact category, and that $T$ commutes with pushouts. Then there is canonical structure inflation primitive-exact category $(\mathrm{Alg}_{T}(\mathcal{C}),\mathcal{Q}_{T})$ on $\mathrm{Alg}_{T}(\mathcal{C})$, whereby a morphism $f:X\rightarrow Y$ in $\mathrm{Alg}_{T}(\mathcal{C})$ is an admissible monomorphism precisely if $|f|:|X|\rightarrow|Y|$ is an admissible monomorphism in $(\mathcal{C},\mathcal{Q})$, where $|-|:\mathrm{Alg}_{T}(\mathcal{C})\rightarrow\mathcal{C}$ denotes the forgetful functor. Moreover, $(\mathrm{Alg}_{T}(\mathcal{C}),\mathcal{Q}_{T})$ is a category with QE-mono class whenever $(\mathcal{C},\mathcal{Q})$ is. Finally, $(\mathrm{Alg}_{T}(\mathcal{C}),\mathcal{Q}_{T})$ satisfies the strong left obscure axiom whenever $(\mathcal{C},\mathcal{Q})$ does.
       \item 
          Let $(\mathcal{C},\mathcal{Q})$ be a proto-exact category, and $T:\mathcal{C}\rightarrow\mathcal{C}$ a monad that preserves pushouts. Then there is a proto-exact structure $\mathcal{Q}_{T}$ on $\mathrm{Alg}_{T}(\mathcal{C})$ whereby a sequence
   $$0\rightarrow X\rightarrow Y\rightarrow Z\rightarrow 0$$
   is short exact precisely if 
      $$0\rightarrow |X|\rightarrow |Y|\rightarrow |Z|\rightarrow 0$$
      is short exact in $(\mathcal{C},\mathcal{Q})$, where $|-|:\mathrm{Alg}_{T}(\mathcal{C})\rightarrow\mathcal{C}$ denotes the forgetful functor. This proto-exact structure is right (resp. left) total precisely if the proto-exact category $(\mathcal{C},\mathcal{Q})$ is right (resp. left) total. Moreover, $(\mathrm{Alg}_{T}(\mathcal{C}),\mathcal{Q}_{T})$ satisfies the strong right/left obscure axiom whenever $(\mathcal{C},\mathcal{Q})$ does. Finally, $\mathrm{Alg}_{T}(\mathcal{C})$ is parabelian precisely if $\mathcal{C}$ is.
\end{enumerate}
\end{proposition}

\begin{proof}
We prove the third claim, the others are similar.
This works as in \cite{kelly2016homotopy}*{Proposition 2.73}.
    The claim follows from the general fact that if $T$ is a monad on any category $\mathcal{C}$ that preserves pushouts, then the forgetful functor $\mathrm{Alg}_{T}(\mathcal{C})\rightarrow\mathcal{C}$ creates limits and pushouts, and reflects isomorphisms. For a proof of this, see \cite{handbook} Proposition 4.3.1 and Proposition 4.3.2.
\end{proof}

\begin{example}\label{example:monoid}
    Let $(\mathcal{C},\mathcal{Q})$ be a proto-exact category, such that $\mathcal{C}$ is equipped with a closed symmetric monoidal structure $(R,\otimes,\underline{\mathrm{Hom}})$. Let $A\in\mathrm{Ass}(\mathcal{C})$ be a unital associative monoid. Consider the category ${}_{A}\mathrm{Mod}(\mathcal{C})$ of left $A$-modules over $A$ internal to $\mathcal{C}$. Then $A\otimes(-):\mathcal{C}\rightarrow\mathcal{C}$ may be equipped with the structure of a monad. Since the monoidal structure is closed, $(-)\otimes(-)$ commutes with colimits in each variable. Therefore, there exists a proto-exact structure on ${}_{A}\mathrm{Mod}(\mathcal{C})$ consisting of those kernel-cokernel pairs that forget to sequences in $\mathcal{Q}$. Similarly, one can define induced proto-exact structures on categories of bimodules.
\end{example}

\begin{definition}
     Let $(\mathcal{C},\mathcal{Q})$ be a proto-exact category, such that $\mathcal{C}$ is equipped with a closed symmetric monoidal structure $(R,\otimes,\underline{\mathrm{Hom}})$. An object $C$ of $\mathcal{C}$ is said to be \textit{flat} (relative to $\mathcal{Q}$) if the functor $C\otimes(-):(\mathcal{C},\mathcal{Q})\rightarrow(\mathcal{C},\mathcal{Q})$ is exact.
\end{definition}

\subsubsection{Localisations of parabelian categories}
Here we show that we can often construct parabelian categories as localisations of larger parabelian categories. This leads to the example of sheaves.
$$\mathcal{X}_{*}\defeq{}_{\bullet\backslash}\mathcal{X}.$$
\begin{proposition}\label{prop:parabelianloc}
    Let $\mathcal{C}$ be a parabelian category, and let $\mathcal{E}\subset\mathcal{C}$ be a full subcategory such that the inclusion $i:\mathcal{E}\rightarrow\mathcal{C}$ admits a finite limit-preserving left adjoint. Then $\mathcal{E}$ is parabelian. Moreover, a morphism $f:X\rightarrow Y$ in $\mathcal{E}$ is a strict monomorphism if and only if $i(f)$ is a strict monomorphism in $\mathcal{C}$.
\end{proposition}
\begin{proof}
We prove the second claim first. Let $f:X\rightarrow Y$ be a morphism in $\mathcal{E}$ such that $i(f)$ is a strict monomorphism in $\mathcal{C}$. Then $i(f)$ is isomorphic to the kernel $\mathrm{Ker}(i(Y)\rightarrow\mathrm{Coker}(i(f)))$. Hence, $f=\lambda(\mathrm{Ker}(i(Y)\rightarrow\mathrm{Coker}(i(f))))\cong\mathrm{Ker}(Y\rightarrow\mathrm{Coker}(f))$. Conversely, suppose that $f$ is a strict monomorphism in $\mathcal{E}$. Then $f\cong\mathrm{Ker}(g)\rightarrow Y$ for some morphism $g:Y\rightarrow Z$. Hence $i(f)\cong\mathrm{Ker}(i(g))\rightarrow i(Z)$, and is strict. Thus $f$ is a strict monomorphism in $\mathcal{E}$ if and only if $i(f)$ is a strict monomorphism in $\mathcal{C}$.
Now, let $f:X\rightarrow Y$ be a strict monomorphism and $g:X\rightarrow Z$ be a strict epimorphism in $\mathcal{C}$. Consider $\tilde{Z}=\mathrm{Coker}(\mathrm{Ker}(i(g))\rightarrow i(X))$. Write $\tilde{g}:i(X)\rightarrow\tilde{Z}$ for the canonical strict epimorphism. Note that $\lambda(\tilde{g})=g$. Consider the pushout square in $\mathcal{C}$
\begin{displaymath}
    \xymatrix{
    i(X)\ar[d]^{i(f)}\ar[r]^{\tilde{g}} & \tilde{Z}\ar[d]^{f'}\\
    i(Y)\ar[r]^{\tilde{g}'} & P.
    }
\end{displaymath}
Then $f'$ is an admissible monomorphism, and $f'$ is a strict epimorphism. Localising gives a pushout
\begin{displaymath}
    \xymatrix{
    X\ar[d]^{f}\ar[r]^{g} & Z\ar[d]^{\lambda(f')}\\
    Y\ar[r]^{\lambda(g')} & \lambda(P).
    }
\end{displaymath}
Since $\lambda$ is left exact, $\lambda(f')$ is a strict monomorphism. Since $\lambda$ is right exact (as a left adjoint), $\lambda(g')$ is a strict epimorphism.
Now, let $f:A\rightarrow C$ be a strict monomorphism, and $g:B\rightarrow C$ be a strict epimorphism. The morphism $i(g)$ factors as $i(g):i(B)\rightarrow\tilde{C}\rightarrow C$, where $\tilde{g}:i(B)\rightarrow\tilde{C}$ is a strict epimorphism, and $\tilde{C}\rightarrow C$ localises to an isomorphism in $\mathcal{E}$. Consider the following diagram in which every square is a pullback:
\begin{displaymath}
    \xymatrix{
    P\ar[d]^{\tilde{g}'}\ar[r]^{f'} & i(B)\ar[d]^{\tilde{g}}\\
    P'\ar[d]\ar[r]^{f''} & \tilde{C}\ar[d]\\
    i(A)\ar[r]^{i(f)}& i(C).
    }
\end{displaymath}
By \cite{mozgovoy}*{Lemma 1.3}, $f''$ and $f'$ are strict monomorphisms. Since $\mathcal{C}$ is parabelian, $\tilde{g}'$ is a strict epimorphism. Localising gives that $\lambda(P)\rightarrow \lambda(P')$ is a strict epimorphism. Now, $\lambda(P')\rightarrow A$ is an isomorphism, so $\lambda(P)\rightarrow A$ is a strict epi. But $\lambda(P)\rightarrow A$ is the pullback $g'$ of $g$ along $f$. Moreover, $\lambda(f')$ is also a strict monomorphism.
\end{proof}
\comment{
\begin{remark}\label{rem:generalised}
incorporate into prop
    All that is needed in the proof is the following:
    \begin{enumerate}
        \item if $f'$ is a pushout in $\mathcal{C}$ of a morphism of the form $i(f)$, with $f$ being strict in $\mathcal{E}$, then $\lambda(f')$ is a strict monomorphism, 
        \item 
        if $g:i(Y)\rightarrow Z$ is a strict epimorphism in $\mathcal{C}$, then the natural map $\lambda(\mathrm{Ker}(g))\rightarrow\mathrm{Ker}(\lambda(g))$ is an isomorphism, and 
        \item 
        the functor $\lambda$ preserves pullback squares in $\mathcal{C}$ in which one leg is a morphism of the form $i(f)$, with $f$ being strict in $\mathcal{E}$,
    \end{enumerate}
\end{remark}
}
\begin{corollary}
    Let $\mathcal{X}$ be a Grothendieck topos. The category $\mathcal{X}_{*}$ of pointed objects in $\mathcal{X}$ is parabelian.
\end{corollary}
\begin{proof}
    The category $\mathcal{X}$ is a left exact localisation of a presheaf category. It therefore suffices to prove that a category of the form $\mathrm{Fun}(\mathcal{T},\mathrm{Set}_{*})$ is a parabelian category. Since limits and colimits are computed object-wise, this is clear.
\end{proof}
\begin{corollary}\label{cor:cor:sheavesparabelian}
    Let $\mathcal{E}$ be a parabelian category that is complete and cocomplete, and in which filtered colimits commute with finite limits. Let $\mathcal{T}$ be a small site. Then $\mathrm{Shv}(\mathcal{T},\mathcal{E})$ is a parabelian category.
\end{corollary}

\subsubsection{Approximate pushouts, tameness, and effective unions.}
Here we introduce some central concepts needed to prove cofibrant generation of various classes of admissible monomorphisms.
\begin{definition}
    Let $(\mathcal{C},\mathcal{Q})$ be a primitive-exact category. By an \textit{approximate pushout} we mean a diagram of the form
    \begin{displaymath}
    \xymatrix{
    K\ar[d]^{i}\ar[r] & X\ar[d]^{j}\\
    Y\ar[r] & Z
    }
\end{displaymath}
in which all morphisms are admissible monomorphisms, and the canonical morphism $\mathrm{Coker}(i)\rightarrow\mathrm{Coker}(j)$ is an isomorphism. 
\end{definition}
Note that, if 
    \begin{displaymath}
    \xymatrix{
    K\ar[d]^{i}\ar[r] & X\ar[d]^{j}\\
    Y\ar[r] & Z
    }
\end{displaymath}
is an approximate pushout, and the pushout of the top and left-hand morphisms $P$ exists, then there is a canonical morphism $c:P\rightarrow Z$. Suppose that the pushout $i'$ of $i$ is an admissible monomorphism. Then we get a commutative diagram of short exact sequence
\begin{displaymath}
    \xymatrix{
    \emptyset\ar[r] & X\ar[r]^{i'}\ar@{=}[d] & P\ar[d]^{c}\ar[r] & \mathrm{Coker}(i)\ar[d]^{\psi}\ar[r] & \bullet\\
    \emptyset\ar[r] & X\ar[r]^{j} & Z\ar[r] &\mathrm{Coker}(j)\ar[r] & \bullet
    }
\end{displaymath}
where the last vertical morphism $\psi$ is an isomorphism.
\begin{definition}
    A primitive-exact category $(\mathcal{C},\mathcal{Q})$ is said to be \textit{tame} if whenever
\begin{displaymath}
    \xymatrix{
    K\ar[d]^{i}\ar[r] & X\ar[d]^{j}\\
    Y\ar[r] & Z
    }
\end{displaymath}
is an approximate pushout diagram, then the pushout of the top and left-hand morphisms $P$ exists, and the canonical morphism $P\rightarrow Z$ is an isomorphism.
\end{definition}
\begin{proposition}
    If $(\mathpzc{E},\mathcal{Q})$ is a right exact category then it is tame as a proto-exact category.
\end{proposition}
\begin{proof}
    This follows from the Five Lemma \cite{bazzoni2013one}*{Lemma 5.3 }. 
\end{proof}
\begin{proposition}
    Suppose that $(\mathcal{C},\mathcal{Q})$ is totally proto-exact. If $(\mathcal{C},\mathcal{Q})$ is tame, then $\mathcal{C}$ is semi-additive. 
\end{proposition}
\begin{proof}
    In a totally proto-exact category, the morphisms $A\rightarrow A\vee B$, $A\rightarrow A\times B$, $B\rightarrow A\vee B$, and $B\rightarrow A\times B$ are admissible monomorphisms, with cokernels $B$, $B$, $A$, and $A$, respectively. The morphism $0\rightarrow B$ is also admissible with cokernel $B$. Consider the commutative square
    \begin{displaymath}
        \xymatrix{
        0\ar[d]\ar[r] & A\ar[d]\\
        B\ar[r] & A\times B.
        }
    \end{displaymath}
    The pushout of the top and left legs is precisely $A\vee B$. By assumption then, the canonical map $A\vee B\rightarrow A\times B$ is an isomorphism.
\end{proof}
\begin{example}
   It seems likely to us that a tame, totally proto-exact category, that further satisfies the strong obscure axiom, is necessarily additive. However, we have not been able to come up with a proof or counterexample.
\end{example}
We have the following related condition, that holds in many non-semiadditive categories.
\begin{definition}
    A category with regular monomorphism $(\mathcal{C},\mathcal{Q})$ is said to \textit{have} $\mathcal{Q}$\textit{-effective unions} if whenever 
    $X\rightarrow Z$ and $Y\rightarrow Z$ are admissible monomorphisms then
    \begin{enumerate}
        \item the fibre product $X\cap_{Z} Y\defeq X\times_{Z}Y$ exists, and the canonical morphisms $X\cap_{Z} Y\rightarrow X$ and $X\cap_{Z} Y\rightarrow Y$ are admissible monomorphisms, and
        \item 
        the pushout $X\cup_{Z} Y$ of the morphisms $X\cap_{Z} Y\rightarrow X$ and $X\cap_{Z} Y\rightarrow Y$ exists, the morphisms $X\rightarrow X\cup_{Z}Y$ and  $Y\rightarrow X\cup_{Z}Y$ are admissible monomorphisms, and the natural map $X\cup_{Z} Y\rightarrow  Z$ is an admissible monomorphism.
    \end{enumerate}
\end{definition}
\begin{example}
\begin{enumerate}
\item 
Any abelian category has $\mathcal{Q}_{max}$-effective unions.
    \item 
    The condition holds in any topos, or the category of pointed objects in a topos.
\end{enumerate}
For both of these assertions, see \cite{barr2006categories}.
\end{example}
\begin{proposition}
    Let $(\mathcal{C},\mathcal{Q})$ be an inflation primitive-exact category, and $T:\mathcal{C}\rightarrow\mathcal{C}$ a monad that commutes with finite colimits. If $(\mathcal{C},\mathcal{Q})$ has $\mathcal{Q}$-effective unions, then so does $(\mathrm{Alg}_{T}(\mathcal{C}),\mathcal{Q}_{T})$.
\end{proposition}
\begin{proof}
    Let $|-|:\mathrm{Alg}_{T}(\mathcal{C})\rightarrow\mathcal{C}$ denote the forgetful functor. Let $X\rightarrow Z$ and $Y\rightarrow Z$ be admissible monomorphisms in $(\mathrm{Alg}_{T}(\mathcal{C}),\mathcal{Q}_{T})$. Then $|X\cap_{Z} Y|=|X|\cap_{|Z|} |Y|$, and $|X\cup_{Z} Y|=|X|\cup_{|Z|} |Y|$. Moreover, $|-|$ preserves  and reflects strict monomorphisms. From this the result immediately follows.
\end{proof}
\begin{proposition}
    Let $\mathcal{C}$ be a parabelian category, and let $\mathcal{E}$ be a subcategory such that the inclusion $\mathcal{E}\rightarrow\mathcal{C}$ admits a finite limit-preserving left adjoint $\lambda$. If $\mathcal{C}$ has $\mathcal{Q}_{max}$-effective unions then $\mathcal{E}$ has $\mathcal{Q}_{max}$-effective unions.
\end{proposition}
\begin{proof}
    Let $f:A\rightarrow C$ and $g:B\rightarrow C$ be strict monomorphisms in $\mathcal{E}$. Then $i(A\cap_{C} B)\cong i(A)\cap_{i(C)} i(B)$. The morphisms $i(A)\cap_{i(C)} i(B)\rightarrow i(A)$ and $i(A)\cap_{i(C)}i(B)\rightarrow i(B)$ are strict monomorphisms by assumption. Applying $\lambda$ gives that $A\cap_{C} B\rightarrow A$ and $A\cap_{C} B\rightarrow B$ are strict monomorphisms. Now, we consider the pushout square
    \begin{displaymath}
        \xymatrix{
        i(A\cap_{C} B)\ar[d]\ar[r] & i(B)\ar[d]\\
        i(A)\ar[r] & i(A)\cup_{i(C)} i(B),
        }
    \end{displaymath}
    in which all morphisms are strict monomorphisms. By assumption, $i(A)\cup_{i(C)} i(B)\rightarrow i(C)$ is a strict monomorphism. Appliyng $\lambda$ gives the result.
\end{proof}
\begin{corollary}
    Let $\mathcal{E}$ be a parabelian category with $\mathcal{Q}_{max}$-effective unions. Suppose further that $\mathcal{E}$ is complete and cocomplete, and that filtered colimits commute with finite limits in $\mathcal{E}$. Let $\mathcal{T}$ be a small site. Then $\mathrm{Shv}(\mathcal{T},\mathcal{E})$ has $\mathcal{Q}_{max}$-effective unions.
\end{corollary}
\begin{proposition}
    Suppose that a left totally proto-exact category $(\mathcal{C},\mathcal{Q})$ has $\mathcal{Q}$-effective unions, and whenever 
    \begin{displaymath}
        \xymatrix{
        0\ar[r] & X\ar@{=}[d]\ar[r] & Y\ar[d]^{f}\ar[r] & Z\ar@{=}[d]\ar[r] & 0\\
        0\ar[r] & X\ar[r] & P\ar[r] & Z\ar[r] & 0
        }
    \end{displaymath}
    is a commutative diagram of exact sequences, the map $Y\rightarrow P$ is an epimorphism. Then $(\mathcal{C},\mathcal{Q})$ is tame.
\end{proposition}
\begin{proof}
Consider a weak pushout diagram
 \begin{displaymath}
    \xymatrix{
    K\ar[d]^{i}\ar[r] & X\ar[d]^{j}\\
    Y\ar[r] & Z
    }
\end{displaymath}
 Take the intersection $X\cap_{Z} Y$. We consider the diagram
\begin{displaymath}
    \xymatrix{
    K\ar[d]\ar[r] & X\cap_{Z} Y\ar[d]\ar[r] & X\ar[d]\\
    Y\ar[r] & P'\ar[r] & P
    }
\end{displaymath}
in which both squares, and hence the entire rectangle, are defned to be pushouts. Then, by assumption $P\rightarrow Z$ is an admissible, and hence regular, monomorphism. But $\mathrm{Coker}(X\rightarrow P)\cong\mathrm{Coker}(K\rightarrow Y)\cong\mathrm{Coker}(X\rightarrow Z)$. By assumption, $P\rightarrow Z$ is also an epimorphism. Hence it is an isomorphism. 
\end{proof}

\subsection{Cotorsion and enveloping theory}

In this subsection we define cotorsion pairs in proto-exact categories and, more importantly, we introduce cotorsion and enveloping theory.

\subsubsection{Projectives and injectives.}

We begin by considering the simplest case of a cotorsion pair, a notion that we will define fully subsequently. 

\begin{definition}
    An object $P$ of a one-sided proto-exact category $(\mathcal{C},\mathcal{Q})$ is said to be \textit{projective} if any commutative diagram 
    \begin{displaymath}
        \xymatrix{
        0\ar[d]\ar[r] & X\ar[d]^{f}\\
        P\ar[r] & Y
        }
    \end{displaymath}
    with $f\in\mathcal{E}_{\mathcal{Q}}$ admits a lift, that is there exists a commutative diagram
       \begin{displaymath}
        \xymatrix{
        0\ar[d]\ar[r] & X\ar[d]^{f}\\
        P\ar[r]\ar[ur] & Y
        }
    \end{displaymath}
\end{definition}

We will denote by $\mathrm{Proj}(\mathcal{C},\mathcal{Q})$ the class of projective objects. The following is essentially the definition of a projective.

\begin{corollary}    An object $P$ in a one-sided proto-exact category $(\mathcal{C},\mathcal{Q})$ is projective if and only if the functor
    $$\mathrm{Hom}(P,-):\mathcal{C}\rightarrow\mathrm{Set}_{*}$$
    is left semi-exact. 
\end{corollary}

\begin{proof}
    This is immediate from the fact that all epimorphisms in $\mathrm{Set}$, and hence in $\mathrm{Set}_{*}$, are effective.
\end{proof}

\begin{corollary}
Suppose that a one-sided proto-exact category $(\mathcal{C},\mathcal{Q})$ satisfying the strong right obscure axiom has an admissible generator $P$ that is projective. Then a morphism $X\rightarrow Y$ is an admissible epimorphism if and only if the map
$$\mathrm{Hom}(P,X)\rightarrow\mathrm{Hom}(P,Y)$$
is an epimorphism of pointed sets.
\end{corollary}

Of course, we have the dual concept of injectives.

\begin{definition}
    An object $I$ of a one-sided proto-exact category  $(\mathcal{C},\mathcal{Q})$ is said to be \textit{injective} if any commutative diagram 
    \begin{displaymath}
        \xymatrix{
        X\ar[d]^{g}\ar[r] & I\ar[d]\\
        Y\ar[r] & 0
        }
    \end{displaymath}
    where $g\in\mathcal{M}_{\mathcal{Q}}$, admits a lift. 
\end{definition}

\begin{definition}
    \begin{enumerate}
        \item 
        A one-sided proto-exact category $(\mathcal{C},\mathcal{Q})$ is said to \textit{have enough projectives} if for every object $C$ of $\mathcal{C}$ there is a projective $P$ and a morphism $P\rightarrow C$ in $\mathcal{E}_{\mathcal{Q}}$.
        \item 
         A one-sided proto-exact category  $(\mathcal{C},\mathcal{Q})$ is said to \textit{have enough injectives} if for every object $C$ of $\mathcal{C}$ there is an injective $I$ and a morphism $C\rightarrow I$ in $\mathcal{M}_{\mathcal{Q}}$. 
    \end{enumerate}
\end{definition}

\subsubsection{General cotorsion theory}

Let $(\mathcal{C},\mathcal{Q})$ be a one-sided proto-exact category and $\mathcal{F}$ a class of objects in $\mathcal{C}$. Denote by $\mathbf{AdMon}_{\mathcal{F}}$ the class of admissible monomorphisms whose cokernel is in $\mathcal{F}$, and by $\mathbf{AdEpi}_{\mathcal{F}}$ the class of admissible epimorphisms whose kernel is in $\mathcal{F}$. We denote by $\mathcal{F}^{\perp}$ the class of objects $K$ such that any diagram of the form

    \begin{displaymath}
        \xymatrix{
        X\ar[d]^{g}\ar[r] & K\ar[d]\\
        Y\ar[r] & 0
        }
    \end{displaymath}
    with $g\in\mathbf{AdMon}_{\mathcal{F}}$ admits a lift, and by ${}^{\perp}\mathcal{F}$ the class of objects $K$ such that any diagram

        \begin{displaymath}
        \xymatrix{
       0\ar[d]\ar[r] & X\ar[d]^{f}\\
        K\ar[r] & Y
        }
    \end{displaymath}
    with  $f\in\mathbf{AdEpi}_{\mathcal{F}}$ admits a lift. For a class $\mathcal{S}$ of morphisms, we denote by $\mathcal{S}^{\llp}$ the class of morphisms with the right lifting property against $\mathcal{S}$.

    \begin{definition}
        Let $(\mathcal{C},\mathcal{Q})$ be a one-sided proto-exact category. A \textit{cotorsion pair} on $\mathcal{C}$ is a pair of classes of objects $(\mathfrak{L},\mathfrak{R})$ such that
        \begin{enumerate}
            \item $\mathfrak{L}={}^{\perp}\mathfrak{R}$,
            \item and $\mathfrak{R}=\mathfrak{L}^{\perp}$.
        \end{enumerate}
    \end{definition}

    The following is immediate from the definitions.

\begin{proposition}
    The pairs $(\mathrm{Proj}(\mathcal{C},\mathcal{Q}),\mathrm{Ob}(\mathcal{C}))$ and $(\mathrm{Ob}(\mathcal{C}),\mathrm{Inj}(\mathcal{C},\mathcal{Q}))$ are cotorsion pairs. 
\end{proposition}

    In a non-additive proto-exact category such pairs seem to be quite rare, apart from the projective and injective ones. However, we do have the following. 
    
    \begin{lemma}\label{lem:rlpF}
        Let $\mathcal{F}$ be a class of objects in a proto-exact category satisfying the strong right obscure axiom. If $\mathcal{F}$ contains an admissible generating class, then $\mathbf{AdMon}_{\mathcal{F}}^{\llp}\subset\mathbf{AdEpi}_{\mathcal{F}^{\perp}}$. 
    \end{lemma}

    \begin{proof}
        Suppose $g:X\rightarrow Y$ has the right lifting property against all maps in $\mathbf{AdMon}_{\mathcal{F}}$. Since $\mathcal{F}$ contains an admissible generating class, $g$ is an admissible epimorphism by Lemma \ref{lem:adepicheckgen}. Indeed, the lifting property against the maps $0\rightarrow G$ for $G\in\mathcal{G}$ precisely means that $\mathrm{Hom}(G,f)$ is an epimorphism. Now, let $K=\mathrm{Ker}(g)$. 
            The class $\mathbf{AdMon}_{\mathcal{F}}^{\llp}$ is closed under pullbacks, so $K\rightarrow 0$ is in $\mathbf{AdMon}_{\mathcal{F}}^{\llp}$, as required.
    \end{proof}

Again, as effective epimorphisms in the category $\mathrm{Set}_{*}$ are just surjections, the following is clear.

    \begin{proposition}
        Let $\mathcal{F}$ be a class of objects in $(\mathcal{C},\mathcal{Q})$. Let $T$ be an object of $\mathcal{C}$.  The following are equivalent.
        \begin{enumerate}
            \item 
            The object $T$ is in $\mathcal{F}^{\perp}$.
            \item 
            For any exact sequence
            $$0\rightarrow X\rightarrow Y\rightarrow F\rightarrow 0$$
            with $F\in\mathcal{F}$, the sequence
            $$0\rightarrow\mathrm{Hom}(F,T)\rightarrow\mathrm{Hom}(Y,T)\rightarrow\mathrm{Hom}(X,T)\rightarrow 0$$
            is left semi-exact.
        \end{enumerate}
    \end{proposition}

    \begin{corollary}
        Let 
        $$0\rightarrow X\rightarrow Y\rightarrow Z\rightarrow 0$$
        be an exact sequence with $Z\in\mathcal{F}$ and $X\in\mathcal{F}^{\perp}$. Then the map $X\rightarrow Y$ is split.
    \end{corollary}

    \begin{corollary}\label{cor:prodright}
       Let $\mathcal{F}$ be a class of objects in $(\mathcal{C},\mathcal{Q})$. The class $\mathcal{F}^{\perp}$ is closed under products.
    \end{corollary}

\begin{definition}
        Let $\mathcal{A}$ be a class of objects in a pointed category $\mathcal{C}$. 
        \begin{enumerate}
            \item 
            An
            $\mathcal{A}$-\textit{precover} of an object $E$ is a map $A\rightarrow E$ with $A\in\mathcal{A}$ such that $\mathrm{Hom}(A',A)\rightarrow\mathrm{Hom}(A',E)$ is an epimorphism for any $A'\in\mathcal{A}$.
            \item 
            An $\mathcal{A}$-\textit{preenvelope} of an object $E$ is a map $E\rightarrow A$ such that $A\in\mathcal{A}$ and $\mathrm{Hom}(A,A')\rightarrow\mathrm{Hom}(E,A')$ is an epimorphism for all $A'\in\mathcal{A}$.
        \end{enumerate} Now, suppose that $(\mathcal{C},\mathcal{Q})$ is a proto-exact category.
        
        \begin{enumerate}
            \item 
            An $\mathcal{A}$-special precover is an admissible epimorphism $\pi:A\rightarrow E$ with $A\in\mathcal{A}$ and $\mathrm{Ker}(\pi)\in\mathcal{A}^{\perp}$.
            \item
            An $\mathcal{A}$-special preenvelope is an admissible monomorphism $i:E\rightarrow A$ such that $A\in\mathcal{A}$ and $\mathrm{Coker}(i)\in{}^{\perp}\mathcal{A}$.
        \end{enumerate}
\end{definition}

We will address the existence of covers and envelopes in the next section.

We conclude with the following useful observation.

\begin{proposition}\label{prop:rightadjointinjective}
       Let $(\mathcal{C},\mathcal{Q})$ and $(\mathcal{D},\mathcal{R})$ be proto-exact categories, and let
    $$\adj{L}{\mathcal{C}}{\mathcal{D}}{R}$$
    be an adjunction. If $L:(\mathcal{C},\mathcal{Q})\rightarrow(\mathcal{D},\mathcal{R})$ is exact, then $R(I)$ is injective in $(\mathcal{C},\mathcal{Q})$ for any injective $I$ in $(\mathcal{D},\mathcal{R})$.
\end{proposition}

\begin{proof}
    Let 
    $$0\rightarrow X\rightarrow Y\rightarrow Z\rightarrow 0$$
    be an exact sequence in $(\mathcal{C},\mathcal{Q})$. Then
        $$0\rightarrow L(X)\rightarrow L(Y)\rightarrow L(Z)\rightarrow 0$$
        is an exact sequence in $(\mathcal{D},\mathcal{R})$. Applying $\mathrm{Hom}(-,I)$ gives a left semi-exact sequence
            $$0\rightarrow \mathrm{Hom}(L(X),I)\rightarrow \mathrm{Hom}(L(Y),I)\rightarrow \mathrm{Hom}(L(Z),I)\rightarrow 0$$ 
            of pointed sets. This is isomorphic to the sequence
            $$0\rightarrow \mathrm{Hom}(X,R(I))\rightarrow \mathrm{Hom}(Y,R(I))\rightarrow \mathrm{Hom}(Z,R(I))\rightarrow 0.$$
            Hence, this sequence is left semi-exact, so $R(I)$ is injective. 
\end{proof}

\subsubsection{Injectives and monoidal structures}
Let us here outline the relationship between flat objects and injective objects in closed symmetric monoidal proto-exact categories.
\begin{proposition}[cf. \cite{gillespie2006flat} Lemma 5.3]\label{prop:cohominj}
    Let $(\mathcal{C},\mathcal{Q})$ be a one-sided proto-exact category that is finitely complete and finitely cocomplete, such that $\mathcal{C}$ is equipped with a closed symmetric monoidal structure $(R,\otimes,\underline{\mathrm{Hom}})$. Let $T,S$ be associative monoids in $(\mathcal{C},\mathcal{Q})$. Consider the one-sided proto-exact category of $(T,S)$-bimodules, $({}_{T}\mathrm{Mod}_{S}(\mathcal{C}),{}_{T}\mathcal{Q}_{S})$ of Proposition \ref{prop:monadstr}.
    Let $F$ be a $(T,S)$-bimodule such that $(-)\otimes_{T}F:({}_{T}\mathrm{Mod}(\mathcal{C}),\mathcal{Q}_{T})\rightarrow(\mathrm{Mod}_{S}(\mathcal{C}),{}_{T}\mathcal{Q})$ is exact, and let $I$ be injective as a right $S$-module. Then $\underline{\mathrm{Hom}}_{S}(F,I)$ is injective as a left $T$-module. 
\end{proposition}
\begin{proof}
    Let 
    $$0\rightarrow X\rightarrow Y\rightarrow Z\rightarrow 0$$
    be an exact sequence of left $T$-modules. We need to show that 
    $$0\rightarrow\mathrm{Hom}(Z,\underline{\mathrm{Hom}}_{S}(F,I))\rightarrow \mathrm{Hom}(Y,\underline{\mathrm{Hom}}(F,I))\rightarrow \mathrm{Hom}_{S}(X,\underline{\mathrm{Hom}}_{S}(F,I))\rightarrow 0$$
    is left semi-exact. But this is isomorphic to 
        $$0\rightarrow\mathrm{Hom}_{S}(Z\otimes_{T} F,I)\rightarrow \mathrm{Hom}_{S}(Y\otimes_{T} F,I)\rightarrow \mathrm{Hom}_{S}(X\otimes_{T} F,I)\rightarrow 0.$$
        By assumption, 
        the sequence 
      $$0\rightarrow F\otimes_{T} X\rightarrow F\otimes_{T} Y\rightarrow F\otimes_{T} Z\rightarrow 0$$
      is exact. Then, applying the exact functor $\mathrm{Hom}_{S}(-,I)$ to this gives a left semi-exact sequence.
\end{proof}
\begin{proposition}
       Let $(\mathcal{C},\mathcal{Q})$ be a one-sided proto-exact category that is finitely complete and finitely cocomplete, such that $\mathcal{C}$ is equipped with a closed symmetric monoidal structure $(R,\otimes,\underline{\mathrm{Hom}})$. Suppose that $(\mathcal{C},\mathcal{Q})$ satisfies the strong right obscure axiom, and admits an admissible generating class that consists of flat objects. Then any injective object of $(\mathcal{C},\mathcal{Q})$ is \textit{internally injective}, that is, $\underline{\mathrm{Hom}}(-,I)$ is exact.
\end{proposition}
\begin{proof}
    Let $\mathcal{G}$ be an admissible generating class consisting of flat objects. It suffices to prove that $\mathrm{Hom}(G,\underline{\mathrm{Hom}}(-,I))$ is left semi-exact for all $G\in\mathcal{G}$. However, this functor is isomorphic to 
    $$\mathrm{Hom}(G\otimes(-),I).$$
    Since $G$ is flat, $G\otimes(-)$ is exact. Since $I$ is injective the result follows. 
\end{proof}
\begin{proposition}\label{prop:goingdowninjectives}
    Let $R\rightarrow S$ be a morphism of unital commutative monoids in a closed symmetric monoidal proto-exact category $(\mathcal{C},\mathcal{Q})$ that is finitely complete and finitely cocomplete. Suppose that 
    \begin{enumerate}
        \item
        $S$ is flat as an $R$-module, and
        \item 
        for any $M\in{}_{R}\mathrm{Mod}(\mathcal{C})$, the natural map 
    $M\rightarrow S\otimes_{R}M$ is a strict monomorphism. 
    \end{enumerate}
If $({}_{S}\mathrm{Mod}(\mathcal{C}),\mathcal{Q}_{S})$ has enough injectives then so does $({}_{R}\mathrm{Mod}(\mathcal{C}),\mathcal{Q}_{R})$.
\end{proposition}
\begin{proof}
    The functor $S\otimes_{R}(-):{}_{R}\mathrm{Mod}(\mathcal{C})\rightarrow {}_{S}\mathrm{Mod}(\mathcal{C})$ is exact. In particular, if $I$ is an injective $S$-module, then it is injective when regarded as an $R$-module. Take an $R$-module $M$. There exists an injective $S$-module $I$ and an admissible monomorphism $S\otimes_{R}M\rightarrow I$. Then $M\rightarrow S\otimes_{R}M\rightarrow I$ is an admissible monomorphism of $R$-modules, with $I$ injective as an $R$-module.
\end{proof}
\begin{example}\label{example:goingdown}
    Let $R\rightarrow S$ be a morphism of unital commutative monoids in a closed symmetric monoidal proto-exact category $(\mathcal{C},\mathcal{Q})$. Suppose that every object of $({}_{R}\mathrm{Mod}(\mathcal{C}),\mathcal{Q}_{R})$ is flat - that is, for any $M\in{}_{R}\mathrm{Mod}(\mathcal{C})$, the functor $M\otimes_{R}(-)$ is exact. Suppose further that $R\rightarrow S$ is a strict monomorphism. Then the conditions of Proposition \ref{prop:goingdowninjectives} are satisfied. 
\end{example}

\subsubsection{Forcing injectivity: primitive-exact structures relative to classes of objects}
Later in the article we will have cause to modify certain proto-exact structures by declaring some class of objects to be injective. Note that this may turn a proto-exact category into a one-sided proto-exact category. We therefore work in this slightly greater generality.
\begin{definition}\label{defn:forcinginj}
        Let $(\mathcal{C},\mathcal{Q})$ be an inflation proto-exact category, and let $\mathcal{I}$ be a class of objects in $\mathcal{C}$. Let $\mathcal{M}_{\mathcal{I}}$ denote the subclass of $\mathcal{E}_{\mathcal{Q}}$ consisting of those admissible monomorphisms $i:X\rightarrow Y$ such that $\mathrm{Hom}(i,I):\mathrm{Hom}(Y,I)\rightarrow\mathrm{Hom}(X,I)$ is a surjection of sets for all $I\in\mathcal{I}$. Let $\mathcal{Q}_{\mathcal{I}}\subset\mathcal{Q}$ denote the subclass consisting of pairs $(f,g)$ with $f\in\mathcal{M_{I}}$.
\end{definition}
\begin{proposition}[cf. \cite{gillespie2016derived}*{Proposition 3.3}]
    The pair $(\mathcal{C},\mathcal{Q_{I}})$ is an inflation proto-exact category. If $(\mathcal{C},\mathcal{Q})$ is a category with QE-mono class then so is  $(\mathcal{C},\mathcal{Q_{I}})$. If $(\mathcal{C},\mathcal{Q})$ satisfies the left obscure axiom then so does $(\mathcal{C},\mathcal{Q_{I}})$.
   If $(\mathcal{C},\mathcal{Q})$ is proto-exact, then $(\mathcal{C},\mathcal{Q_{I}})$ is proto-exact if and only if $\mathcal{E}_{\mathcal{Q_{I}}}$ is closed under composition. Finally, when $(\mathcal{C},\mathcal{Q})$ is exact, $\mathcal{E}_{\mathcal{Q_{I}}}$ is always closed under composition.
\end{proposition}
\begin{proof}
    Since the class of epimorphisms of sets is closed under composition, and $\mathcal{M}_{\mathcal{Q}}$ is closed under composition, closure of $\mathcal{M}_{\mathcal{I}}$ under composition is clear. Containment of identity morphisms in $\mathcal{M}_{\mathcal{I}}$ is also clear. Now, let $i:X\rightarrow Y$ be a morphism in $\mathcal{M}_{\mathcal{I}}$, and let $f:X\rightarrow A$ be an admissible epimorphism (resp. an arbitrary morphism in the QE-mono case). Computing the pushout gives a commutative diagram of exact sequences
    \begin{displaymath}
        \xymatrix{
        0\ar[r] & X\ar[d]^{f}\ar[r]^{i} & Y\ar[d]\ar[r] & Z\ar@{=}[d]\ar[r] & 0\\
        0\ar[r] & A\ar[r]^{i'} & P\ar[r] & Z\ar[r] & 0.
        }
    \end{displaymath}
    Applying $\mathrm{Hom}(-,I)$ gives the following commutative diagram
        \begin{displaymath}
        \xymatrix{
        0\ar[r] & \mathrm{Hom}(Z,I)\ar@{=}[d]\ar[r] & \mathrm{Hom}(P,I)\ar[d]\ar[r] & \mathrm{Hom}(A,I)\ar[d]\ar[r] & 0\\
        0\ar[r] & \mathrm{Hom}(Z,I)\ar[r] & \mathrm{Hom}(Y,I)\ar[r] & \mathrm{Hom}(X,I)\ar[r] & 0
        }
    \end{displaymath}
    wherein the right-hand square is a pullback. Now, $\mathrm{Hom}(Y,I)\rightarrow\mathrm{Hom}(X,I)$ is surjective by assumption. As the pullback of a surjection, $\mathrm{Hom}(P,I)\rightarrow\mathrm{Hom}(A,I)$ is surjective.
    Next, consider the class $\mathcal{E}_{\mathcal{I}}$. For any object $X$, $0\rightarrow X$ is clearly in $\mathcal{M}_{\mathcal{I}}$. Therefore, $\mathrm{Id}_{X}\in\mathcal{E}_{\mathcal{I}}$. Consider a commutative diagram
       \begin{displaymath}
        \xymatrix{
        0\ar[r] & A\ar@{=}[d]\ar[r] & P\ar[d]\ar[r] & Z\ar[d]\ar[r] & 0\\
        0\ar[r] & A\ar[r]^{i} & B\ar[r] & C\ar[r] & 0.
        }
    \end{displaymath}
    where the bottom row is an exact sequence in $(\mathcal{C},\mathcal{Q})$, $i\in\mathcal{M}_{\mathcal{I}}$, and the right-hand square is a pullback. Applying $\mathrm{Hom}(-,I)$ gives the commutative diagram
         \begin{displaymath}
        \xymatrix{
        0\ar[r] & \mathrm{Hom}(C,I)\ar[d]\ar[r] & \mathrm{Hom}(B,I)\ar[d]\ar[r] & \mathrm{Hom}(A,I)\ar@{=}[d]\ar[r] & 0\\
        0\ar[r] & \mathrm{Hom}(Z,I)\ar[r] & \mathrm{Hom}(P,I)\ar[r] & \mathrm{Hom}(A,I)\ar[r] & 0.
        }
    \end{displaymath}
    Now, $\mathrm{Hom}(B,I)\rightarrow\mathrm{Hom}(A,I)$ is surjective by assumption. Clearly, then $\mathrm{Hom}(P,I)\rightarrow\mathrm{Hom}(A,I)$ is surjective. Hence, the only axiom missing for $(\mathcal{C},\mathcal{Q}_{\mathcal{I}})$ to be proto-exact is closure of $\mathcal{E}_{\mathcal{I}}$ under composition.
Finally, suppose that $(\mathcal{C},\mathcal{Q})$ is exact. Let $f:Y\rightarrow Z$ and $g:Z\rightarrow U$ be in $\mathcal{E}_{\mathcal{I}}$. Let $K$ denote the kernel of $g$. We consider the pullback of the morphism $K\rightarrow Z$ along $Y\rightarrow Z$. This leads to a commutative diagram of short exact sequences dual to the one in \cite{gillespie2016derived} Proposition 3.3 (4). As in the proof of that proposition, the Snake Lemma implies that $Y\rightarrow U$ is in $\mathcal{E}_{\mathcal{I}}$.
\end{proof}
\begin{remark}
    Every object of $\mathcal{I}$ is injective in $(\mathcal{C},\mathcal{Q_{I}})$.
\end{remark}

\section{Coherence and deconstructibility}\label{sec:cohdec}

In this section we develop the tools necessary to establish the existence of covers and envelopes.

\subsection{Deconstructibility and coherence}\label{sec:deconstruct}

In this subsection we generalise the notion of deconstructibility for exact categories pioneered in \cite{saorin2011exact} and \cite{vst2012exact}. Deconstructibility is essentially a minimal condition that allows one to run small object arguments.

\subsubsection{The definition}

For cotorsion theory, we will need some control over
the objects in our categories. More precisely, we will need them to all be ‘unions’ of a set of objects.

\begin{definition}\label{defn:dec}
Let $\mathcal{C}$ be a category, and let $\mathcal{T}$ be a class of morphisms in $\mathcal{C}$ that is closed under transfinite compositions and pushouts.

\begin{enumerate}
    \item  The class $\mathcal{T}$ in $\mathcal{C}$ is said to be \textit{(approximately) deconstructible in itself} if there is a \textit{set} of morphisms $\tilde{\mathcal{T}}\subset\mathcal{T}$ such that any morphism in $\mathcal{T}$ can be written as a transfinite composition of (approximate) pushouts of maps in $\tilde{\mathcal{T}}$.
    \item 
The class $\mathcal{T}$ is said to be \textit{(approximately) presentably deconstructible in itself} if $\tilde{\mathcal{T}}$ can be chosen such that every object in the domain of a map in $\tilde{\mathcal{T}}$ is small relative to $\mathcal{T}$.
\end{enumerate}
\end{definition}

Note that if $\mathcal{C}$ is an accessible category then all classes that are (approximately) deconstructible in themselves are (approximately) presentably deconstructible in themselves.

 \begin{definition}\label{defn:deconstrexact}[cf. \cite{vst2012exact} Definition 3.9]
Let $(\mathcal{C},\mathcal{Q})$ be a proto-exact category. A class of objects $\mathcal{A}$ in $\mathcal{C}$ is said to be \textit{(approximately/ presentably) deconstructible in itself} if the class of maps $\mathbf{AdMon}_{\mathcal{A}}$ is (approximately/ presentably) deconstructible in itself. If the set $\tilde{\mathcal{T}}$ can be chosen such that the domains and codomains of maps in $\tilde{\mathcal{T}}$ are in $\mathcal{A}$, then the class $\mathcal{A}$ is said to be \textit{hereditarily (approximately/ presentably) deconstructible in itself}.
\end{definition}

It turns out that the hereditary property is useful for proving left properness of the model category of commutative algebras in the flat model structure for monoidal exact categories. This is addressed in \cite{kelly2024flat}.

\subsubsection{Deconstructibility in exact categories}
Let us here precisely clarify the relationship between the definition of deconstructibility above in the setting of exact categories, and \cite{vst2012exact}*{Definition 3.9}.

\begin{proposition}
    Let $(\mathpzc{E},\mathcal{Q})$ be an exact category, and $\mathcal{A}$ a class of objects in $\mathpzc{E}$. If $\mathcal{A}$ is deconstructible in itself in the sense of Definition \ref{defn:deconstrexact}, then it is deconstructible in itself in the sense of \cite{vst2012exact}*{Definition 3.9}.
\end{proposition}

\begin{proof}
    For the set of objects required by \cite{vst2012exact}*{Definition 3.9}, take the set of cokernels of maps appearing in the set $\tilde{\mathcal{T}}$, furnished by Definition \ref{defn:deconstrexact}.
\end{proof}

Now, let us establish a partial converse. 

\begin{definition}[\cite{saorin2011exact}*{Definition 2.3}]
Let $(\mathpzc{E},\mathcal{Q})$ be an exact category and $\mathcal{I}$ a class of admissible monomorphisms. We say that $\mathcal{I}$ is
\begin{enumerate}
\item
 \textit{homological} if the following conditions are equivalent for any object $T\in\mathpzc{E}$:
\begin{enumerate}
\item
$\mathrm{Ext}^{1}(S,T)=0$ for all $S\in\mathrm{Coker}(\mathcal{I})$.
\item
The map $T\rightarrow0$ has the right lifting property against maps in $\mathcal{I}.$
\end{enumerate}
\item
\textit{strongly homological} if given any $j:A\rightarrow B\in\mathbf{AdMon}_{\mathrm{Coker}(\mathcal{I})}$, there is a morphism $i:A'\rightarrow B'$ in $\mathcal{I}$ giving rise to a commutative diagram whose rows are exact sequences
\begin{displaymath}
\xymatrix{
0\ar[r] & A'\ar[r]\ar[d] & B'\ar[r]\ar[d] & S\ar[d]^{\mathrm{Id}_{S}}\ar[r] & 0\\
0\ar[r] & A\ar[r] & B\ar[r] & S\ar[r] & 0.
}
\end{displaymath}
\end{enumerate}
\end{definition}

\begin{remark}
Note that in Condition $(2)$ of the above definition, the left-hand square is a pushout diagram.
\end{remark}

In general, we have the following.

\begin{proposition}[\cite{saorin2011exact}*{Proposition 2.7}]\label{prop:gens}
Let $(\mathpzc{E},\mathcal{Q})$ be an exact category such that either
\begin{enumerate}
    \item 
    $\mathpzc{E}$ has enough projectives, or
    \item 
    $\mathpzc{E}$ has a set of generators, arbitrary coproducts of which exist, and such that every section in $\mathpzc{E}$ has a cokernel.
\end{enumerate}
 Then, for each set of objects $\mathcal{A}$, there is a strongly homological set of admissible monomorphisms $\mathcal{I}$ such that $\mathrm{Coker}(\mathcal{I})$ is the class of objects isomorphic to objects of $\mathcal{A}$. Moreover, if $\mathcal{A}\subset\mathcal{B}$, where $\mathcal{B}$ is a larger class that is closed under kernels of admissible epimorphisms between objects in $\mathcal{B}$, and 
 \begin{enumerate}
     \item 
     in case (1) above, projectives are contained in $\mathcal{B}$, or
     \item 
     in case (2) above $\mathcal{B}$ contains a set $\mathcal{G}$ of generators and arbitrary coproducts of objects of $\mathcal{G}$ exist in $\mathcal{B}$,
 \end{enumerate}
then $\mathcal{I}$ can be chosen such that the domains and codomains of maps in $\mathcal{I}$ are in $\mathcal{B}$. Finally, any admissible monomorphism with cokernel in $\mathcal{A}$ is a pushout of a map in $\mathcal{I}$.
\end{proposition}

\begin{remark}
The claim of the above proposition concerning the domain and codomain being in $\mathcal{B}$ is not stated explicitly in \cite{saorin2011exact}*{Proposition 2.7}, but may be immediately deduced from the constructive proof. Typically $\mathcal{B}$ will be the class $\mathrm{Filt}(\mathcal{A})$ of $\mathcal{A}$-filtered objects, discussed in loc. cit.
\end{remark}

\begin{corollary}
 Let $(\mathpzc{E},\mathcal{Q})$ be an exact category such that either
\begin{enumerate}
    \item 
    $(\mathpzc{E},\mathcal{Q})$ has enough projectives, or
    \item 
     $(\mathpzc{E},\mathcal{Q})$ has a set of generators, arbitrary coproducts of which exist, and such that every section in $\mathpzc{E}$ has a cokernel.
\end{enumerate}
If $\mathcal{A}$ is deconstructible in itself in the sense of \cite{vst2012exact}*{Definition 3.9}, then it is deconstructible in itself in the sense of Definition \ref{defn:deconstrexact}. 
\end{corollary}

\begin{proof}
    For the class of morphisms required by Definition \ref{defn:deconstrexact}, take the class $\mathcal{I}$ furnished by Proposition \ref{prop:gens}.
\end{proof}

 \subsubsection{Deconstructibility, covers and envelopes}

We have the following general result.

\begin{proposition}\label{pgrop:precoverenv}
Let $(\mathcal{C},\mathcal{Q})$ be a proto-exact category. Suppose that a class $\mathcal{A}$ of objects in $\mathcal{C}$ is such that $\mathbf{AdMon}_{\mathcal{A}}$ is  presentably deconstructible in itself. Then the following hold.
\begin{enumerate}
    \item 
    Any object $X$ admits an admissible monomorphism $i:X\rightarrow B$ with $B\in\mathcal{A}^{\perp}$. Moreover, $i$ can be chosen to be an $\mathcal{A}^{\perp}$-special preenvelope. 
    \item 
    If $(\mathcal{C},\mathcal{Q})$ satisfies the strong right obscure axiom, then any object $X$ admits an $\mathcal{A}$-cover $\pi:A\rightarrow X$. If $\mathcal{A}$ contains a generator, the map $\pi$ can be chosen to be an admissible epimorphism. If further $(\mathcal{C},\mathcal{Q})$ is an exact category, then $\pi$ can be chosen to be an $\mathcal{A}$-special precover.
    \end{enumerate}
    \end{proposition}

\begin{proof}
\begin{enumerate}
\item 
By the small object argument, we may factor any map $X\rightarrow 0$ as $X\rightarrow B\rightarrow 0$, where $X\rightarrow B\in\mathbf{AdMon}_{\mathcal{A}}$, and $B\rightarrow 0$ has the right lifting property against all maps in $\mathbf{AdMon}_{\mathcal{A}}$. In particular, $B\in\mathcal{A}^{\perp}$.
\item
By the small object argument we may factor any map $0\rightarrow X$ as $0\rightarrow A\rightarrow X$ with $0\rightarrow A\in\mathbf{AdMon}_{\mathcal{A}}$ and $A\rightarrow X$ having the right lifting property with respect to all maps in $\mathbf{AdMon}_{\mathcal{A}}$. In particular, $A\in\mathcal{A}$, and $A\rightarrow X$ has the right lifting property against all maps of the form $0\rightarrow A'$, with $A'\in\mathcal{A}$. In other words, $A\rightarrow X$ is an $\mathcal{A}$-precover. If $\mathcal{A}$ contains a generating set and $(\mathcal{C},\mathcal{Q})$ satisfies the strong right obscure axiom, then an $\mathcal{A}$-precover is an admissible epimorphism by Lemma \ref{lem:adepicheckgen}. The final claim concerning the exact category case proceeds identically to \cite{saorin2011exact} Theorem 2.13.
\end{enumerate}
\end{proof}

\subsection{Weakly elementary proto-exact categories}
We will need certain colimits to be exact.

\begin{definition}
    Let $(\mathcal{C},\mathcal{Q})$ be a primitive-exact category, and $\mathcal{A}$ a class of objects in $\mathcal{C}$. We say that $(\mathcal{C},\mathcal{Q})$ is \textit{weakly} $\mathcal{A}$\textit{-elementary} if whenever 
    \begin{displaymath}
        \xymatrix{
        A_{0}\ar[d]\ar[r] & A_{1}\ar[r]\ar[d] &\ldots\ar[r]\ar[d] & A_{\lambda}\ar[d]\ar[r] & \ldots\ar[d]\ar[r] &\\
         B_{0}\ar[d]\ar[r] & B_{1}\ar[r]\ar[d] &\ldots\ar[r]\ar[d] & B_{\lambda}\ar[d]\ar[r] & \ldots\ar[d]\ar[r] &\\
              C_{0}\ar[r] & C_{1}\ar[r] &\ldots\ar[r] & C_{\lambda}\ar[r] & \ldots\ar[r] &.
        }
    \end{displaymath}
    is a transfinite sequence indexed by some $\Lambda$, in which each row is exact, and each $C_{\lambda}\in\mathcal{A}$, then 
    $$\emptyset\rightarrow\colim_{\lambda\in\Lambda}A_{\lambda}\rightarrow \colim_{\lambda\in\Lambda}B_{\lambda}\rightarrow \colim_{\lambda\in\Lambda}C_{\lambda}\rightarrow\bullet$$
    is short exact, and $\colim_{\lambda\in\Lambda}C_{\lambda}\in\mathcal{A}$.
\end{definition}

\begin{definition}
   Let $(\mathcal{C},\mathcal{Q})$ be a primitive-exact category, and $\mathcal{A}$ a class of objects in $\mathcal{C}$. We say that $\mathcal{A}$ is \textit{closed under extensions} if, whenever 
   $$\emptyset\rightarrow X\rightarrow Y\rightarrow Z\rightarrow\bullet$$
   is a short exact sequence exact sequence with $X$ and $Z$ in $\mathcal{A}$, then $Y\in\mathcal{A}$.
 \end{definition}

\begin{proposition}
         Let $(\mathcal{C},\mathcal{Q})$ be a primitive-exact category, and $\mathcal{A}$ a class of objects in $\mathcal{C}$. Suppose that $(\mathcal{C},\mathcal{Q})$ is weakly $\mathcal{A}$-elementary, and that $\mathcal{A}$ is closed under extensions. Let $\Lambda$ be an ordinal, and $F:\Lambda\rightarrow\mathcal{C}$ a cocontinuous functor with $F(\lambda)\rightarrow F(\lambda+1)$ being in $\mathbf{AdMon}_{\mathcal{A}}$ for every $\lambda\in\Lambda$. Then the map $F(0)\rightarrow\colim_{\lambda\in\Lambda}F(\lambda)$ is in $\mathbf{AdMon}_{\mathcal{A}}$.
\end{proposition}

\begin{proof}
    Consider the diagram
        \begin{displaymath}
        \xymatrix{
        F(0)\ar[d]\ar[r] & F(0)\ar[r]\ar[d] &\ldots\ar[r]\ar[d] & F(0)\ar[d]\ar[r] & \ldots\ar[d]\ar[r] &\\
         F(0)\ar[d]\ar[r] & F(1)\ar[r]\ar[d] &\ldots\ar[r]\ar[d] & F(\lambda)\ar[d]\ar[r] & \ldots\ar[d]\ar[r] &\\
              C_{0}\ar[r] & C_{1}\ar[r] &\ldots\ar[r] & C_{\lambda}\ar[r] & \ldots\ar[r] &.
        }
    \end{displaymath}
    By transfinite induction each column is in $\mathbf{AdMon}_{\mathcal{A}}$. Indeed, the successor ordinal case is a consequence of Proposition \ref{prop:cokernelcomp}, and the limit ordinal case is a consequence of weak $\mathcal{A}$-elementarity. Thus the map $F(0)\rightarrow\colim_{\lambda\in\Lambda}F(\lambda)$ is in $\mathbf{AdMon}_{\mathcal{A}}$.
\end{proof}

In particular, in the situation of the above proposition, $\mathcal{A}$ is closed under trans-finite extensions.

\begin{proposition}

\begin{enumerate}
    \item 
    Suppose that $(\mathcal{C},\mathcal{Q})$ is left totally primitive-exact. Then the class $\mathbf{AdMon}_{\mathcal{A}}$ is closed under pushouts.
     \item 
    Suppose that $(\mathcal{C},\mathcal{Q})$ is a primitive-exact category that satisfies the strong left obscure axiom and is finitely cocomplete. If the class $\mathcal{A}$ is closed under retracts, then the class $\mathbf{AdMon}_{\mathcal{A}}$ is closed under retracts.
\end{enumerate}
  
\end{proposition}

\begin{proof}
\begin{enumerate}
\item 
This follows from Proposition \ref{prop:pullker}.
    \item 
     For retracts, consider a commutative diagram
         \begin{displaymath}
        \xymatrix{
        X\ar[d]^{f}\ar[r] & Z\ar[d]^{g}\ar[r] & X\ar[d]^{f}\\
        Y\ar[r] & W\ar[r] & Y
        }
    \end{displaymath}
    where the top and bottom rows compose to the identity, and $g$ is an admissible monomorphism with cokernel in $\mathcal{A}$. The obscure axiom implies that $X\rightarrow Z$ is an admissible monomorphism. Then the composition $X\rightarrow Z\rightarrow W$ is an admissible monomorphism. Again by the obscure axiom, $X\rightarrow Y$ is an admissible monomorphism. The cokernel of $f$ is then clearly a retract of the cokernel of $g$, and hence is in $\mathcal{A}$. Note that without the additional restriction on cokernels, closure under retracts is dual to \cite{positselski2025pure}*{Lemma 10.1}.
\end{enumerate}

\end{proof}

We also have the following important, but weaker, condition on exactness of colimits.

\begin{definition}\label{defn:coherent}
Let $(\mathcal{C},\mathcal{Q})$ be a primitive-exact and $\mathcal{B}$ a class of objects in $\mathcal{C}$. A class of objects $\mathcal{A}$ in $\mathcal{C}$ is said to be $\mathcal{B}$-\textit{unionable} if for any ordinal $\Lambda$, and any $\Lambda$-indexed transfinite sequence
\begin{displaymath}
\xymatrix{
M_{0}\ar[r]\ar[d] &\ldots\ar[r] & M_{\alpha}\ar[d]\ar[r] &  \ldots\\
M\ar[d]\ar[r] &\ldots\ar[r] & M\ar[r]\ar[d] & \ldots\\
C_{0}\ar[r]  &\ldots\ar[r] &C_{\alpha}\ar[r] & \ldots
}
\end{displaymath}
where $M\in\mathcal{B}$, each $C_{\alpha}\in\mathcal{A}$ for all $\alpha\in\Lambda$, each $M_{\alpha}\rightarrow M_{\alpha+1}$ is an admissible monomorphism with cokernel in $\mathcal{A}$, and each column is exact, the sequence
$$\emptyset\rightarrow\colim_{\alpha\in\Lambda}M_{\alpha}\rightarrow M\rightarrow\colim_{\alpha\in\Lambda}C_{\alpha}\rightarrow \bullet$$
is an exact sequence, and $\colim_{\alpha\in\Lambda}C_{\alpha}\in\mathcal{A}$.

In particular, we say that $\mathcal{A}$ is \textit{unionable} if it is $\mathrm{Ob}(\mathcal{C})$-unionable, where $\mathrm{Ob}(\mathcal{C})$ is the class of all objects of $\mathcal{C}$.
\end{definition}

We think of unionability as being able to take unions of admissible subobjects.

\subsection{Pure monomorphisms}

We will also need some tools from purity theory.

 \begin{definition}[\cite{adamek1994locally} Definition 2.27 (for Part (1))]
 Let $\mathcal{C}$ be any category.
 \begin{enumerate}
 \item
 A morphism $f:A\rightarrow B$ is said to be $\lambda$-\textit{pure} if for each commutative square
 \begin{displaymath}
 \xymatrix{
 A'\ar[d]^{f'}\ar[r]^{u} & A\ar[d]^{f}\\
 B'\ar[r]^{v} & B
 }
 \end{displaymath}
 with $A'$ and $B'$ $\lambda$-presentable, there is a morphism $\overline{u}:B'\rightarrow A$ such that $u=\overline{u}\circ f'$. 
 \item
  A morphism $f:X\rightarrow Y$ in $\mathcal{C}$ is said to be a $\lambda$-\textit{pure epimorphism} if for all $\lambda$-presentable objects $E$, $\mathrm{Hom}(E,f):\mathrm{Hom}(E,X)\rightarrow\mathrm{Hom}(E,Y)$ is an epimorphism of sets. 
 \end{enumerate}
 \end{definition}

 By \cite{adamek1994locally}*{Proposition 2.29}, in a $\lambda$-accessible category, $\lambda$-pure morphisms are monomorphisms. In a $\lambda$-presentable category, we further have the following.
 
   \begin{proposition}[\cite{adamek1994locally} Proposition 2.30 (2)]\label{prop:lambdadirected}
Let $\mathcal{C}$ be a locally $\lambda$-presentable category. A morphism $f:X\rightarrow Y$ in $\mathcal{C}$ is a $\lambda$-pure monomorphism if and only if it is a $\lambda$-directed colimit in $\mathrm{Mor}(\mathcal{C})$ of split monomorphisms. 
\end{proposition}

The class of all $\lambda$-pure monomorphisms will be denoted $\mathbf{PureMon}_{\lambda}$.

\begin{proposition}[\cite{adamek1994locally}*{Remark (1) after 2.28 Examples}, \cite{positselski2025pure}*{Proposition 13.2}]
    The class of $\lambda$-pure monomorphisms in a locally $\lambda$-presentable category is
    \begin{enumerate}
    \item
    closed under composition;
        \item 
            closed under pushout along all morphisms. 
    \end{enumerate}
\end{proposition}

\begin{theorem}[\cite{adamek1994locally} Theorem 2.33]\label{thm:deconstructpres}
Let $\mathcal{K}$ be a $\lambda$-accessible category. There exist arbitrarily large regular cardinals $\gamma\ge\lambda$ such that every map $A\rightarrow B$ in $\mathcal{K}$ with $A$ $\gamma$-presentable factors through a $\lambda$-pure monomorphism $\overline{f}:\overline{A}\rightarrow B$ with $\overline{A}$ $\gamma$-presentable.
\end{theorem}

\begin{corollary}\label{cor:diagramofsbuboejcts}
Let $\mathcal{C}$ be a locally $\lambda$-presentable category. There are arbitrarily large regular cardinals $\gamma\ge\lambda$ such that any object $E$ of $\mathcal{C}$ can be written as a $\gamma$-filtered colimit $E\cong\colim_{\mathcal{J}}E_{j}$ where 
\begin{enumerate}
\item
each $E_{j}$ is $\gamma$-presentable;
\item
each map $E_{j}\rightarrow E$ is a $\lambda$-pure monomorphism. 
\end{enumerate}
\end{corollary}

\begin{proof}
By Theorem \ref{thm:deconstructpres} there are arbitrarily large regular cardinals $\gamma$ such that whenever $D\rightarrow E$ is a map with $D$ $\gamma$-presentable, then $D\rightarrow E$ factors as $D\rightarrow\overline{D}\rightarrow E$ where $\overline{D}$ is $\gamma$-presentable and $\overline{D}\rightarrow E $ is $\lambda$-pure. Now, $\mathcal{C}$ is also locally $\gamma$-presentable. 
Let $\mathcal{J}$ denote the poset of $\lambda$-pure, $\gamma$-presentable subobjects of $E$. For each $j\in\mathcal{J}$ fix a representative $E_{j}$ of the corresponding equivalence class of subobjects. We claim that $E\cong\colim_{E_{j}\in\mathcal{J}}E_{j}$. The category $\mathcal{J}$ is $\gamma$-filtered. Let $\{E_{i}\rightarrow E\}_{i\in\mathcal{I}}$ be a $\gamma$-small collection of $\lambda$-pure, $\gamma$-presentable subobjects. Consider the coproduct $\coprod_{i\in\mathcal{I}}E_{i}$. This is $\gamma$-presentable. The induced morphism $\coprod_{i\in\mathcal{I}}E_{i}\rightarrow E$ factors through a $\gamma$-presentable, $\lambda$-pure subobject $E_{\mathcal{I}}\rightarrow E$. Then all of the maps $E_{i}\rightarrow E$ factor as $E_{i}\rightarrow E_{\mathcal{I}}\rightarrow E$. Thus, $\colim_{E_{j}\in\mathcal{J}}E_{j}\rightarrow E$ is a $\gamma$-pure monomorphism and hence a regular monomorphism by \cite{adamek1994locally}*{Proposition 2.31}. Let $G$ be a $\gamma$-presentable object. Then $\mathrm{Hom}(G,\colim_{E_{j}\in\mathcal{J}}E_{j})\rightarrow \mathrm{Hom}(G,E)$ is an epimorphism. Indeed, any map $G\rightarrow E$ factors through a $\lambda$-pure, $\gamma$-presentable subobject $\overline{G}\rightarrow E$. Since $\mathcal{C}$ is also $\gamma$-presentable, the $\gamma$-presentable objects provide a set of generators, which means that $\colim_{E_{j}\in\mathcal{J}}E_{j}\rightarrow E$ is an epimorphism, and hence an isomorphism.

\end{proof}

\begin{definition}
    Let $\lambda$ be a regular cardinal, and $(\mathcal{C},\mathcal{Q})$ a primitive-exact category, with $\mathcal{C}$ being $\lambda$-accessible (resp. locally $\lambda$-presentable). The category $(\mathcal{C},\mathcal{Q})$ is said to be \textit{purely }$\lambda$-\textit{accessible} (resp. \textit{purely locally} $\lambda$-\textit{presentable}) if $\lambda$-pure monomorphisms are admissible monomorphisms.
\end{definition}

\begin{remark}
    If $(\mathcal{C},\mathcal{Q})$  is a primitive-exact category such that $\mathcal{C}$ is locally $\lambda$-presentable, split monomorphisms are admissible, and $\lambda$-filtered colimits are exact - that is the colimit of any diagram of term-wise short exact sequences is a short exact sequence - then by Proposition \ref{prop:lambdadirected}, $(\mathcal{C},\mathcal{Q})$ is purely locally $\lambda$-presentable. 
\end{remark}

\subsection{Pre-coherence}\label{subsubsec:coherence}

The various notions of pre-coherence in this subsection are generalisations of similar concepts introduced in \cite{positselski2023locally} (in the exact setting) and \cite{positselski2025pure}.

For $\mathcal{C}$ a $\lambda$-accessible category, we denote by $\mathcal{C}_{\lambda}$ the subcategory of $\lambda$-presentable objects. 

\begin{definition}\label{defn:coherentproto}
\begin{enumerate}
    \item 
       Let $(\mathcal{C},\mathcal{Q})$ be a primitive exact category with $\mathcal{C}$ $\lambda$-accessible. Let $\gamma\ge\lambda$ be a cardinal. We say that $(\mathcal{C},\mathcal{Q})$ is \textit{ locally }$(\gamma,\lambda)$-\textit{subobject pre-coherent} (resp. \textit{ locally }$(\gamma,\lambda)$-\textit{pre-coherent}) if any diagram
    \begin{displaymath}
        \xymatrix{
         & S\ar[d]^{f}\\
         D\ar[r]^{i} & E        
        }
    \end{displaymath}
    with $f$ an admissible epimorphism, $i$ an admissible monomorphism (resp. $i$ an arbitrary morphism), and $D\in\mathcal{C}_{\lambda}$ extends to a diagram
        \begin{displaymath}
        \xymatrix{
        X\ar[d]^{g}\ar[r]^{j} & S\ar[d]^{f}\\
         D\ar[r]^{i} & E        
        }
    \end{displaymath}
    with $g$ an admissible epimorphism, $j$ an admissible monomorphism (resp. an arbitrary morphism), and $X\in\mathcal{C}_{\gamma}$.
    \item 
    We say that $(\mathcal{C},\mathcal{Q})$ is \textit{upwardly locally }$(\gamma,\lambda)$-\textit{subobject pre-coherent} (resp. \textit{ upwardly locally }$(\gamma,\lambda)$-\textit{pre-coherent}) if there are arbitrarily large $\mu\ge\gamma$, $\kappa\ge\lambda$ such that $(\mathcal{C},\mathcal{Q})$ is locally  $(\mu,\kappa)$-pre-coherent (resp. locally $(\mu,\kappa)$-subobject pre-coherent).
    \item
        We say that $(\mathcal{C},\mathcal{Q})$ is \textit{locally }$\lambda$-\textit{pre-coherent} (resp.  \textit{locally }$\lambda$-\textit{subobject pre-coherent}) if it is $(\lambda,\lambda)$-pre-coherent (resp. $(\lambda,\lambda)$-subobject pre-coherent).
        \item 
            We say that $(\mathcal{C},\mathcal{Q})$ is \textit{upwardly locally }$\lambda$-\textit{subobject pre-coherent} (resp. \textit{ upwardly locally }$\lambda$-\textit{pre-coherent}) if there are arbitrarily large $\kappa\ge\lambda$ 
            such that $(\mathcal{C},\mathcal{Q})$ is locally $\kappa$-subobject pre-coherent (resp. locally  $\kappa$-pre-coherent).
\end{enumerate}

\begin{remark}
     Note that if $(\mathcal{C},\mathcal{Q})$ is upwardly locally $\lambda$-pre-coherent, then there are arbitrarily large $\kappa\ge\lambda$ such that $(\mathcal{C},\mathcal{Q})$ is locally $(\kappa,\kappa)$-pre-coherent. In particular, $(\mathcal{C},\mathcal{Q})$ is upwardly locally $(\lambda,\lambda)$-pre-coherent. The same is true for the subobject variants.
\end{remark}

\end{definition}
 \begin{example}\label{ex:projcoh}
     Let $(\mathcal{C},\mathcal{Q})$ be a deflation primitive-exact category with $\mathcal{C}$ $\lambda$-accessible, in which every  object is projective. Then $(\mathcal{C},\mathcal{Q})$ is locally $(\lambda,\lambda)$-pre-coherent. If in addition $(\mathcal{C},\mathcal{Q})$ satisfies the strong left obscure axiom, then $(\mathcal{C},\mathcal{Q})$ is locally $(\lambda,\lambda)$-subobject pre-coherent for any $\lambda$. 
       For example, $\mathrm{Set}_{*}$ is locally $\lambda$-pre-coherent and locally $\lambda$-subobject pre-coherent for any $\lambda$.
     \end{example}

\begin{proposition}
    Let $(\mathpzc{E},\mathcal{Q})$ be a category with QE-epi class, with $\mathcal{C}$ being $\kappa$-accessible, and $\mathcal{Q}$ being locally $\kappa$-coherent in the sense of \cite{positselski2025pure}*{Before Lemma 8.4}. Suppose further that $(\mathcal{C},\mathcal{Q})$  satisfies the right obscure axiom. Then $(\mathcal{C},\mathcal{Q})$ is upwardly locally $\lambda$-pre-coherent.
\end{proposition}

\begin{proof}
This follows from \cite{positselski2025pure}*{Proposition 11.1}, noting that if $\mathcal{Q}$ is locally $\kappa$-coherent, then for arbitrarily large $\lambda\ge\kappa$, $\mathcal{Q}$ is also locally $\lambda$-coherent. In the exact category case, the result also follows from \cite{positselski2023locally}*{Lemma 1.5, Lemma 1.9}.
\end{proof}

Later we will give a partial converse. 

The reader may wonder why we require the additional flexibility of $(\gamma,\lambda)$-pre-coherence and its upward variant. Hopefully the results below will make this clear.

\begin{proposition}\label{prop:obscuresub}
    Let $(\mathcal{C},\mathcal{Q})$ be a purely $\lambda$-accessible deflation primitive-exact category satisfying the strong left obscure axiom, and in which split epimorphisms are admissible. If $(\mathcal{C},\mathcal{Q})$ is (upwardly) $(\gamma,\lambda)$-pre-coherent then there is $\mu\ge\gamma$ such that it is (upwardly) $(\mu,\lambda)$-subobject pre-coherent.
\end{proposition}

\begin{proof}
Since for any $\kappa\ge\lambda$, $(\mathcal{C},\mathcal{Q})$ is also a purely $\kappa$-accessible exact category satisfying the right obscure axiom, the upward case will follow from the non-upward case.
  Let $\mu\ge\gamma$ be such that any morphism $A\rightarrow B$ with $A$ being $\mu$-presentable factors through a $\lambda$-pure monomorphism $\overline{A}\rightarrow B$ with $\overline{A}$ being $\mu$-presentable. Consider a diagram
        \begin{displaymath}
        \xymatrix{
         & S\ar[d]^{f}\\
         D\ar[r]^{i} & E        
        }
    \end{displaymath}
    with $f$ an admissible epimorphism, $i$ an admissible monomorphism, and $D$ being $\lambda$-presentable. By $(\gamma,\lambda)$-pre-coherence, extend to a diagram 
            \begin{displaymath}
        \xymatrix{
        X\ar[d]^{g}\ar[r]^{j} & S\ar[d]^{f}\\
         D\ar[r]^{i} & E        
        }
    \end{displaymath}
    with $g$ an admissible epimorphism and $X$ $\gamma$-presentable. 
    
    Consider the fibre product $D\times_{E}S$. The map $i':D\times_{E}S\rightarrow S$ is an admissible monomorphism. By the universal property of the pullback the maps $X\rightarrow D$ and $X\rightarrow S$ factor through $X\rightarrow D\times_{E}S$.
     The map $X\rightarrow D\times_{E}S$ in turn factors through $X\rightarrow\overline{X}\rightarrow D\times_{E}S$ with $\overline{X}$ being $\mu$-presentable, and $\overline{X}\rightarrow D\times_{E}S$ being a $\lambda$-pure, and hence admissible, monomorphism. We therefore get a commutative diagram
           \begin{displaymath}
        \xymatrix{
        \overline{X}\ar[d]^{g'}\ar[r]^{j'} & S\ar[d]^{f}\\
         D\ar[r]^{i} & E        
        }
    \end{displaymath}
    where $j'$ is now an admissible monomorphism, and $g'$ is an admissible epimorphism by the strong right obscure axiom. 
\end{proof}

\begin{proposition}\label{prop:monadprecoh}
      Let $(\mathcal{C},\mathcal{Q})$ be upwardly $(\gamma,\lambda)$-pre-coherent. Let $\mathrm{T}:\mathcal{C}\rightarrow\mathcal{C}$ be a monad that commutes with finite colimits and is $\lambda$-accessible. Then there are arbitrarily large regular cardinals $\mu\ge\gamma$, $\kappa\ge\lambda$ such that the category of algebras for $\mathrm{T}$, $\mathrm{Alg_{T}}$, is $(\mu,\kappa)$-pre-coherent. In other words, $\mathrm{Alg_{T}}$ is upwardly $(\gamma,\lambda)$-pre-coherent.
\end{proposition}

\begin{proof}
    The category $\mathrm{Alg_{T}}(\mathcal{C})$ is also $\gamma$-accessible for some $\gamma$ by \cite{adamek1994locally}*{Theorem 2.78}. The set of $\gamma$-small objects, $\mathrm{Alg_{T}}(\mathcal{C})_{\gamma}$ therefore forms a set. There is therefore a $\kappa\ge\gamma$ such that for any object $A$ of $\mathrm{Alg_{T}}(\mathcal{C})_{\gamma}$,  the underlying object $|A|$ is $\kappa$-presentable.
     Let  

    \begin{displaymath}
        \xymatrix{
         & S\ar[d]^{f}\\
         D\ar[r]^{i} & E        
        }
    \end{displaymath}
    be a commutative diagram in $\mathrm{Alg_{T}}$ with $D$ $\gamma$-presentable. Pass to the diagram in the underlying category
        \begin{displaymath}
        \xymatrix{
         & |S|\ar[d]^{|f|}\\
         |D|\ar[r]^{|i|} & |E|.        
        }
    \end{displaymath}
There is $\mu\ge\kappa$ such that we can find a lift
        \begin{displaymath}
        \xymatrix{
         X\ar[r]\ar[d] & |S|\ar[d]^{|f|}\\
         |D|\ar[r]^{|i|} & |E|.        
        }
    \end{displaymath}
    with $X$ $\mu$-presentable. Then we get a commutative diagram 
            \begin{displaymath}
        \xymatrix{
         \mathrm{Free_{T}}(X)\ar[r]\ar[d] & S\ar[d]^{|f|}\\
         D\ar[r]^{|i|} & E       
        }
    \end{displaymath}
    in $\mathrm{Alg_{T}}(\mathcal{C})$ with $\mathrm{Free_{T}}(X)$ being $\mu$-presentable.
\end{proof}
\begin{remark}\label{rem:failuresubT}
    When $(\mathcal{C},\mathcal{Q})$ satisfies the strong right obscure axiom, Proposition \ref{prop:monadprecoh} does remain true after replacing pre-coherence by subobject pre-coherence, thanks to Proposition \ref{prop:obscuresub}.  However, absent the right obscure axiom, it is not clear that  the statement of Proposition \ref{prop:monadprecoh} remains true when we replace pre-coherence by subobject pre-coherence. 
\end{remark}

The next result is inspired by the techniques of \cite{gillespie2006flat}*{Lemma 4.9}.

\begin{lemma}\label{lem:subepi}
Let $(\mathcal{C},\mathcal{Q})$ be a purely $\lambda$-accessible deflation primitive-exact category satisfying the strong obscure axiom. Then there is $\gamma\ge\lambda$ such that $(\mathcal{C},\mathcal{Q})$ is upwardly $\gamma$-subobject pre-coherent. 
 \end{lemma}

 \begin{proof}

 By passing to an (arbitrarily) larger cardinal, we may assume that any morphism $A\rightarrow B$ with $A$ $\gamma$-presentable factors through a $\lambda$-pure monomorphism $\overline{A}\rightarrow B$ with $\overline{A}$ $\gamma$-presentable.
 Consider a diagram
     \begin{displaymath}
        \xymatrix{
         & S\ar[d]^{f}\\
         D\ar[r]^{i} & E        
        }
    \end{displaymath}
    with $D$ $\gamma$-presentable, $i$ an admissible monomorphism, and $f$ an admissible epimorphism.
  Since $i$ is an admissible monomorphism, we can pull back $i$ along $f$. The pullback of $i$, $i':D\times_{E}S\rightarrow S$ is still an admissible monomorphism. In particular, we may assume from the outset that $D=E$ and $i$ is the identity map. 
 
 Let $K=\mathrm{Ker}(f)$. Using Corollary \ref{cor:diagramofsbuboejcts}, write $K\cong\colim_{i\in\mathcal{I}}K_{i}$ as a $\gamma$-filtered colimit with each $K_{i}$ being $\gamma$-presentable, and each $K_{i}\rightarrow K$ being a $\lambda$-pure monomorphism. 

 Also write $S\cong\colim_{j\in\mathcal{J}}S_{j}$ as a $\gamma$-filtered colimit, with each $S_{j}$ being $\gamma$-presentable and each $S_{j}\rightarrow S$ being a $\lambda$-pure monomorphism. By passing to a cofinal diagram we may assume that $\mathcal{I}=\mathcal{J}$, and that $K\rightarrow S$ is a colimit of maps $K_{i}\rightarrow S_{i}$. By the left obscure axiom, each $K_{i}\rightarrow S_{i}$ is an admissible monomorphism. Let $E_{i}= \mathrm{Coker}(K_{i}\rightarrow S_{i})$. Then $S_{i}\rightarrow E_{i}$ is an admissible epimorphism, and $E_{i}$ is $\gamma$-presentable. We then have $E\cong\colim_{i\in\mathcal{I}}E_{i}$ with each $E_{i}$ being $\gamma$-presentable. Since $E$ is $\gamma$-presentable, the map $\mathrm{Id}:E\rightarrow E$ must then factor through some $E_{i}$. The morphism $E\rightarrow E_{i}$ is then a split, and hence admissible, epimorphism. Thus $S_{i}\rightarrow E_{i}\rightarrow E$ is an admissible epimorphism. 
 
     \end{proof}

\subsubsection{Pre-coherence and coherence}

Here we prove some results that compare our notion of pre-coherence with Positselski's notions of coherence \cite{positselski2025pure} (and, in the exact case, \cite{positselski2023locally}).


\begin{definition}[\cite{positselski2025pure}*{Before Lemma 6.4 and before Lemma 8.4}]
\begin{enumerate}
    \item 
      Let $\mathcal{C}$ be a $\kappa$-accessible category, and let $\mathcal{M}$ be a QE-mono class on $\mathcal{C}$. We say that $\mathcal{M}$ is \textit{locally }$\kappa$-\textit{coherent} if any map $f\in\mathcal{M}$ is a $\kappa$-directed colimit of a diagram of admissible monomorphisms $f_{i}:C_{i}\rightarrow D_{i}$ with $C_{i},D_{i},$ and $\mathrm{Coker}(f_{i})$ being $\kappa$-presentable.
      \item
           Let $\mathcal{C}$ be a $\kappa$-accessible category, and let $\mathcal{E}$ be a QE-epi class on $\mathcal{C}$. We say that $\mathcal{E}$ is \textit{locally }$\kappa$-\textit{coherent} if any sequence in $\mathcal{Q}_{\mathcal{E}}$ is a $\kappa$-directed colimit of sequences in $\mathcal{Q}_{\mathcal{E}}$ all of wohse terms are $\kappa$-presentable.
\end{enumerate}
   
\end{definition}

\begin{definition}
    Say that a left totally conflation (resp. right totally deflation) proto-exact category $(\mathcal{C},\mathcal{Q})$, with $\mathcal{C}$ $\kappa$-accessible, is locally $\kappa$-coherent, if $\mathbf{AdMon}$ (resp. $\mathbf{AdEpi}$) is a locally $\kappa$-coherent QE-mono class. 
\end{definition}

Note that an exact category $(\mathpzc{E},\mathcal{Q})$ with $\mathpzc{E}$ $\kappa$-accessible is locally $\kappa$-coherent in the sense of \cite{positselski2023locally}, if and only if it is locally $\kappa$-coherent as a proto-exact category in the sense above. 

\begin{proposition}\label{prop:fpreserves}
    Let $F:\mathcal{C}\rightarrow\mathcal{D}$ be an accessible functor between locally presentable categories. Then there is a regular cardinal $\kappa\ge\lambda$, such that for any regular cardinal $\mu$ with $\kappa\triangleleft\mu$, $F$ preserves $\mu$-presentable objects. 
\end{proposition}

\begin{proof}
This argument is due to \cite{ncatcafe:mythsharp}. As we cannot find a more citable reference, we include the proof. By passing to a larger cardinal if necessary, we may assume that $F$ is $\lambda$-accessible, and both $\mathcal{C}$ and $\mathcal{D}$ are $\lambda$-presentable. Now, there is a set of isomorphism classes of objects of the form $F(X)$ for $X$ $\lambda$-presentable. There is then $\kappa\ge\lambda$ such that all objects of the form $F(X)$ are $\kappa$-presentable. By passing again to a larger cardinal using \cite{adamek1994locally}*{Corollary 2.14}, we may assume that $\lambda\triangleleft\kappa$. Now, by \cite{adamek1994locally}*{Remark 2.15}, for $\kappa\triangleleft\mu$ any $\mu$-presentable object $X$ is a retract of a $\mu$-small $\lambda$-directed colimit of $\lambda$-presentable objects. Then $F(X)$ is a retract of a $\mu$-small $\lambda$-directed colimit of $\kappa$-presentable, and hence $\mu$-presentable, objects. Thus $F(X)$ is itself $\mu$-presentable, again by \cite{adamek1994locally}*{Remark 2.15}.
\end{proof}

\begin{lemma}\label{lem:mukapppull}
    Let $\mathcal{C}$ be a locally $\lambda$-presentable category. Then there is  a cardinal $\kappa\ge\lambda$ such that for all cardinals $\mu$ with $\kappa\triangleleft\mu$, the class $\mathcal{C}_{\mu}$ of $\mu$-presentable objects is closed under finite limits. 
\end{lemma}

\begin{proof}
    Again, this is well known. See, for example, \cite{324621} and \cite{306129}. The locally finitely presentable case is a consequence of \cite{beke2012abstract}*{Proposition 4.3}. Once more we include the proof for the general case due to our inability to find a citable reference. First we may assume that the terminal object is $\lambda$-presentable. We therefore only need to consider pullbacks. Let $\mathcal{C}^{(\rightarrow\leftarrow)}$ be the category of cospans in $\mathcal{C}$. This category is locally $\lambda$-presentable, and the $\lambda$-presentable objects are those diagrams that are pointwise $\lambda$-presentable. 
    
    The pullback functor
    $$P:\mathcal{C}^{(\rightarrow\leftarrow)}\rightarrow\mathcal{C}$$
    is a right adjoint, and hence is accessible. By Proposition \ref{prop:fpreserves} there is a $\kappa\ge\lambda$ such that whenever $\kappa\triangleleft\mu$, $P$ preserves $\mu$-presentable objects. This completes the proof.

\end{proof}

\begin{proposition}
   Let $(\mathcal{C},\mathcal{Q})$ be a category with QE-epi class that is $\lambda$-accessible, is locally $\lambda$-pre-coherent, and satisfies the strong right obscure axiom. Suppose further that the category $\mathcal{C}_{\lambda}$ is closed under limits of pullback diagrams in $\mathcal{C}$ with one leg being an admissible epimorphism. Then the class of exact sequences is contained within the class of $\lambda$-filtered colimits of exact sequences between $\lambda$-presentable objects. In particular, if $\lambda$-filtered colimits are exact, then the class of exact sequences coincides with the class of $\lambda$-filtered colimits of exact sequences between $\lambda$-presentable objects.
\end{proposition}

\begin{proof}
This can be deduced from the `if' direction of the proof of \cite{positselski2025pure}*{Proposition 11.1}.
\end{proof}

Now, using Lemma \ref{lem:mukapppull}, we get the following.

\begin{corollary}
    Suppose that a category with QE-epi class $(\mathcal{C},\mathcal{Q})$, with $\mathcal{C}$ being locally $\lambda$-presentable, is upwardly locally $\lambda$-pre-coherent and satisfies the right obscure axiom. Then there is $\kappa$ with $\lambda\triangleleft\kappa$, such that for any regular cardinal $\mu$ with $\kappa\triangleleft\mu$, the class of exact sequences is contained within the class of $\mu$-filtered colimits of exact sequences between $\mu$-presentable objects. In particular, if $\kappa$-filtered colimits are exact, then for all $\kappa\triangleleft\mu$, the class of exact sequences coincides with the class of $\mu$-filtered colimits of exact sequences between $\mu$-presentable objects.
\end{corollary}

\begin{corollary}
   If $(\mathcal{C},\mathcal{Q})$ category with QE-epi class with $\mathcal{C}$ locally $\lambda$-presentable, is upwardly locally $\lambda$-pre-coherent, and satisfies the strong right obscure axiom, then there are arbitrarily large $\kappa\ge\lambda$ such that $(\mathcal{C},\mathcal{Q})$ is locally $\kappa$-coherent.
\end{corollary}

\subsection{Deconstructibility from coherence}

Now, we use coherence techniques to establish deconstructibility of certain classes of admissible monomorphisms. Let $(\mathcal{C},\mathcal{Q})$ be a purely locally $\lambda$-presentable proto-exact category, and $\mathcal{A}$ a class of objects in $\mathcal{C}$.

\begin{definition}
\begin{enumerate}
\item
Let $\mathcal{S}$ be a class of admissible monomorphisms in $(\mathcal{C},\mathcal{Q})$. We say that a class $\mathcal{A}$ of objects in $\mathcal{C}$ is $\mathcal{S}$-\textit{subobject stable} if whenever $M\rightarrow N$ is a map in $\mathcal{S}$ with $N\in\mathcal{A}$, then $M$ and $N\big\slash M$ are in $\mathcal{A}$.
\item
Let $(\mathcal{C},\mathcal{Q})$ be a purely $\lambda$-accessible proto-exact category. A class of objects $\mathcal{A}$ in $\mathcal{C}$ is said to be $\lambda$-\textit{pure subobject stable} if it is $\mathbf{PureMon}_{\lambda}$-subobject stable.
\end{enumerate}
\end{definition}

The following result can be seen as a generalisation of \cite{Gillespie2} Proposition 4.9. Compare with a similar result, \cite{MR3649814} Lemma 4.14, in the abelian case (the proof uses some similar ideas, but is different).

\begin{lemma}\label{lem:decomp}
    Let $(\mathcal{C},\mathcal{Q})$ be a purely $\lambda$-accessible proto-exact category. Let $\mathcal{A},\mathcal{B}$ be classes of objects in $\mathcal{C}$ such that
    \begin{enumerate}
\item 
$\mathcal{A}$ is $\lambda$-pure subobject stable, and
\item 
$\mathcal{A}$ is $\mathcal{B}$-unionable.
        
    \end{enumerate}Let $G$ be an epimorphic generator that is $\kappa$-presentable, and let $\gamma\ge\mathrm{max}\{\lambda,\kappa\}$ be such that any morphism $X\rightarrow Y$ with $X$ $\gamma$-presentable factors through a $\lambda$-pure monomorphism $\overline{X}\rightarrow Y$, with $\overline{X}$ $\gamma$-presentable. Let $N\rightarrow M\in\mathbf{AdMon}_{\mathcal{A}}$, with $M\in\mathcal{B}$. Then there is a transfinite sequence
    $$M_{\alpha\in A}\rightarrow M$$
    in the overcategory $\mathcal{C}_{\big\slash M}$
    such that:
    \begin{enumerate}
    \item 
    $M_{0}=N$,
    \item 
    for each ordinal $\alpha< A$, $M_{\alpha}\rightarrow M$ is an admissible monomorphism with cokernel in $\mathcal{A}$,
     \item 
     for any $\beta\in A$, $M_{\beta}\rightarrow M_{\beta+1}$ is an admissible monomorphism with cokernel being $\gamma$-presentable and in $\mathcal{A}$,
    \item 
    and $\mathrm{colim}_{\alpha\in A}M_{\alpha}\rightarrow M$ is an isomorphism.
    \end{enumerate}
\end{lemma}

\begin{proof}
    Consider the set $\mathrm{Hom}(G,M)$. Fix a well-order $A$ on $\mathrm{Hom}(G,M)$ 
    with minimal element $0$. 
     For $\alpha\in A$ write $X_{\alpha}=\bigcup_{\beta\le\alpha}\{f_{\beta}\}$ so that $X_{\alpha+1}\setminus X_{\alpha}=\{f_{\alpha+1}\}$ is a singleton. Then $\mathrm{Hom}(G,M)=\bigcup_{\alpha\in A}X_{\alpha}$. We construct each $M_{\alpha}$ by transfinite induction on $A$, with the property that for all $\gamma<\alpha$, $f_{\gamma}:G\rightarrow M$ factors through $M_{\alpha}$. Define $M_{0}=N$. Let $\alpha\in A$ and suppose that for all $\beta<\alpha$, $M_{\beta}\rightarrow M$ has been constructed.
    Suppose $\alpha=\beta+1$ is a successor ordinal. Consider $M\big\slash M_{\beta}$. The map $\overline{f}_{\beta}:G\rightarrow M\rightarrow M\big\slash M_{\beta}$ given by composing $f_{\beta}:G\rightarrow M$ with the quotient map, factors through some $\overline{M}_{\alpha}\rightarrow M\big\slash M_{\beta}$ with $\overline{M}_{\alpha}\rightarrow M\big\slash M_{\beta}$ being a $\lambda$-pure monomorphism and $\overline{M}_{\alpha}$ being $\gamma$-presentable. Now, $M\big\slash M_{\beta}$ is in $\mathcal{A}$, and since $\overline{M}_{\alpha}\rightarrow M\big\slash M_{\beta}$ is $\lambda$-pure, both $\overline{M}_{\alpha}$ and $(M\big\slash M_{\beta})\big\slash\overline{M}_{\alpha}$ are in $\mathcal{A}$. Consider the diagram
\begin{displaymath}
\xymatrix{
M_{\beta}\ar[d]\ar[r] & M_{\alpha}\ar[d]\ar[r] & \overline{M}_{\alpha}\ar[d]\\
M_{\beta}\ar[r] & M\ar[r] & M\big\slash M_{\beta},
}
\end{displaymath}
where $M_{\alpha}$ is defined such that the right-hand square is a pullback, and we are implicitly identifying the kernel of the top-right map with $M_{\beta}$ using Proposition \ref{prop:pullker}. In particular, both rows are exact. Then $M_{\beta}\rightarrow M_{\alpha}$ is an admissible monomorphism with cokernel $\overline{M}_{\alpha}\in\mathcal{A}$. Moreover, $M_{\alpha}\rightarrow M$ is an admissible monomorphism, as it is the pullback of an admissible monomorphism along an admissible epimorphism. Next, the right-hand square is also a pushout by Lemma \ref{lem:pullbacker}, so we have $M\big\slash M_{\alpha}\cong (M\big\slash M_{\beta})\big\slash\overline{M}_{\alpha}$ by Proposition \ref{prop:pullker}. Hence $(M\big\slash M_{\alpha})$ is in $\mathcal{A}$. Note that, by construction, the diagram 
\begin{displaymath}
    \xymatrix{
    G\ar[d]^{f_{\alpha}}\ar[r] & \overline{M}_{\alpha}\ar[d]\\
    M\ar[r] & M\big\slash M_{\beta}
    }
\end{displaymath}
commutes, and therefore $f_{\beta}:G\rightarrow M$ factors through $M_{\alpha}\cong \overline{M}_{\alpha}\times_{M\big\slash M_{\beta}}M$.
When $\alpha$ is a limit ordinal, define $M_{\alpha}=\colim_{\beta<\alpha}M_{\beta}$. By $\mathcal{B}$-unionability of $\mathcal{A}$, and transfinite induction, we find that $M_{\alpha}\rightarrow M$ is in $\mathbf{AdMon}_{\mathcal{A}}$. For the same reason, the map $\colim_{\alpha\in A}M_{\alpha}\rightarrow M$ is an admissible monomorphism with cokernel in $\mathcal{A}$. Finally, 
$$\mathrm{Hom}(G,\colim_{\alpha\in A} M_{\alpha})\rightarrow\mathrm{Hom}(G,M)$$
is an epimorphism. Indeed, let $f\in\mathrm{Hom}(G,M)$. Then $f=f_{\beta}$ for some $\beta$, and $f$ factors through $M_{\beta+1}$.
 Hence $\colim_{\alpha\in A}M_{\alpha}\rightarrow M$ is an epimorphism. It is therefore an isomorphism.
\end{proof}

For a pair of cardinals $(\kappa,\gamma)$, and a class of objects $\mathcal{A}$ in a locally presentable proto-exact category $(\mathcal{C},\mathcal{Q})$, let $\mathbf{I}^{\kappa,\gamma}_{\mathcal{A}}$ denote the class of admissible monomorphisms $i:X\rightarrow Y$ with $Y\in\mathcal{C}_{\kappa}$, and $\mathrm{Coker}(i)\in\mathcal{A}\cap\mathcal{C}_{\gamma}$. Note that, up to isomorphism, there is a set of such morphisms. 

\begin{definition}
Let $(\mathcal{C},\mathcal{Q})$ be a primitive-exact category. Let $\mathcal{A}$ be a class of objects in $\mathcal{C}$. We say that an object $B$ of $\mathcal{C}$ is a \textit{transfinite extension} of objects in $\mathcal{A}$ if there exists a transfinite sequence
$$\emptyset=X_{0}\rightarrow X_{1}\rightarrow X_{2}\rightarrow\ldots\rightarrow X_{\lambda}\rightarrow\ldots\rightarrow B$$
such that 
\begin{enumerate}
    \item 
    the cocone over $B$ is a colimit, and
    \item 
    each $X_{\lambda}\rightarrow X_{\lambda+1}$ is in $\mathcal{M_{Q}}$, with cokernel in $\mathcal{A}$.
\end{enumerate}
We say that $\mathcal{A}$ is \textit{closed under extensions} if any object that is a transfinite extension of objects in $\mathcal{A}$ is still an object of $\mathcal{A}$.
\end{definition}

\begin{proposition}\label{prop:pushoutadmonA}
     Let $(\mathcal{C},\mathcal{Q})$ be a purely $\lambda$-accessible one-sided exact category that is upwardly $(\lambda,\lambda)$-subobject pre-coherent. Let $\mathcal{A}$ be a $\lambda$-pure subobject stable class closed under transfinite extensions and unionable. Then there are arbitrarily large cardinals $\gamma\ge\lambda$ and $\kappa\ge\gamma\ge\lambda$ such that any map in $\mathbf{AdMon}_{\mathcal{A}}$ may be written as a transfinite composition of weak pushouts of maps in $\mathbf{I}^{\kappa,\gamma}_{\mathcal{A}}$. If $(\mathcal{C},\mathcal{Q})$ is tame then, in our construction of $(\kappa,\gamma)$, we have $\mathbf{AdMon}_{\mathcal{A}}=\mathbf{I}^{\kappa,\gamma}_{\mathcal{A}}$-cof, where we use the notation of \cite{hoveybook}*{Section 2.1.2}.
     \end{proposition}

\begin{proof}
Let $A\rightarrow B$ be an arbitrary morphism in $\mathbf{AdMon}_{\mathcal{A}}$. By Lemma \ref{lem:decomp}, there are arbitrarily large $\gamma\ge\lambda$ such that we may write $B\cong\colim_{\alpha<\Gamma}B_{\alpha}$ with $B_{0}=A$, where $B_{\alpha}\rightarrow B_{\alpha+1}\in\mathbf{AdMon}_{\mathcal{A}}$, and each $B_{\alpha+1}\big\slash B_{\alpha}$ is $\gamma$-presentable and in $\mathcal{A}$.

   Let $C\rightarrow D$ be an admissible monomorphism with $D\big\slash C$ in $\mathcal{A}$, and $D\big\slash C\in\mathcal{C}_{\gamma}$. By upward $(\lambda,\lambda)$-sub-object pre-coherence, there are arbitrarily large $\kappa\ge\gamma$ such that there exists
   $D'\rightarrow D$ with $D'\in\mathcal{C}_{\kappa}$ such that $D'\rightarrow D\rightarrow D\big\slash C$ is an admissible epimorphism. Let $C'=\mathrm{Ker}(D'\rightarrow D\big\slash C)$. Then $C'\rightarrow D'$ is in $\mathbf{I}^{\kappa,\gamma}_{\mathcal{A}}$. Note that
   \begin{displaymath}
    \xymatrix{
    C'\ar[d]\ar[r] & C\ar[d]\\
    D'\ar[r] & D
    }
\end{displaymath}
is a weak pushout diagram. Now, let us specialise to the case that $(\mathcal{C},\mathcal{Q})$ is tame. In particular, the diagram above is a pushout diagram. Since $\mathbf{I}^{\kappa,\gamma}_{\mathcal{A}}$-cof is stable by pushouts and contains $\mathbf{I}_{\mathcal{A}}^{\kappa,\gamma}$, we find that $C\rightarrow D$ is in $\mathbf{I}_{\mathcal{A}}^{\kappa,\gamma}$-cof.

 Thus $A\rightarrow B$ is a transfinite extension of maps in $\mathbf{I}_{\mathcal{A}}^{\kappa,\gamma}$-cof, and hence is in  $\mathbf{I}_{\mathcal{A}}^{\kappa,\gamma}$-cof.

Since $(\mathcal{C},\mathcal{Q})$ is weakly $\mathbf{AdMon}_{\mathcal{A}}$-elementary and $\mathcal{A}$ is closed under transfinite extensions, we clearly have $\mathbf{I}_{\mathcal{A}}^{\kappa,\gamma}$, and hence $\mathbf{I}_{\mathcal{A}}^{\kappa,\gamma}$-cell, is contained in $\mathbf{AdMon}_{\mathcal{A}}$. Now, let $f\in\mathbf{I}^{\kappa,\gamma}_{\mathcal{A}}$-cof. By \cite{hoveybook} Corollary 2.1.15, $f$ is a retract of a map in $\mathbf{I}_{\mathcal{A}}^{\kappa,\gamma}$-cell, and hence is in $\mathbf{AdMon}_{\mathcal{A}}$.
\end{proof}

\begin{corollary}\label{cor:deconAcoversenvs}
      Let $(\mathcal{C},\mathcal{Q})$ be a purely locally $\lambda$-presentable, totally proto-exact category that satisfies the obscure axiom and is tame.  Let $\mathcal{A}$ be a $\lambda$-pure subobject stable class closed under transfinite extensions and unionable. Then $\mathcal{A}$ is presentably deconstructible in itself.
      
In particular, for any object $X$ of $\mathcal{C}$ there exists 
      \begin{enumerate}
            \item 
            an $\mathcal{A}^{\perp}$-special preenvelope $X\rightarrow B$ that is an admissible monomorphism, and
            \item
               an $\mathcal{A}$-cover $A\rightarrow X$ that, if $\mathcal{A}$ contains a collection of generators, is an admissible epimorphism.
      \end{enumerate}
\end{corollary}

\begin{proof}
    Deconstructibility of $\mathcal{A}$ follows from Lemma \ref{lem:subepi}, Lemma \ref{lem:decomp}, and Proposition \ref{prop:pushoutadmonA}. The claim regarding covers and envelopes follows from Proposition \ref{pgrop:precoverenv}.
\end{proof}

By applying the previous result to the class $\mathrm{Ob}(\mathcal{C})$ of all objects of $\mathcal{C}$, we get the following.

\begin{corollary}\label{cor:enoghinj}
      Let $(\mathcal{C},\mathcal{Q})$ be a purely locally $\lambda$-presentable, tame, totally proto-exact category that satisfies the obscure axiom and is weakly elementary. Then $\mathcal{C}$ has enough injectives.
\end{corollary}

Let $(\mathcal{C},\mathcal{Q})$ be a purely locally $\lambda$-presentable, tame totally proto-exact category that satisfies the obscure axiom and is weakly elementary. Let $\mathrm{T}:\mathcal{C}\rightarrow\mathcal{C}$ be a monad that commutes with colimits. Since the forgetful functor then commutes with colimits, and also detects exactness, $(\mathrm{Alg_{T}}(\mathcal{C}),\mathcal{Q}_{T})$ is also a purely locally $\lambda$-presentable, totally proto-exact category that satisfies the obscure axiom and is weakly elementary. Again since the forgetful functor commutes with colimits, $(\mathrm{Alg_{T}}(\mathcal{C}),\mathcal{Q}_{T})$  is tame. In particular, we get the following.

\begin{corollary}
     Let $(\mathcal{C},\mathcal{Q})$ be a purely locally $\lambda$-presentable, tame totally proto-exact category that satisfies the obscure axiom and is weakly elementary. Let $\mathrm{T}:\mathcal{C}\rightarrow\mathcal{C}$ be a monad that commutes with colimits. Then $(\mathrm{Alg_{T}}(\mathcal{C}),\mathcal{Q}_{T})$ has enough injectives.
\end{corollary}

Immediately from this, we get the following.

\begin{corollary}
     Let $(\mathcal{C},\mathcal{Q})$ be a purely locally $\lambda$-presentable, tame totally proto-exact category that satisfies the obscure axiom and is weakly elementary. Suppose that $\mathcal{C}$ is equipped with a closed symmetric monoidal structure $(R,\otimes,\underline{\mathrm{Hom}})$. Let $A\in\mathrm{Ass}(\mathcal{C})$ be an associative monoid. Then the category $({}_{A}\mathrm{Mod}(\mathcal{C}),\mathcal{Q}_{A})$ of left $A$-modules has enough injectives.
\end{corollary}

In the additive case, we get the following result. 

\begin{corollary}\label{cor:injectivemodel}
    Let $(\mathpzc{E},\mathcal{Q})$ be a purely locally $\lambda$-presentable exact category in which the class of all objects is unionable. Suppose further that
        the exact category $(\mathpzc{E},\mathcal{Q})$ has exact countable products.
    Then  $\mathrm{Ch}(\mathpzc{E})$ may be equipped with the injective model structure.
\end{corollary}

\begin{proof}
This is an immediate consequence of Corollary \ref{cor:enoghinj} and the dual of \cite{kelly2016homotopy}*{Theorem 4.59}.
\end{proof}

As mentioned in the introduction, this result (and its crucial ingredient, Lemma \ref{lem:decomp} for exact categories) is a very moderate simultaneous generalisation of 
\begin{enumerate}
    \item 
    \cite{positselski2023locally} - by allowing $\kappa$-presentability rather than finite presentability, and
    \item 
    \cite{MR3649814} Lemma 4.14 (allowing exact, rather than just abelian, categories, by allowing exact categories rather than abelian categories.
\end{enumerate}
Unionability is also a mild relaxation of the condition on filtered colimits of monomorphisms with fixed target required in \cite{MR3649814} Lemma 4.14 - however, we believe the proof in loc. cit. only uses unionability.

\subsubsection{Injectives from effective unions}
Let us give another situation where we can prove the existence of enough injectives.
\begin{proposition}\label{prop:subobjectmono}
    Let $(\mathcal{C},\mathcal{Q})$ be a category with class of regular monomorphisms that has $\mathcal{Q}$-effective unions. Suppose further that 
    \begin{enumerate}
        \item 
        the category $\mathcal{C}$ is locally presentable and admits a set of strong generators $\mathcal{G}$, 
        \item 
        the category $\mathcal{C}$ admits a set of objects $\mathcal{H}$ such that any morphism $g:G\rightarrow Y$ with $G\in\mathcal{G}$ factors through an admissible monomorphism $H\rightarrow Y$ with $H\in\mathcal{H}$, and
        \item 
        transfinite compositions of admissible monomorphisms are admissible monomorphisms.
    \end{enumerate}
Then $(\mathcal{C},\mathcal{Q})$ has enough injectives.
\end{proposition}
\begin{proof}
   The proof works essentially identically to \cite{barr2006categories} Proposition 2.1, and \cite{MR1780498} Proposition 1.12. Pick a set of isomorphism classes of admissible subobjects of members of $\mathcal{H}$. Let $I$ be the set of inclusions of such subobjects. We claim that $\mathcal{M}_{\mathcal{Q}}=\mathrm{cell}(I)$. Suppose that $m:X\rightarrow Y$ were a map in $\mathcal{M}_{\mathcal{Q}}$ not in $\mathrm{cell}(I)$. We construct a proper class of objects $P_{\lambda}$, indexed by cardinals, between $X$ and $Y$, with each $X\rightarrow P_{\lambda}$ being in $\mathrm{cell}(I)$. Put $P_{0}=X$. Let $\lambda$ be a cardinal, and suppose that $X\rightarrow P_{\lambda}$ has been constructed. Since $m$ is not in $\mathrm{cell}(I)$, $P_{\lambda}\rightarrow Y$ is not an isomorphism. Therefore, there exists a $G\in\mathcal{G}$, and a map $G\rightarrow Y$ that does not lift to $P_{\lambda}$. Factor $G\rightarrow Y$ through an admissible monomorphism $H\rightarrow Y$, with $H\in\mathcal{H}$. Consider $P_{\lambda+1}=P_{\lambda}\cup_{Y} H$. By construction $G\rightarrow Y$ factors through $P_{\lambda+1}$. Hence, the map $P_{\lambda}\rightarrow P_{\lambda+1}$ cannot be an isomorphism. However this map is the pushout of $P_{\lambda}\cap H\rightarrow H$. In particular, it is in $\mathrm{cell}(I)$. In the limit case, we just take the transfinite extension. 
\end{proof}
\begin{remark}
     The first condition of Proposition \ref{prop:subobjectmono} is satisfied if, for example, $\kappa$-pure monomorphisms are admissible for sufficiently large $\kappa$.
\end{remark}

\subsubsection{Injectives by construction}
Now, we treat the situation that $(\mathcal{C},\mathcal{Q})$ is a primitive-exact category that is neither tame, nor has $\mathcal{Q}$-effective unions.
\begin{construction}\label{construct:smallaprox}
    Let $(\mathcal{C},\mathcal{Q})$ be a left totally primitive-exact category that is locally $\lambda$-presentable for some $\lambda$. Let $\mathcal{N}\subset\mathcal{M}_{\mathcal{Q}}$ be a set of admissible monomorphisms. Let $I_{\mathcal{N}}$ denote the class of objects in $\mathcal{C}$ that have the right lifting property against morphisms in $\mathcal{N}$. 
    \begin{enumerate}
    \item
     We consider the QE-mono class $\mathcal{M_{N}}\defeq\mathcal{M}_{I_{\mathcal{N}}}$ from Definition \ref{defn:forcinginj}. We write $\mathcal{Q}_{\mathcal{N}}\defeq\mathcal{Q}_{\mathcal{M}_{N}}$.
        \item 
           Suppose further that transfinite compositions of pushouts of maps in $\mathcal{N}$ are admissible monomorphisms, and that admissible monomorphisms are closed under retracts. We consider the QE-mono class $\tilde{\mathcal{N}}=\mathrm{llp(rlp(\mathcal{N}))}$ of morphisms with the left lifting property against those morphisms with the right lifting property against $\mathcal{N}$. It is equivalently the class of retracts of transfinite compositions of pushouts of maps in $\mathcal{N}$. 
    \end{enumerate}
\end{construction}
\begin{remark}
\begin{enumerate}
    \item 
        Note that, when  transfinite compositions of pushouts of maps in $\mathcal{N}$ are admissible monomorphisms, and admissible monomorphisms are closed under retracts, we have that $\tilde{\mathcal{N}}$ is a QE-mono class. Indeed, by the small object argument, $\tilde{\mathcal{N}}$ is just the class of retracts of transfinite compositions of pushouts of maps in $\mathcal{N}$. Our assumptions imply immediately that this class is closed under pushouts and composition. Evidently it contains all isomorphisms.
        \item 
        Suppose still that  transfinite compositions of pushouts of maps in $\mathcal{N}$ are admissible monomorphisms, and admissible monomorphisms are closed under retracts. Then $\tilde{\mathcal{N}}\subset\mathcal{M}_{\mathcal{N}}$ by construction. 
\end{enumerate}
\end{remark}
\begin{proposition}\label{prop:smallob}
    Let $(\mathcal{C},\mathcal{Q})$, $\mathcal{N}$, and $\mathcal{M_{N}}$ be as in Construction \ref{construct:smallaprox}. Suppose that transfinite extensions of pushouts of morphisms in $\mathcal{N}$ are admissible monomorphisms. Then the category with QE-mono class $(\mathcal{C},\mathcal{Q}_{\mathcal{N}})$ has enough injectives. When admissible monomorphisms are closed under retracts, then this is also true for $(\mathcal{C},\mathcal{Q}_{\tilde{\mathcal{N}}})$.
\end{proposition}
\begin{proof}
    By the small object argument, any morphism $f:X\rightarrow Y$ in $\mathcal{C}$ factors as $p\circ i$, where $i:X\rightarrow Z$ is a transfinite composition of pushouts of morphisms in $\mathcal{N}$, and $p:Z\rightarrow Y$ has the right lifting property against all morphisms in $\mathcal{N}$. Note that a transfinite composition of pushouts of maps in $\mathcal{N}$ is a map in $\mathcal{M}_{\mathcal{N}}$ since, as a class defined by a left lifting property, $\mathcal{M_{N}}$ is closed under these operations. Moreover, $p$ also has the right lifting property against all morphisms that are transfinite compositions of pushouts of morphisms in $\mathcal{N}$. Fix an object $X$ of $\mathcal{C}$, and factor the morphism $X\rightarrow 0$ as above as $X\rightarrow I$ followed by $I\rightarrow 0$. Now, $X\rightarrow I$ is an admissible monomorphism, and $I\rightarrow 0$ has the right lifting property against transfinite compositions of pushouts of maps in $\mathcal{N}$. In particular, $I\in I_{\mathcal{N}}$. Moreover, any $J\in I_{\mathcal{N}}$ also has the right lifting property against $X\rightarrow I$. Therefore, $X\rightarrow I$ is an admissible monomorphism in $(\mathcal{C},\mathcal{Q_{N}})$. This proof above also works in the $\tilde{\mathcal{N}}$ case.
\end{proof}
\begin{construction}\label{construction:smallapprox}
    Now, let $(\mathcal{C},\mathcal{Q})$ be proto-exact, and let $\gamma$ be sufficiently large such that $(\mathcal{C},\mathcal{Q})$ has a set of $\gamma$-presentable, admissible generators $\mathcal{G}$. For each $\gamma\ge\kappa$, we consider the collection of admissible monomorphisms $\mathbf{I}^{\kappa,\gamma}$, and define the inflation primitive-exact category
$$(\mathcal{C},\mathcal{Q}_{\kappa,\gamma})\defeq(\mathcal{C},\mathcal{Q}_{\mathbf{I}^{\kappa,\gamma}}).$$
When $\kappa=\gamma$, we just write 
$$(\mathcal{C},\mathcal{Q}_{\kappa,\kappa})\defeq(\mathcal{C},\mathcal{Q}_{\kappa}).$$
\end{construction}
\begin{remark}\label{rem:decreasing}
    If $\lambda\ge\gamma\ge\kappa$, then the identity functor $(\mathcal{C},\mathcal{Q}_{\kappa,\gamma})\rightarrow(\mathcal{C},\mathcal{Q}_{\gamma,\lambda})$ is exact. In particular, if $I$ is injective in $(\mathcal{C},\mathcal{Q}_{\gamma,\lambda})$ then it is injective in $(\mathcal{C},\mathcal{Q}_{\kappa,\gamma})$.
\end{remark}
\subsubsection{Injective model structures}
The construction of $(\mathcal{C},\mathcal{Q}_{I})$ is particularly well-behaved in the additive setting of exact categories. 
More generally, we consider weakly idempotent complete exact categories $(\mathpzc{E},\mathcal{Q})$ that have enough injectives. Note that the category $\mathrm{Ch}_{-}(\mathpzc{E},\mathcal{Q})$ of bounded below complexes may be equipped with the injective model structure, by the dual of \cite{kelly2016homotopy}*{Theorem 4.59}.
\begin{proposition}\label{prop:Injchangequill}
  Let $(\mathpzc{E},\mathcal{Q})$ and $(\mathpzc{E},\mathcal{R})$ be finitely complete and finitely cocomplete exact categories with enough injectives, with $\mathcal{Q}\subset\mathcal{R}$. Then there is a Quillen reflection
  $$\adj{\mathrm{Id}}{\mathrm{Ch}_{-}(\mathpzc{E},\mathcal{Q})}{\mathrm{Ch}_{-}(\mathpzc{E},\mathcal{R})}{\mathrm{Id}}$$
 where the left-hand side is equipped with any model structure whose weak equivalences are the quasi-isomorphisms and whose cofibrations are in particular degree-wise admissible monomorphisms, and the right-hand side is equipped with the injective model structure with respect to $(\mathpzc{E},\mathcal{R})$. If the injective model also exists on $\mathrm{Ch}(\mathpzc{E},\mathcal{R})$, then this is also true for unbounded complexes. 
\end{proposition}
\begin{proof}
 The identity functor $(\mathpzc{E},\mathcal{Q})\rightarrow (\mathpzc{E},\mathcal{R})$ is exact. Therefore, the identity functor $\mathrm{Id}:\mathrm{Ch}_{-}(\mathpzc{E},\mathcal{Q})\rightarrow \mathrm{Ch}_{-}(\mathpzc{E},\mathcal{R})$ preserves cofibrations and acyclic cofibrations. Now, let $X_{\bullet}$ be an object of $\mathrm{Ch}_{-}(\mathpzc{E},\mathcal{R})$. Construct an injective resolution $X_{\bullet}\rightarrow I_{\bullet}$. Then $\mathbb{R}\mathrm{Id}(X_{\bullet})\cong I_{\bullet}$, and 
 $$\mathbb{L}\mathrm{Id}\mathbb{R}\mathrm{Id}(X_{\bullet})\cong\mathbb{L}\mathrm{Id}(I_{\bullet})\cong\mathrm{Id}(I_{\bullet})\cong I_{\bullet}.$$
\end{proof}

\subsection{Complexes, homology, and injective resolutions}

Let $\mathcal{C}$ be a finitely complete and finitely cocomplete pointed category. 

\begin{definition}
    A \textit{chain complex in} $\mathcal{C}$ is a sequence $X_{\bullet}$ of objects of $\mathcal{C}$
    \begin{displaymath}
        \xymatrix{
        \ldots X_{n+1}\ar[r]^{d_{n+1}} & X_{n}\ar[r]^{d_{n}} & \ldots
        }
    \end{displaymath}
    indexed by $\mathbb{Z}$, such that $d_{n}\circ d_{n+1}=0$. A \textit{morphism of chain complexes} is given by an obvious commutative diagram. This gives rise to the category $\mathrm{Ch}(\mathcal{C})$ of chain complexes in $\mathcal{C}$.
\end{definition}

We also define the full subcategory $\mathrm{Ch}_{-}(\mathcal{C})\subset \mathrm{Ch}(\mathcal{C})$ to be the one spanned by those complexes $M_{\bullet}$ such $M_{n}=0$ for all sufficiently large $n$. Such a complex is said to be \textit{bounded above}

For a chain complex $X_{\bullet}$ we write $Z_{n}X_{\bullet}=\mathrm{Ker}(d_{n}),$ for each $n\in\mathbb{Z}$.

\begin{definition}
    For a complex $X_{\bullet}\in\mathrm{Ch}(\mathcal{C})$, we define 
    $$H_{n}(X_{\bullet})\defeq\mathrm{Coker}(X_{n+1}\rightarrow\mathrm{Ker}(d_{n}^{X})).$$
\end{definition}

\begin{definition}
    Let $(\mathcal{C},\mathcal{Q})$ be a finitely complete and finitely cocomplete proto-exact category. We say that a complex $X_{\bullet}\in\mathrm{Ch}(\mathcal{C})$ is \textit{acyclic} (relative to $\mathcal{Q}$) if, for each $n\in\mathbb{Z}$, $X_{n+1}\rightarrow\mathrm{Ker}(d_{n})$ is an admissible epimorphism (relative to $\mathcal{Q})$. 
\end{definition}

\begin{remark}
    If $X_{\bullet}\in\mathrm{Ch}(\mathcal{C})$ is acyclic relative to some proto-exact structure, then $H_{n}(X_{\bullet})\cong 0$ for all $n$. The converse is false in general.
\end{remark}

\begin{definition}
       Let $(\mathcal{C},\mathcal{Q})$ be a finitely complete and finitely cocomplete proto-exact category. Let $A\in\mathcal{C}$. By a \textit{right resolution} $\eta:A\rightarrow Q_{\bullet}$ \textit{of }$A$ we mean an acyclic complex
       \begin{displaymath}
          \xymatrix{
          0\ar[r] & A\ar[r]^{\eta} & Q_{0}\ar[r] & Q_{-1}\ar[r] & \ldots
          } 
\end{displaymath}
where $A$ is in degree $1$. 
\end{definition}

The following is clear.

\begin{proposition}
    Suppose that a finitely complete and finitely cocomplete proto-exact category $(\mathcal{C},\mathcal{Q})$ has enough injectives. Then any object $A$ admits a resolution $\eta:A\rightarrow I_{\bullet}$ with each $I_{n}$ being injective.
\end{proposition}

\begin{proposition}[cf, \cite{deitmar2016belian}*{Lemma 1.6.2}]\label{prop:injectivelift}
    Let $(\mathcal{C},\mathcal{Q})$ be a finitely complete and finitely cocomplete proto-exact category, let $\eta:A\rightarrow Q_{\bullet}$ be a resolution, let $f:A\rightarrow B$ be a morphism, and let 
    \begin{displaymath}
          \xymatrix{
          0\ar[r] & B\ar[r]^{\epsilon} & I_{0}\ar[r] & I_{1}\ar[r] & \ldots
          } 
\end{displaymath}
be a complex with $B$ in degree $1$, and each $I_{j}$ being injective. Then there exists a commutative diagram
    \begin{displaymath}
          \xymatrix{
           0\ar[r] & A\ar[r]^{\eta}\ar[d]^{f} & Q_{0}\ar[r]\ar[d]^{f_{0}} & Q_{1}\ar[r]\ar[d]^{f_{1}} & \ldots\\
          0\ar[r] & B\ar[r]^{\epsilon} & I_{0}\ar[r] & I_{1}\ar[r] & \ldots.
          } 
\end{displaymath}
\end{proposition}

\begin{proof}
Since all that is required is the lifting property of injectives, and no morphisms need to be summed, this works identically to the additive case (see e.g., the dual of \cite{Weibel} Theorem 2.2.6).
\end{proof}

\section{Seminormed, normed, and Banach monoids, and their modules}\label{sec:bancoh}

We give our main examples, the categories of seminormed, normed, and Banach modules. For more on seminormed rings, one can consult \cite{dang}, for example.

\begin{definition}
A \textit{seminormed ring} is a unital ring $R$ equipped with a norm $|-|:R\rightarrow\mathbb{R}_{\ge0}$ such that for all $r,s\in R$

\begin{enumerate}
\item $|0|=0$,
\item $|-1|\le 1$,
    \item 
    $|rs|\le |r| |s|,$ and
    \item 
    $|r+s|\le |r|+|s|.$
\end{enumerate}
We say that $R$ is \textit{normed} if $|r|=0\Leftrightarrow r=0$, and \textit{Banach} if it is complete with respect to the metric induced by the norm. Finally, we say that $R$ is \textit{non-Archimedean} if $|r+s|\le\textrm{max}\{|r|,|s|\}$. A seminormed, normed, or Banach ring $R$ is said to be \textit{commutative} if its underlying ring is commutative.
\end{definition}

\begin{definition}
Let $R$ be a seminormed ring. 
\begin{enumerate}
    \item 
A \textit{seminormed right} $R$-\textit{module} is a right $R$-module $M$ equipped with a function $\rho_{M}:M\rightarrow\mathbb{R}_{\ge0}$ such that for all $r\in R$ and $m,n\in M$, 

\begin{enumerate}
\item 
    $\rho_{M}(mr)\le |r| \rho_{M}(m),$
    \item 
    $\rho_{M}(m+n)\le \rho_{M}(m)+\rho_{M}(n).$
\end{enumerate}
 A seminormed $R$-module $M$ is said to be \textit{normed} if $\rho_{M}(m)=0\Leftrightarrow m=0$, and Banach if it is complete with respect to the metric induced by $\rho_{M}$. We call $M$ \textit{non-Archimedean} if $\rho_{M}(m+n)\le\mathrm{max}\{\rho_{M}(m),\rho_{M}(n)\}$ for all $m,n\in M$.
 \item 
 By a \textit{bounded morphism of seminormed} $R$-modules, we mean a morphism of $R$-modules $f:M\rightarrow N$ such that there is $C_{f}>0$ with $\rho_{N}(f(m))\le C_{f}\rho_{M}(m)$ for all $m\in M$. Such a morphism is said to be \textit{non-expanding} if $C_{f}$ can be taken to be $1$.
 \end{enumerate}
  We denote by $\mathrm{SNrm}_{R}$ the category of seminormed right $R$-modules and bounded morphisms, and $\mathrm{SNrm}_{R}^{\le1}$ the category of seminormed $R$-modules with non-expanding morphisms. We denote by $\mathrm{Nrm}_{R}$ (resp. $\mathrm{Nrm}_{R}^{\le1}$) and $\mathrm{Ban}_{R}$ (resp. $\mathrm{Ban}_{R}^{\le1}$) the full subcategories consisting of normed and Banach $R$-modules, respectively.
\end{definition}

\begin{remark}
    If $R$ is a non-zero non-Archimedean seminormed, normed, or Banach ring, then the corresponding categories of non-Archimedean modules over it are non-zero. If $R=\mathbb{Z}$ is the Banach ring of integers equipped with the restriction of the Euclidean norm, then a non-Archimedean seminormed module over $\mathbb{Z}$ is precisely a non-Archimedean seminormed abelian group. This includes any abelian group equipped with the trivial norm, whereby any non-zero element has norm $1$. On the other hand, suppose that $R=\mathbb{R}$. If $V$ is a non-Archimedean seminormed $R$-module, then for any $v\in V$ and any integer $n$, we have 
    $$\frac{|n|}{C}\rho_{V}(v)\le \rho_{V}(nv)\le \rho_{V}(v).$$
    This can only hold for all $n$ if $\rho_{V}(v)=0$. Hence, the only non-Archimedean $\mathbb{R}$-modules are those for which the seminorms of all elements are $0$.
\end{remark}

 The natural inclusion $\mathrm{Nrm}_{R}\rightarrow\mathrm{SNrm}_{R}$ has a left adjoint $\mathrm{Sep}$, which sends $M$ to $M\big\slash\{m:\rho_{M}(m)=0\}$. The inclusion $\mathrm{Ban}_{R}\rightarrow\mathrm{Nrm}_{R}$ also has a left adjoint $\mathrm{Cpl}$, given by completion. We denote 
$$\overline{\mathrm{Cpl}}\defeq\mathrm{Cpl}\circ\mathrm{Sep}:\mathrm{SNrm}_{R}\rightarrow\mathrm{Ban}_{R}.$$

We further have the full subcategories
$$\mathrm{SNrm}_{R}^{nA},\mathrm{Nrm}_{R}^{nA},\mathrm{Ban}_{R}^{nA}$$
of $\mathrm{SNrm}_{R}$, and the full subcategories
$$\mathrm{SNrm}_{R}^{nA,\le1},\mathrm{Nrm}_{R}^{nA,\le1},\mathrm{Ban}_{R}^{nA,\le1}$$
of $\mathrm{SNrm}_{R}^{\le1},$
consisting of non-Archimedean modules. 

The functors $\mathrm{Sep}$,  $\mathrm{Cpl}$ and $\overline{\mathrm{Cpl}}$ all restrict to left adjoints on the non-expanding categories, and all of the non-Archimedean variants.

When our discussions apply to the seminormed, normed, and Banach categories at once, and we do not wish to repeat ourselves, we will write $\mathrm{XNrm}_{R}$, $\mathrm{XNrm}_{R}^{\le1},\mathrm{XNrm}_{R}^{nA},\mathrm{XNrm}_{R}^{nA,\le1}$.

Recall that if $M$ is a seminormed module, with subspace $N$, then the quotient seminorm on $M\big\slash N$ is $\rho_{M\big\slash N}([m])=\mathrm{inf}_{n\in N}\rho_{M}(m+n)$.

\begin{remark}
    Note that the categories $\mathrm{XNrm}_{R}$ and $\mathrm{XNrm}_{R}^{nA}$ are additive. On the other hand,$\mathrm{XNrm}_{R}^{\le1}$ will not be, simply because the sum of two maps of norm $\le1$ need not be of norm $\le1$. However, the categories $\mathrm{XNrm}_{R}^{nA,\le1}$ \textit{are additive}. In fact, as we shall see below, they are \textit{quasi-abelian}.
\end{remark}

The functor of \textit{rescaling} is useful.
For $M$ an object of $\mathrm{XNrm}_{R}$ and $\delta\in\mathbb{R}_{>0}$ a positive real number, we denote by $M_{\delta}$ the object with the same underlying module as $M$, but with $\rho_{M_{\delta}}(m)=\delta\rho_{M}(m)$. Note that for $\epsilon,\delta>0$ the identity on the underlying module
$$M_{\epsilon}\rightarrow M_{\delta}$$
is a map of norm $\frac{\delta}{\epsilon}$. In $\mathrm{XNrm}_{R}$ it is an isomorphism. However, it is only defined in $\mathrm{XNrm}_{R}^{\le1}$ when $\epsilon\ge\delta$. Moreover, it is only an isomorphism in $\mathrm{XNrm}_{R}^{\le1}$ when either $\epsilon=\delta$ or every element of $M$ has seminorm equal to $0$.

\begin{remark}\label{rem:rescalingprops}
    The construction $M\mapsto M_{\delta}$ is functorial on both $\mathrm{XNrm}_{R}^{\le1}$ and $\mathrm{XNrm}$, with inverse $N\mapsto N_{\frac{1}{\delta}}$. Note that we then have 
$$\mathrm{Hom}(M_{\delta},N)\cong\mathrm{Hom}(M,N_{\frac{1}{\delta}}).$$
In fact $\mathrm{Hom}(M_{\delta},N)$ is the set of bounded maps $M\rightarrow N$ of seminorm at most $\delta$. In particular, 
$$\mathrm{Hom}(R_{\delta},M)=\overline{B}_{M}(0,\delta)$$
is the closed ball in $M$ of radius $\delta$. 
\end{remark}

\subsection{Categorical properties}

Let us establish some properties of the categories we have introduced. Some of the results below also appear in \cite{aoki2024rosenberg} Section 2 - specifically, that $\mathrm{Ban}_{\mathbb{Z}}^{\le1}$ is a closed symmetric monoidal, $\aleph_{1}$-presentable category. Some of the non-Archimedean results - specifically identifying categorical generators - have also appeared in \cite{kelly2016homotopy}, Section 3.1.3.

The following is standard over $\mathbb{R},\mathbb{C}$, and a non-Archimedean non-trivially valued field. The usual proof generalises to arbitrary Banach rings, and we record it here, again for the avoidance of doubt.

\begin{proposition}\label{prop:quotBanclosed}
    If $M$ is normed (resp. Banach) and $N\subset M$ is a submodule, then we have $\mathrm{Sep}(M\big\slash N)\cong M\big\slash\overline{N}$ (resp. $\overline{\mathrm{Cpl}}(M\big\slash N)\cong M\big\slash\overline{N}$), where $\overline{N}$ denotes the closure.
\end{proposition}

\begin{proof}

        First, let $M$ be a normed $R$-module, and $N\subset M$ a closed submodule. Let $[m]\in M\big\slash N$ be such that $\rho_{M\big\slash N}([m])=\mathrm{inf}_{n\in N}\rho_{M}(m-n)=0$. Then there exists a sequence $(n_{k})_{k\in\mathbb{N}_{0}}$ in $N$ such that $\lim_{k\rightarrow\infty}\rho_{M}(m-n_{k})=0$. In other words, $n_{k}$ converges to $m$, and $\{[m]:\rho_{M\big\slash N}([m])=0\}=\{[m]:m\in\overline{N}\}$. Thus $\mathrm{Sep}(M\big\slash N)\cong M\big\slash\overline{N}.$
    
    Let $M$ be a Banach $R$-module, and $N\subset M$ a closed subspace. The same proof as \cite{319920} Proposition 11.8 shows that $M\big\slash N$ is already Banach. Returning to the case of arbitrary submodules $N$, we see that $\mathrm{Sep}(M\big\slash N)\cong M\big\slash\overline{N}$ is already complete, and is therefore the completion of $M\big\slash N$.

\end{proof}

For the case of a Banach field, most of the statements in the following Lemma are already in \cite{koren}*{Appendix A}, and the proofs are the same.

\begin{lemma}\label{lem:catpropnon}
\begin{enumerate}
    \item 
    Let $f,g:X\rightarrow Y$ be morphisms in the following categories: $\mathrm{SNrm}_{R}$, $\mathrm{SNrm}_{R}^{\le1}$, $\mathrm{SNrm}_{R}^{nA}$, or $\mathrm{SNrm}_{R}^{nA,\le1}$. Then
    \begin{enumerate}
        \item 
        $\mathrm{Eq}(f,g)$ is the subspace of $X$ consisting of those $x\in X$ such that $f(x)=g(x)$, equipped with the subspace seminorm;
        \item
        $\mathrm{Coeq}(f,g)$ is $Y\big\slash\{f(x)-g(x)\}$ equipped with the quotient seminorm.
    \end{enumerate}
\item
     Let $f,g:X\rightarrow Y$ be morphisms in any of the following categories: $\mathrm{Nrm}_{R}$, $\mathrm{Nrm}_{R}^{\le1}$, $\mathrm{Nrm}_{R}^{nA}$, $\mathrm{Ban}_{R}$, $\mathrm{Ban}_{R}^{\le1}$, $\mathrm{Ban}_{R}^{nA}$, or $\mathrm{Ban}_{R}^{nA,\le1}$. Then
        \begin{enumerate}
        \item 
        $\mathrm{Eq}(f,g)$ is the subspace of $X$ consisting of those $x\in X$ such that $f(x)=g(x)$, equipped with the subspace seminorm;
        \item
         $\mathrm{Coeq}(f)$ is $Y\big\slash\overline{\{f(x)-g(x)\}}$ equipped with the quotient norm, where $\overline{\{f(x)-g(x)\}}$ denotes the closure.
    \end{enumerate}
    \item
   If $\{(X_{i},\rho_{i})\}_{i\in\mathcal{I}}$ is a collection of objects of $\mathrm{SNrm}_{R}^{\le1}$, $\mathrm{Nrm}_{R}^{\le1}$, $\mathrm{Ban}_{R}^{\le1}$, $\mathrm{SNrm}_{R}^{nA,\le1}$, $\mathrm{Nrm}_{R}^{nA,\le1}$, or $\mathrm{Ban}_{R}^{nA,\le1}$, then the product,
   $$\prod_{i\in\mathcal{I}}^{\le1}(X_{i},\rho_{i}),$$
   is the seminormed, normed, or Banach module whose underlying module is the sub-module of the algebraic product consisting of those tuples $(x_{i})_{i\in\mathcal{I}}$ such that $\mathrm{sup}_{i\in\mathcal{I}}\rho_{i}(x_{i})<\infty$, with norm given by 
   $$\rho_{\prod}((x_{i})_{i\in\mathcal{I}})=\mathrm{sup}_{i\in\mathcal{I}}\rho_{i}(x_{i}).$$
   If the collection $\mathcal{I}$ is finite, then this is also the product in $\mathrm{SNrm}_{R}$, $\mathrm{Nrm}_{R}$, $\mathrm{Ban}_{R}$, $\mathrm{SNrm}_{R}^{nA}$, $\mathrm{Nrm}_{R}^{nA}$, or $\mathrm{Ban}_{R}^{nA}$.
   \item 
    If $\{(X_{i},\rho_{i})\}$ is a collection of objects of $\mathrm{SNrm}_{R}^{\le1}$ or $\mathrm{Nrm}_{R}^{\le1}$, then $\coprod_{i\in\mathcal{I}}^{\le1}(X_{i},\rho_{i})$ is the module whose underlying module is the algebraic coproduct, and with norm given by 
    $$\rho_{\coprod}((x_{i})_{i\in\mathcal{I}})=\sum_{i\in\mathcal{I}}\rho_{i}(x_{i}).$$
     If the collection $\mathcal{I}$ is finite, then this is also the coproduct in $\mathrm{SNrm}_{R}$, $\mathrm{Nrm}_{R}$, or $\mathrm{Ban}_{R}$.
   \item 
    If $\{(X_{i},\rho_{i})\}$ is a collection of objects of $\mathrm{SNrm}_{R}^{nA,\le1}$ or $\mathrm{Nrm}_{R}^{nA,\le1}$, then $\coprod_{i\in\mathcal{I}}^{nA,\le1}X_{i}$ is the module whose underlying module is the algebraic coproduct, and with norm given by 
    $$\rho_{\coprod}((x_{i})_{i\in\mathcal{I}})=\mathrm{max}_{i\in\mathcal{I}}\rho_{i}(x_{i}).$$
     If the collection $\mathcal{I}$ is finite, then this is also the coproduct in $\mathrm{SNrm}_{R}^{nA}$, $\mathrm{Nrm}_{R}^{nA}$, or $\mathrm{Ban}_{R}^{nA}$.
    \item 
    If $\{(X_{i},\rho_{i})\}$ is a collection of objects of $\mathrm{Ban}_{R}^{\le1}$ or $\mathrm{Ban}_{R}^{nA,\le1}$, then the coproduct is given by the completion of the coproduct as normed modules. 
   \end{enumerate}
\end{lemma}

\begin{proof}
\begin{enumerate}
    \item 
    \begin{enumerate}
    \item 
        Denote by $\widetilde{\mathrm{Eq}}(f,g)$ the subspace of $X$ consisting of those $x\in X$ such that $f(x)=g(x)$ equipped with the subspace seminorm. Let $h:Z\rightarrow X$ be a map such that $f\circ h=g\circ h$. Then clearly $i$ factors through the inclusion $\widetilde{\mathrm{Eq}}(f,g)\rightarrow X$. Since this latter map is a monomorphism, the factorisation must be unique.
    \item 
       Denote by $\widetilde{\mathrm{Coeq}}(f,g)$ the space $Y\big\slash\{f(x)-g(x)\}$ equipped with the quotient seminorm. Let $h:Y\rightarrow Z$ be a morphism such that $h\circ f=h\circ g$. We get a unique map of $R$-modules $\tilde{h}:\widetilde{\mathrm{Coeq}}(f,g)\rightarrow Z$ such that $\tilde{h}\circ q=h$, where $q:Y\rightarrow\widetilde{\mathrm{Coeq}}(f,g)$ is the quotient map. This map is easily seen to be bounded. 
    \end{enumerate}
    \item 
    Both $\mathrm{Nrm}_{R}$ and $\mathrm{Ban}_{R}$ (and their variants) are reflective subcategories of $\mathrm{SNrm}_{R}$. Thus, they are closed under limits, and this gives the claim for equalisers. For coequalisers in $\mathrm{Nrm}_{R}$, we have 
    $$\mathrm{Coeq}(f,g)\cong\mathrm{Sep}(\mathrm{Coeq}_{\mathrm{SNrm}}(f,g)),$$
    where $\mathrm{Coeq}_{\mathrm{SNrm}}(f,g)$ denotes the coequaliser in the seminormed category. In the Banach category, we have 
 $$\mathrm{Coeq}(f,g)\cong\overline{\mathrm{Cpl}}(\mathrm{Coeq}_{\mathrm{SNrm}}(f,g)).$$
   The claim now follows from Proposition \ref{prop:quotBanclosed}.

    \item 
    Let $\{f_{i}:(Y,\rho)\rightarrow (X_{i},\rho_{i})\}_{i\in\mathcal{I}}$ be a collection of non-expanding morphisms of $R$-modules. There is a unique map of modules to the product 
    $$f:Y\rightarrow\prod_{i\in\mathcal{I}}X_{i},\; y\mapsto (f_{i}(y))_{i\in\mathcal{I}}.$$
    This is obviously non-expanding for the supremum seminorm.
    \item 
    Let $\{f_{i}:(X_{i},\rho_{i})\rightarrow (Y,\rho)\}$ be a collection of non-expanding morphisms of $R$-modules. We get a unique map of $R$-modules 
    $$f:\coprod_{i\in\mathcal{I}}X_{i}\rightarrow Y,\; (x_{i})_{i\in\mathcal{I}}\mapsto\sum_{i\in\mathcal{I}}f_{i}(x_{i}).$$
    For the sum norm this is evidently non-expanding.
    \item 
       Let $\{f_{i}:(X_{i},\rho_{i})\rightarrow (Y,\rho)\}$ be a collection of non-expanding morphisms of $R$-modules. We get a unique map of $R$-modules 
    $$f:\coprod_{i\in\mathcal{I}}X_{i}\rightarrow Y,\; (x_{i})_{i\in\mathcal{I}}\mapsto\sum_{i\in\mathcal{I}}f_{i}(x_{i}).$$
    Since $Y$ is non-Archimedean, for the supremum norm this is evidently non-expanding.
    \item 
    This again follows from the fact that the non-expanding categories of Banach spaces are reflective subcategories of the respective normed categories.
\end{enumerate}
\end{proof}

\begin{corollary}
    In the following categories:
    $$\mathrm{SNrm}_{R},\mathrm{Nrm}_{R},\mathrm{Ban}_{R},\mathrm{SNrm}_{R}^{nA}, \mathrm{Nrm}_{R}^{nA}, \mathrm{Ban}_{R}^{nA},\mathrm{SNrm}_{R}^{nA,\le1}, \mathrm{Nrm}_{R}^{nA,\le1}, \mathrm{Ban}_{R}^{nA,\le1},$$ 
    the natural map
    $$X\coprod Y\rightarrow X\prod Y$$
    is an isomorphism for any pair of objects $X$ and $Y$. In particular, these are all additive categories.
\end{corollary}

\begin{proof}
    In the non-Archimedean cases this is immediate, as the norms on either side are equal. In the Archimedean cases it follows from the fact that for $a,b\ge0$
    $$\frac{1}{2}(a+b)\le \mathrm{max}\{a,b\}\le a+b.$$
\end{proof}

Even when working in the additive categories $\mathrm{XNrm}_{R}$, $X\oplus Y$ will always denote the product with the sum norm, and in $\mathrm{XNrm}^{nA}_{R}$ it will denote the product with the max norm. In both cases $X\prod Y$ will denote the product with the max norm.

\subsubsection{The quasi-abelian and parabelian properties}
 
The following is an immediate consequence of Proposition \ref{prop:strictnessequiv} and Lemma \ref{lem:catpropnon}.

\begin{proposition}
\begin{enumerate}
    \item 
        A morphism $f:X\rightarrow Y$ in $$\mathrm{SNrm}^{\le1}_{R},\mathrm{Nrm}^{\le1}_{R},\mathrm{Ban}^{\le1}_{R},\mathrm{SNrm}^{nA,\le1}_{R},\mathrm{Nrm}^{nA,\le1}_{R},\mathrm{Ban}^{nA,\le1}_{R}$$ is a strict epimorphism if and only if it is surjective, and $\rho_{Y}(y)=\mathrm{inf}_{x\in f^{-1}(y)}\rho_{X}(x)$ for all $y\in Y$.
        \item 
            A morphism $f:X\rightarrow Y$ in $\mathrm{SNrm}^{\le1}_{R}$ or $\mathrm{SNrm}^{nA,\le1}_{R}$ is a strict monomorphism if and only if it is an isometry.
            \item 
            A morphism $f:X\rightarrow Y$ in $\mathrm{Nrm}^{\le1}_{R},\mathrm{Ban}^{\le1}_{R},\mathrm{Nrm}^{nA,\le1}_{R}$, or $\mathrm{Ban}^{nA,\le1}_{R}$ is a strict monomorphism if and only if it is an isometry with closed image.
            \item 
             A morphism $f:X\rightarrow Y$ in $\mathrm{SNrm}_{R},\mathrm{Nrm}_{R},\mathrm{Ban}_{R},\mathrm{SNrm}^{nA}_{R},\mathrm{Nrm}^{nA}_{R}$, or $\mathrm{Ban}^{nA}_{R}$ is a strict epimorphism if and only if it is surjective, and $\rho_{Y}$ is equivalent to the seminorm $y\mapsto\mathrm{inf}_{x\in f^{-1}(y)}\rho_{X}(x)$ for all $y\in Y$.
              \item 
            A morphism $f:X\rightarrow Y$ in $\mathrm{SNrm}_{R}$ or $\mathrm{SNrm}^{nA}_{R}$ is a strict monomorphism if and only if the induced map $f:X\rightarrow\mathrm{Im}(f)$ is an isomorphism, where $\mathrm{Im}(f)\subset Y$ is the set-theoretic image of $f$, equipped with the sub-module seminorm.
            \item 
            A morphism $f:X\rightarrow Y$ in $\mathrm{Nrm}_{R},\mathrm{Ban}_{R},\mathrm{Nrm}^{nA}_{R}$, or $\mathrm{Ban}^{nA}_{R}$ is a strict monomorphism if and only if the set-theoretic image $\mathrm{Im}(f)$ is closed in $Y$, and the induced map $f:X\rightarrow\mathrm{Im}(f)$ is an isomorphism, where $\mathrm{Im}(f)\subset Y$ is the set-theoretic image of $f$, equipped with the sub-module norm.
\end{enumerate}
\end{proposition}

\begin{proposition}\label{prop:isosio}
    Let 
\begin{displaymath}
    \xymatrix{
    0\ar[r] & X\ar[r]^{f} & Y\ar[r]^{g} & Z\ar[r] & 0
    }
\end{displaymath}
be a kernel-cokernel pair in any of the categories
$$\mathrm{SNrm}_{R},\mathrm{Nrm}_{R},\mathrm{Ban}_{R},\mathrm{SNrm}^{nA}_{R},\mathrm{Nrm}^{nA}_{R},\mathrm{Ban}^{nA}_{R}.$$
Then there is an isomorphism of exact sequences
\begin{displaymath}
    \xymatrix{
    0\ar[r] & X\ar[r]^{f}\ar[d] & Y\ar[r]^{g}\ar@{=}[d] & Z\ar[d]\ar[r] & 0\\
    0\ar[r] & \mathrm{Im}(f)\ar[r]^{i_{f}} & Y\ar[r]^{q_{f}} & Y\big\slash\mathrm{Im}(f)\ar[r] & 0
    }
\end{displaymath}
in which the bottom-row is a kernel-cokernel pair in
$$\mathrm{SNrm}^{\le1}_{R},\mathrm{Nrm}^{\le1}_{R},\mathrm{Ban}^{\le1}_{R},\mathrm{SNrm}^{nA,\le1}_{R},\mathrm{Nrm}^{nA,\le1}_{R},\mathrm{Ban}^{nA,\le1}_{R}.$$
\end{proposition}

\begin{proof}
    Since $f$ is a strict monomorphism, we have that $X\rightarrow\mathrm{Im}(f)$ is an isometry (with closed image in the normed and Banach cases), and $Z\cong\mathrm{Coker}(f)\cong Y\big\slash\mathrm{Im}(f)$. The map $\mathrm{Im}(f)\rightarrow Y$  is an isometry (with closed image in the normed and Banach cases), and $Y\big\slash\mathrm{Im}(f)$ has the quotient norm.
\end{proof}

\begin{theorem}\label{cor:pushpullnonexp}
\begin{enumerate}
    \item 
    Let 
      \begin{displaymath}
        \xymatrix{
         F\ar[d]^{f'}\ar[r]^{g'}& M\ar[d]^{f}\\
         L\ar[r]^{g} & N
        }
    \end{displaymath}
    be a fibre-product diagram in $\mathrm{SNrm}^{\le1}_{R},\mathrm{Nrm}^{\le1}_{R},\mathrm{Ban}^{\le1}_{R},\mathrm{SNrm}^{nA,\le1}_{R},\mathrm{Nrm}^{nA,\le1}_{R}$, or $\mathrm{Ban}^{nA,\le1}_{R}$.
    If $f$ is a strict epimorphism then so is $f'$.
    \item 
Let
 \begin{displaymath}
        \xymatrix{
         K\ar[d]^{g}\ar[r]^{i}&M\ar[d]^{g'}\\
         L\ar[r]^{i'}& P 
        }
    \end{displaymath}
    be a pushout diagram in $\mathrm{SNrm}^{\le1}_{R},\mathrm{Nrm}^{\le1}_{R},\mathrm{Ban}^{\le1}_{R},\mathrm{SNrm}^{nA,\le1}_{R},\mathrm{Nrm}^{nA,\le1}_{R}$, or $\mathrm{Ban}^{nA,\le1}_{R}$. If $i$ is a strict monomorphism then so is $i'$.
    \item 
    Let $g$ and $f$ be a strict epimorphisms in $\mathrm{SNrm}^{\le1}_{R},\mathrm{Nrm}^{\le1}_{R},\mathrm{Ban}^{\le1}_{R},\mathrm{SNrm}^{nA,\le1}_{R},\mathrm{Nrm}^{nA,\le1}_{R}$, or $\mathrm{Ban}^{nA,\le1}_{R}$. Then $g\circ f$ is a strict epimorphism.
        \item 
    Let $g$ and $f$ be strict monomorphisms in $\mathrm{SNrm}^{\le1}_{R},\mathrm{Nrm}^{\le1}_{R},\mathrm{Ban}^{\le1}_{R},\mathrm{SNrm}^{nA,\le1}_{R},\mathrm{Nrm}^{nA,\le1}_{R}$, or $\mathrm{Ban}^{nA,\le1}_{R}$. Then $g\circ f$ is a strict monomorphism.
    \item 
        Let $g$ and $f$ be morphisms in $\mathrm{SNrm}^{\le1}_{R},\mathrm{Nrm}^{\le1}_{R},\mathrm{Ban}^{\le1}_{R},\mathrm{SNrm}^{nA,\le1}_{R},\mathrm{Nrm}^{nA,\le1}_{R}$, or $\mathrm{Ban}^{nA,\le1}_{R}$. If $g\circ f$ is a strict epimorphism, then $g$ is a strict epimorphism.
                \item 
    Let $g$ and $f$ be morphisms in $\mathrm{SNrm}^{\le1}_{R},\mathrm{Nrm}^{\le1}_{R},\mathrm{Ban}^{\le1}_{R},\mathrm{SNrm}^{nA,\le1}_{R},\mathrm{Nrm}^{nA,\le1}_{R}$, or $\mathrm{Ban}^{nA,\le1}_{R}$. If $g\circ f$ is a strict monomorphism, then $f$ is a strict monomorphism.
    \end{enumerate}
Consequently, all the listed categories are totally parabelian and satisfy the obscure axiom. The non-Archimedean versions are in fact quasi-abelian.
\end{theorem}

\begin{proof}
\begin{enumerate}
   \item
    A model for the fibre product is given by $M\times_{N}L=\{(m,l)\in M\times L:f(m)=g(l)\}$. Then $f'$ and $g'$ are the respective projections, and are clearly non-expanding. We have $\rho_{N}(f(m))=\mathrm{inf}_{k\in\mathrm{Ker}(f)}\rho_{M}(m+k)$.

    The kernel of $f'$ is $\{(m,0):m\in\mathrm{Ker}(f)\}$. We have $\rho_{M\times_{N}L}(m,l)=\mathrm{max}\{\rho_{M}(m),\rho_{L}(l)\}$. For any $m$ such that $f(m)=g(l)$, which exists since $f$ is surjective, the element $(m,l)$ maps to $l$. Then $\mathrm{max}\{\rho_{M}(m),\rho_{L}(l)\}\ge\rho_{L}(l)$. 
    
    Thus $\mathrm{inf}\{\rho_{M\times L}(m,l):f'(m,l)=l\}\ge\rho_{L}(l)$. Fix any $(m,l)$ with $f(m)=g(l)$. Now, $\rho_{N}(f(m))=\rho_{N}(g(l))\le \rho_{L}(l)$. Since $f$ is a strict epimorphism, for any $\epsilon>0$, there is $k_{\epsilon}\in\mathrm{Ker}(f)$ such that $\rho_{M}(m+k_{\epsilon})\le\rho_{L}(l)+\epsilon$. Then $\mathrm{max}\{\rho_{M}(m+k_{\epsilon}),\rho_{L}(l)\}\le \rho_{L}(l)+\epsilon$, and $f'(m+k_{\epsilon},l)=l$. Hence $\mathrm{inf}\{\rho_{M\times L}(m,l):f'(m,l)=l\}\le\rho_{L}(l)$. Consequently, $\rho_{L}$ in fact coincides with the quotient seminorm for the map $\mathrm{Ker}(g)\rightarrow M\times_{N}L$. In particular, $g'$ is a strict epimorphism. 

\item
    The map from a seminormed module to its separation/ completion is non-expanding, so we may work in the seminormed categories. 
    The pushout $P$ has a model given by $M\oplus N\big\slash\{(f(x),-g(x))\}$. The map $i'$ is the composition 
    $$L\rightarrow M\oplus L\rightarrow M\oplus L\big\slash\{(i(x),-g(x))\}$$
    which is non-expanding, since the inclusion $L\rightarrow L\oplus M$ is non-expanding, and the quotient map is non-expanding. This is similarly true for $g'$. 

    Now, let $i:K\rightarrow M$ be a strict monomorphism, and let $g:K\rightarrow L$ be any map. Recall that $i$ is then an isometry. We have that $L\oplus_{K}M$ is the quotient of $M\oplus L$ by $\{(i(k),g(k)):k\in K\}$. First observe that if the image of $i$ in $M$ is closed, then the image of the map $k\mapsto(i(k),g(k))$ is closed. Indeed, if $(i(k_{n}),g(k_{n}))$ is a sequence converging to some $(m,l)$, then $i(k_{n})$ converges to $m$. But the image of $i$ is closed in $M$, so $k_{n}$ converges to some $k$ in $K$. Thus $(i(k_{n}),g(k_{n}))$ converges to $(i(k),g(k))$.
In particular, in the normed and Banach categories we do not need to take the closure. Now, we need to show that $L\rightarrow M\oplus L\big\slash\{(i(k),g(k)):k\in K\}, l\mapsto[(0,l)]$ is a strict monomorphism. First, it is a monomorphism since the underlying module is the usual pushout in the category of modules. It suffices to show that it is an isometry onto its image. First let us prove the Archimedean version. Consider an element of the form $[(0,l)]$.  If $[(0,l)]=[(m,l')]$ then $(m,l'-l)=(i(k),g(k))$. Thus $m=i(k)$ and $\rho_{M}(m)=\rho_{K}(k)$. Then $\rho_{L}(l'-l)\le \rho_{K}(k)$ so $\rho_{L}(l)\le\rho_{L}(l')+\rho_{K}(k)=\rho_{L}(l')+\rho_{M}(m)$. Hence $\rho_{L}(l)=\rho([(0,l)])$ and we are done.

Now, let us prove the non-Archimedean version. Again if $[(0,l)]=[(m,l')]$ then $(m,l'-l)=(i(k),g(k))$. Thus $m=i(k)$ and $\rho_{M}(m)=\rho_{K}(k)$.  Suppose that $\rho_{L}(l)>\rho_{L}(l')$. Then $\rho_{L}(l)=\rho_{L}(l-l')\le\rho_{K}(k)=\rho_{M}(m)$. Hence, either $\rho_{L}(l')\ge\rho_{L}(l)$ or $\rho_{M}(m)\ge\rho_{L}(l)$. In either case, $\rho_{M\oplus L}(m,l')\ge\rho_{L}(l)$. Hence, $\rho_{M\oplus _{K}L}([0,l])=\rho_{L}(l)$.
\item 
Let $f:X\rightarrow Y$ and $g:Y\rightarrow Z$ be a strict epimorphisms in $$\mathrm{SNrm}^{\le1}_{R},\mathrm{Nrm}^{\le1}_{R},\mathrm{Ban}^{\le1}_{R},\mathrm{SNrm}^{nA,\le1}_{R},\textrm{ or }\mathrm{Nrm}^{nA,\le1}_{R}.$$ That is, they are surjections, and satisfy $\rho_{Y}(y)=\mathrm{inf}_{x\in f^{-1}(y)}\{\rho_{X}(x)\}$ and $\rho_{Z}(z)=\mathrm{inf}_{y\in g^{-1}(z)}\{\rho_{Y}(y)\}$. In particular, both $f$ and $g$ are non-expanding. Consider $g\circ f:X\rightarrow Z$. Let $z\in Z$. Then
$$\rho_{Z}(z)=\mathrm{inf}_{y\in g^{-1}(z)}\rho_{Y}(y)=\mathrm{inf}_{y\in g^{-1}(z)}\mathrm{inf}_{x\in f^{-1}(y)}\rho_{X}(x)=\mathrm{inf}_{x\in (g\circ f)^{-1}(z)}\rho_{X}(x),$$
as required.

\item 
Let $f:X\rightarrow Y$ and $g:Y\rightarrow Z$ be morphisms in $\mathrm{SNrm}^{\le1}_{R},\mathrm{Nrm}^{\le1}_{R},\mathrm{Ban}^{\le1}_{R},\mathrm{SNrm}^{nA,\le1}_{R},\mathrm{Nrm}^{nA,\le1}_{R}$ that are isometries. Let $x\in X$. Then $\rho_{Z}(g(f(x)))=\rho_{Y}(f(x))=\rho_{X}(x)$, so $g\circ f$ is an isometry. Suppose that $X,Y,Z$ are all normed, and that both $g$ and $f$ have closed image. Let $(g(f(x_{n})))$ be a sequence in $Z$ that converges to some $z$. Then $z=g(y)$ for some $y\in Y$. Since $g$ is an isometry, $f(x_{n})$ converges to $y$, so $y=f(x)$ for some $x$. Then $x_{n}$ converges to $x$. Hence $g(f(x_{n}))$ converges to $g(f(x))=z$. 
\item 
Let $f:X\rightarrow Y$ and $g:Y\rightarrow Z$ be morphisms in $\mathrm{SNrm}^{\le1}_{R},\mathrm{Nrm}^{\le1}_{R},\mathrm{Ban}^{\le1}_{R},\mathrm{SNrm}^{nA,\le1}_{R},\mathrm{Nrm}^{nA,\le1}_{R}$ such that $g\circ f$ is a strict epimorphism. In particular, $g\circ f$ is surjective, so $g$ is also surjective. Consider $z\in Z$. Then 
$$\rho_{Z}(z)\le\mathrm{inf}_{y\in g^{-1}(z)}\rho_{Y}(y)\le\mathrm{inf}_{x\in(g\circ f)^{-1}(z)}\rho_{X}(x)=\rho_{Z}(z).$$
So all inequalities are equalities, and $g$ is strict. 
\item 
Let $f:X\rightarrow Y$ and $g:Y\rightarrow Z$ be morphisms in $\mathrm{SNrm}^{\le1}_{R},\mathrm{Nrm}^{\le1}_{R},\mathrm{Ban}^{\le1}_{R},\mathrm{SNrm}^{nA,\le1}_{R},\mathrm{Nrm}^{nA,\le1}_{R}$ such that $g\circ f$ is an  isometry. Let $x\in X$. Then 
$$\rho_{X}(x)=\rho_{Z}(g(f(x)))\le\rho_{Y}(f(x))\le\rho_{X}(x).$$
Hence, $\rho_{X}(x)=\rho_{Y}(f(x))$, and $f$ is an isometry. Suppose that $X,Y$, and $Z$ are all normed, and that $g\circ f$ is an isometry with closed image. We claim that $f$ has closed image. Let $f(x_{n})$ converge to some $y\in Y$. Then $g(f(x_{n}))$ converges to $g(y)$. Since $g\circ f$ is an isometry with closed image, $g(y)=g(f(x))$ for some $x\in X$. Since $g\circ f$ is an isometry, $x_{n}$ converges to $x$, so $f(x_{n})$ converges to $f(x)$. 
\end{enumerate}
\end{proof}

\begin{lemma}\label{lem:snrmqac}
    The categories
    $$\mathrm{SNrm}_{R},\mathrm{Nrm}_{R},\mathrm{Ban}_{R},\mathrm{SNrm}^{nA}_{R},\mathrm{Nrm}^{nA}_{R},\mathrm{Ban}^{nA}_{R}$$
    are all quasi-abelian.
\end{lemma}

\begin{proof}
    This easily follows from rescaling, and the fact that the non-expanding variants are totally parabelian categories.
\end{proof}

\subsubsection{The monoidal structures}
Suppose in this subsection that $R$ is commutative.
Here we recall from \cite{koren} Appendix A, the closed symmetric monoidal structure on $\mathrm{XNrm}_{R}^{\le1}.$
For $X,Y\in\mathrm{SNrm}_{R}^{\le1}$, we define $X\otimes_{\pi}Y$ to be the algebraic tensor product of $X$ and $Y$ over $R$, equipped with the seminorm
$$\rho(v)=\mathrm{inf}\{\sum_{i=1}^{n}\rho_{X}(x_{i})\rho_{Y}(y_{i}):v=\sum_{i=1}^{n}x_{i}\otimes y_{i}\}$$
where the infimum is taken over all ways of writing $v$ as a sum of basic tensors. The monoidal unit is $R$, and the unitors, associators, and symmetric braiding are defined in the obvious way. 

We define
$$\underline{\mathrm{Hom}}_{R}(X,Y)$$
to be the module of bounded $R$-linear operators, equipped with the operator seminorm
$$\rho_{\underline{\mathrm{Hom}}_{R}(X,Y)}(T)=\mathrm{inf}\{C>0:\rho_{Y}(T(x))\le C\rho_{X}(x)\;\forall x\in X\}.$$

We define a function
$$\psi:\mathrm{Hom}_{\mathrm{SNrm}_{R}^{\le1}}(X\otimes_{\pi}Y,Z)\rightarrow\mathrm{Hom}_{\mathrm{SNrm}_{R}^{\le1}}(X,\underline{\mathrm{Hom}}_{\mathrm{SNrm}_{R}^{\le1}}(Y,Z))$$
by sending
$$f\mapsto(x\mapsto f(x\otimes -)).$$
It is easy to see that this is well-defined. In the other direction, we define
$$\phi:\mathrm{Hom}_{\mathrm{SNrm}_{R}^{\le1}}(X,\underline{\mathrm{Hom}}_{\mathrm{SNrm}_{R}^{\le1}}(Y,Z))\rightarrow\mathrm{Hom}_{\mathrm{SNrm}_{R}^{\le1}}(X\otimes_{\pi}Y,Z)$$
$$g\mapsto (x\otimes y\mapsto g(x)(y)).$$
The maps $\psi$ and $\phi$ are inverse to each other, realising $\underline{\mathrm{Hom}}_{R}(X,Y)$ as an internal hom for the symmetric monoidal structure $(R,\otimes_{\pi})$.

If $Y$ is normed, then so is $\underline{\mathrm{Hom}}_{R}(X,Y)$, and if $Y$ is Banach, then so is $\underline{\mathrm{Hom}}_{R}(X,Y)$. Hence, by Day's Reflection Theorem \cite{day1972reflection}, defining $X\overline{\otimes}_{\pi}Y\defeq\mathrm{Sep}(X\otimes_{\pi}Y)$ (resp. $X\hat{\otimes}_{\pi}Y\defeq\overline{\mathrm{Cpl}}(X\otimes_{\pi}Y)$) gives a closed symmetric monoidal structure on $\mathrm{Nrm}_{R}^{\le1}$ (resp. $\mathrm{Ban}_{R}^{\le1}$).

There is also a version for non-Archimedean modules. The internal hom is the same, but the norm on the tensor product is given by
$$\rho(v)=\mathrm{inf}\{\mathrm{max}_{i=1}^{n}\{\rho_{X}(x_{i})\rho_{Y}(y_{i})\}:v=\sum_{i=1}^{n}x_{i}\otimes y_{i}\}$$
where the infimum is again taken over all ways of writing $v$ as a sum of basic tensors.

\subsubsection{Presentability}

Here we show that the non-expanding categories are all locally $\aleph_{1}$-presentable.

\begin{proposition}\label{prop:filtcolim}
    Let $F:\mathcal{I}\rightarrow\mathrm{SNrm}_{R}^{\le1}$ or $F:\mathcal{I}\rightarrow\mathrm{SNrm}_{R}^{nA,\le1}$ with $\mathcal{I}$ filtered. Then the following hold.
    \begin{enumerate}
        \item 
        The colimit $\colim^{\le1}_{i\in\mathcal{I}}F(i)$ has as underlying $R$-module the colimit of the underlying $R$-modules of each $F(i)$, with seminorm given by 
    $$\rho([x])=\mathrm{inf}\{\rho_{j}(f_{ji}(x)):i\rightarrow j\},$$
    where $x\in F(i)$, and $f_{ji}:F(i)\rightarrow F(j)$ is the connecting map.
    \item 
    If $\mathcal{I}$ is $\aleph_{1}$-filtered and each $F(i)$ is normed, then $\colim^{\le1}_{i\in\mathcal{I}}F(i)$ is normed.
        \item 
    If $\mathcal{I}$ is $\aleph_{1}$-filtered and each $F(i)$ is Banach, then $\colim^{\le1}_{i\in\mathcal{I}}F(i)$ is Banach.
    \end{enumerate}
\end{proposition}

\begin{proof}
\begin{enumerate}
    \item
        Let $g_{i}:F(i)\rightarrow G$ be a compatible system of non-expanding maps. We need to verify that the canonical map $g:\colim_{i\in\mathcal{I}}F(i)\rightarrow G$ is non-expanding. Let $x\in F(i)$. Then $g([x])=g_{i}(x)=g_{j}(f_{ji}(x))$ for all $i\rightarrow j$. Thus $\rho_{G}(g_{j}(f_{ji}(x)))\le\rho_{j}(f_{ji}(x))$. This is true for all $i\rightarrow j$, so $\rho_{G}(g[x])\le\rho([x]).$
        \item
        Suppose that $\mathcal{I}$ is $\aleph_{1}$-filtered and each $F(i)$ is normed. Let $x\in F(i)$ be such that $\rho([x])=0$. This means that for every $n\in\mathbb{N}$ there is $i\rightarrow j_{n}$ such that $\rho_{j_{n}}(f_{j_{n}i}(x))\le\frac{1}{n}$. Let $\{j_{n}\rightarrow j_{\infty}\}$ be a cocone for the maps $i\rightarrow j_{n}$. Then $\rho_{j}(f_{ji}(x))\le\frac{1}{n}$ for all $n$, and so $f_{ji}(x)=0$. Hence $[x]=0$.
        \item 
         Suppose that $\mathcal{I}$ is $\aleph_{1}$-filtered and each $F(i)$ is Banach. We already know that $\colim^{\le1}_{i\in\mathcal{I}}F(i)$ is normed. Let $([x_{n}])$ be a Cauchy sequence in $\colim^{\le1}_{i\in\mathcal{I}}F(i)$. By the $\aleph_{1}$-filtered condition, we may assume they are all representable by elements of the same $F(i)$. Also by the $\aleph_{1}$-condition, we may assume that they form a Cauchy sequence in $F(i)$. Therefore they converge there, and hence in $\colim^{\le1}_{i\in\mathcal{I}}F(i)$.
\end{enumerate}

\end{proof}

\begin{lemma}\label{lem:alpeh1compact}
    In $\mathrm{SNrm}_{R}^{\le1}$, $\mathrm{Nrm}_{R}^{\le1}$, $\mathrm{Ban}_{R}^{\le1}$, $\mathrm{SNrm}_{R}^{nA,\le1}$, $\mathrm{Nrm}_{R}^{nA,\le1}$, and $\mathrm{Ban}_{R}^{nA,\le1}$ each object of the form $R_{\delta}$ is $\aleph_{1}$-presentable. 
\end{lemma}

\begin{proof}
By Proposition \ref{prop:filtcolim} it suffices to consider the seminormed case. Let $F:\mathcal{I}\rightarrow\mathrm{SNrm}_{R}^{\le1}$ be an $\aleph_{1}$-filtered diagram. We need to show that $\overline{B}_{\colim_{\mathcal{I}}F(i)}(0,\delta)=\colim_{\mathcal{I}}\overline{B}_{F(i)}(0,\delta)$. Let $x\in F(i)$ with $\rho([x])=\delta$. Therefore $\mathrm{inf}_{i\rightarrow j}\rho_{j}(\alpha_{ij}(x))=\delta$.  Pick a countable subset $\{i\rightarrow j_{k}:k\in\mathbb{N}\}$ such that $\mathrm{inf}_{k\in\mathbb{N}}\rho_{j_{k}}(\alpha_{ij_{k}}(x))=\delta$. Let $j_{k}\rightarrow u$ be a cocone, for the maps $i\rightarrow j_{k}$, i.e., such that the compositions $i\rightarrow j_{k}\rightarrow u$ all coincide for all $k$. Then $\rho_{u}(\alpha_{iu}(x))\le\rho_{j_{k}}(\alpha_{ij_{k}}(x))$ for all $k$, so $\rho_{u}(\alpha_{iu}(x))\le\delta$, and this completes the proof.
\end{proof}

\begin{proposition}
      For $\mathcal{C}$ one of the categories $\mathrm{SNrm}_{R}^{\le1}$, $\mathrm{Nrm}_{R}^{\le1}$, $\mathrm{SNrm}_{R}^{nA,\le1}$, or $\mathrm{Nrm}_{R}^{nA,\le1}$, and for each $\delta\in\mathbb{R}_{\ge0}$ (or for each $\delta\in\mathbb{R}_{>0}$ in the normed cases) the functor
      $$\mathrm{Hom}(R_{\delta},-):\mathcal{C}\rightarrow\mathrm{Set}$$
      commutes with filtered colimits of functors $F:\mathcal{I}\rightarrow\mathcal{C}$ such that for each map $i\rightarrow j$ in $\mathcal{I}$, the map $F(i)\rightarrow F(j)$ is an isometry. In particular, in $\mathrm{SNrm}_{R}^{nA,\le1}$ and $\mathrm{Nrm}_{R}^{nA,\le1}$, for each $\delta\in\mathbb{R}_{\ge0}$ (or for each $\delta\in\mathbb{R}_{>0}$ in the normed cases)  the functor
       $$\mathrm{Hom}(R_{\delta},-):\mathcal{C}\rightarrow\mathrm{Ab}$$
       commutes with arbitrary direct sums.
\end{proposition}

\begin{proof}
Using Remark \ref{rem:rescalingprops}, we need to prove that for any $\delta$, the natural map
$$\colim_{i\in\mathcal{I}}\overline{B}_{F(i)}(0,\delta)\rightarrow\overline{B}_{\colim_{i\in\mathcal{I}}F(i)}(0,\delta)$$
is bijective. Note that, by the explicit description of the seminorm on the colimit, if $[x]\in\colim_{i\in\mathcal{I}}F(i)$, with $x$ a representative in $F(i)$, then $\rho([x])=\rho_{i}(x)$. In particular, if each $F(i)$ is normed then so is the colimit. Furthermore, each $F(i)\rightarrow\colim_{i\in\mathcal{I}}F(i)$ is itself an isometry. We get injections 
$$\overline{B}_{F(i)}(0,\delta)\rightarrow\overline{B}_{\colim_{i\in\mathcal{I}}F(i)}(0,\delta).$$
Since colimits of injections in $\mathrm{Set}$ are injections, we find that
$$\colim_{i\in\mathcal{I}}\overline{B}_{F(i)}(0,\delta)\rightarrow\overline{B}_{\colim_{i\in\mathcal{I}}F(i)}(0,\delta)$$
is an injection. Thus, it remains to prove that this map is surjective. Let $[x]\in\overline{B}_{\colim_{i\in\mathcal{I}}F(i)}(0,\delta)$ with $x$ a representative in $F(i)$. We have $\delta=\rho([x])=\rho_{i}(x)$. Hence $[x]$ is in the image of $\overline{B}_{F(i)}(0,\delta)\rightarrow\overline{B}_{\colim_{i\in\mathcal{I}}F(i)}(0,\delta)$.

For the non-Archimedean case, observe that a direct sum may be written as a filtered colimit of finite sums, and that the inclusion of a sum into a larger sum is an isometry. Thus, the claim reduces to the fact that in the non-Archimedean, and hence additive, setting $\mathrm{Hom}(R_{\delta},-):\mathrm{XNrm}_{R}^{nA,\le1}\rightarrow\mathrm{Ab}$ commutes with finite sums. 
\end{proof}

\begin{lemma}\label{lem:genR}
    In $\mathrm{SNrm}_{R}^{\le1}$ and $\mathrm{SNrm}_{R}^{nA,\le1}$, the collection of objects $\{R_{\delta}:\delta\in\mathbb{R}_{\ge0}\}$ is a set of admissible generators. In $\mathrm{Nrm}_{R}^{\le1}$, $\mathrm{Ban}_{R}^{\le1}$, $\mathrm{Nrm}_{R}^{nA,\le1}$, and $\mathrm{Ban}_{R}^{nA,\le1}$ the collection of objects $\{R_{\delta}:\delta\in\mathbb{R}_{>0}\}$ is a set of admissible generators.
\end{lemma}

\begin{proof}
    Let $(X,\rho_{X})$ be a seminormed $R$-module. For each $x\in X$ with $x\neq 0$, consider the $R$-module $R_{\rho_{X}(x)}$. There is a natural map $R_{\rho_{X}(x)}\rightarrow X$ sending $r\mapsto rx$. We therefore get an induced map
    $$\pi_{X}:R_{X}=\coprod_{x\in X\setminus\{0\}}R_{\rho_{X}(x)}\rightarrow X$$
    $$(r_{x})_{x\in X\setminus\{0\}}\mapsto\sum_{x\in X\setminus\{0\}}xr_{x}.$$
    Let $x\in X$. Then the standard unit vector $\delta_{x}$ maps to $x$ under $\pi$, and $\rho_{R_{X}}(\delta_{x})=\rho_{X}(x)=\rho_{X}(\pi(\delta_{x}))$. Clearly, then $\pi$ is a strict epimorphism. This proves the seminormed case. Note that if $(X,\rho_{X})$ is separated, then for all $x\in X\setminus\{0\}$, $\rho_{X}(x)\neq0$ by definition. Thus, we have also proved the normed case. The Banach case follows from the fact that $\overline{\mathrm{Cpl}}$ is a left adjoint - it therefore preserves strict epimorphisms and coproducts.
\end{proof}

The next proposition is a consequence of Theorem \ref{thm:locpresentgen},  Lemma \ref{lem:catpropnon}, Theorem \ref{cor:pushpullnonexp} (5), Lemma \ref{lem:alpeh1compact}, and Lemma \ref{lem:genR}.

\begin{proposition}
    The categories  $\mathrm{SNrm}_{R}^{\le1}$, $\mathrm{Nrm}_{R}^{\le1}$, and $\mathrm{Ban}_{R}^{\le1}$, $\mathrm{SNrm}_{R}^{nA,\le1}$, $\mathrm{Nrm}_{R}^{nA,\le1}$, and $\mathrm{Ban}_{R}^{nA,\le1}$ are all locally $\aleph_{1}$-presentable. 
\end{proposition}

\begin{proposition}\label{prop:kappssmallkapsmall}
        Let $R$ be a seminormed ring, let $\kappa\ge |R|$ be a regular uncountable cardinal, and let $M$ be an object of $\mathrm{XNrm}_{R}^{\le1}$. If the underlying set of $M$ is $\kappa$-small, then $M$ is $\kappa$-presentable
\end{proposition}
\begin{proof}
  It suffices to prove the seminormed case. Consider the strict epimorphism $\coprod_{m\in M}R_{\rho_{M}(m)}\rightarrow M$. Since $\kappa\ge |R|$ is a regular cardinal, $\pi_{M}:\coprod_{m\in M}R_{\rho_{M}(m)}$ is $\kappa$-small as a set. The kernel $K$ of $\pi_{M}$ is also $\kappa$-small, and there is also a strict epimorphism $\pi_{K}:\coprod_{k\in K}R_{\rho_{K}(k)}\rightarrow K$. Then $M$ is the cokernel of the canonical map $\coprod_{k\in K}R_{\rho_{K}(k)}\rightarrow \coprod_{m\in M}R_{\rho_{M}(m)}$. Both $\coprod_{k\in K}R_{\rho_{K}(k)}$ and $\coprod_{m\in M}R_{\rho_{M}(m)}$ are $\kappa$-small colimits of $\aleph_{1}$-presentable, and hence $\kappa$-presentable, objects. Consequently, they are $\kappa$-presentable. Hence, $M$ is $\kappa$-presentable.
\end{proof}
\begin{proposition}
   Let $R$ be a seminormed ring, and let $\kappa\ge |R|$ be an uncountable regular cardinal. Then an object $X$ of $\mathrm{SNrm}_{R}^{\le1}$ or $ \mathrm{Nrm}_{R}^{\le1}$ is $\kappa$-presentable if and only if the underlying module of $X$ is $\kappa$-small. If $\kappa$ is such that $X^{\aleph_{0}}$ is $\kappa$-small whenever $X$ is $\kappa$-small, then this is also true in $\mathrm{Ban}_{R}^{\le1}$.
\end{proposition}
\begin{proof}
In all three categories, if $X$ has size strictly less than $\kappa$, then it is $\kappa$-presentable by Proposition \ref{prop:kappssmallkapsmall}.
    Conversely, suppose that $X$ is $\kappa$-presentable. By Remark \ref{rem:lambdapresobjects}, $X$ is a quotient of a $\kappa$-small colimit of the $\aleph_{1}$-compact (and hence $\kappa$-compact), generators $R_{\epsilon}$. In $\mathrm{SNrm}^{\le1}_{R}$ and $\mathrm{Nrm}^{\le1}_{R}$ By the construction of colimits in Proposition \ref{prop:filtcolim} and Lemma \ref{lem:catpropnon}, clearly such objects are $\kappa$-small. In the Banach case, we may have to take a completion. 
\end{proof}
\begin{corollary}\label{cor:adepik}
      Let $R$ be a seminormed ring, and let $\kappa\ge |R|$ be an uncountable regular cardinal. In the Banach case, further assume that $X^{\aleph_{0}}$ is $\kappa$-small whenever $X$ is $\kappa$-small. Then, the category $(\mathrm{XNrm}_{R})_{\kappa}$ of $\kappa$-compact objects is closed under finite limits. Moreover, $\mathrm{XNrm}_{R}$ is $\kappa$-subobject coherent. If $\kappa$ is strongly inaccessible, then $(\mathrm{XNrm}_{R})_{\kappa}$ is closed under $\kappa$-small limits.
\end{corollary}
\begin{proof}
If $X$ is $\kappa$-small, then clearly so is any subobject. Thus, closure under finite limits amounts to observing that finite products of $\kappa$-small sets are $\kappa$-small. When $\kappa$ is strongly inaccessible, $\kappa$-small products of $\kappa$-small sets are still $\kappa$-small. The product of a collection of objects of $\mathrm{XNrm}_{R}^{\le1}$ is contained in the product of their underlying sets.
Now, let $Y$ be in $(\mathrm{XNrm}_{R})_{\kappa}$, and $\pi:X\rightarrow Y$ an admissible epimorphism. For each $y\in Y$, pick a sequence of elements $(x^{y}_{n})_{n=1}^{\infty}$ of $X$ such that $\pi(x^{y}_{n})=y$ for all $y$, and $\rho_{X}(x^{y}_{n})$ converges to $\rho(y)$ as $n\rightarrow\infty$. Consider the submodule $\tilde{X}$ of $X$ generated by $\{x_{n}^{y}:y\in Y,1\le n<\infty\}$ (and take its closure in the normed and Banach settings). Then $\tilde{X}$ is $\kappa$-small, and the restriction of $\pi$ to $\tilde{X}\rightarrow Y$ is an admissible epimorphism by construction. 
\end{proof}
\begin{remark}\label{rem:cofinalcofinal}
    For every cardinal $\kappa$, there is an uncountable regular cardinal $\kappa'\ge\kappa$ for which $X^{\aleph_{0}}$ is $\kappa'$-small whenever $X$ is $\kappa'$-small. As in \cite{bagaria2015definable} Page 571, or \cite{466980}, we can take $\kappa'=(2^{\mathrm{max}(\lambda,\kappa_{0})})^{+}$.
\end{remark}
\begin{remark}
    In the categories $\mathrm{SNrm}_{R}^{\le1}$ and $\mathrm{SNrm}_{R}^{\le1,nA}$, the classes of monomorphisms are generated as a weakly saturated class by a set of monomorphisms. Indeed, this follows from Lemma \ref{lem:catpropnon}, Proposition \ref{prop:filtcolim}, Lemma \ref{lem:alpeh1compact}, and\cite{MR1780498} Proposition 1.12. This means that for any object $X$ of $\mathrm{SNrm}_{R}^{\le1}$ or $\mathrm{SNrm}_{R}^{\le1,nA}$, there is a monomorphism $X\rightarrow I$, where $I$ has the right lifting property against all monomorphisms. This is nothing mysterious, one can take $I$ to be $X$ equipped with the seminorm that takes the constant value $0$.
\end{remark}

\subsubsection{Exactness of filtered colimits}

Next, we consider exactness of filtered colimits.

\begin{proposition}
    If $u:X\rightarrow Y$ is a strict monomorphism in $\mathrm{SNrm}_{R}^{\le1}$ (resp. in $\mathrm{SNrm}_{R}^{nA,\le1}$), then $\overline{\mathrm{Cpl}}(u):\overline{\mathrm{Cpl}}(X)\rightarrow\overline{\mathrm{Cpl}}(Y)$ is a strict monomorphism in $\mathrm{Ban}_{R}^{\le1}$ (resp. in $\mathrm{Ban}_{R}^{nA,\le1}$).
\end{proposition}

\begin{proof}
    Clearly, $\overline{\mathrm{Cpl}}(u)$ is an isometry. We need to show that it has closed image. Let $[(x_{n,m})_{n}]_{m}$ be a sequence in $\overline{\mathrm{Cpl}}(X)\subset\overline{\mathrm{Cpl}}(Y)$ converging to some $y\in\overline{\mathrm{Cpl}}(Y)$. In particular, $[(x_{n,m})_{n}]_{m}$ is Cauchy in $\overline{\mathrm{Cpl}}(Y)$. But then it is Cauchy in $\overline{\mathrm{Cpl}}(X)$, and hence converges in $\overline{\mathrm{Cpl}}(X)$. The non-Archimedean case is identical.
\end{proof}

\begin{corollary}\label{cor:cplexact}
\begin{enumerate}
    \item 
        The functor $\overline{\mathrm{Cpl}}:\mathrm{SNrm}_{R}^{\le1}\rightarrow\mathrm{Ban}_{R}^{\le1}$ is exact for the parabelian proto-exact structure.
    \item 
     The functor $\overline{\mathrm{Cpl}}:\mathrm{SNrm}_{R}^{nA,\le1}\rightarrow\mathrm{Ban}_{R}^{nA,\le1}$ is exact for the quasi-abelian exact structure.
\end{enumerate}

\end{corollary}

\begin{proof}
We prove the first claim, the second is identical.
   The functor $\overline{\mathrm{Cpl}}$ is left adjoint to the inclusion $\mathrm{Ban}_{R}^{\le1}\rightarrow\mathrm{SNrm}_{R}^{\le1}$. Thus it commutes with cokernels. Consider a short exact sequence
   \begin{displaymath}
       \xymatrix{
       0\ar[r] & X\ar[r]^{f} & Y\ar[r]^{g} & Z\ar[r] &0
       }
   \end{displaymath}
   in $\mathrm{SNrm}_{R}^{\le1}$. Then $\overline{\mathrm{Cpl}}(f)$ is a strict monomorphism, so it is the kernel of its cokernel. But its cokernel is $\overline{\mathrm{Cpl}}(g)$. This completes the proof.
\end{proof}

\begin{lemma}
\begin{enumerate}
    \item 
        Filtered colimits are exact in $\mathrm{SNrm}_{R}^{\le1}$ and $\mathrm{Ban}_{R}^{\le1}$. In $\mathrm{SNrm}_{R}^{\le1}$ filtered colimits are in fact strongly exact - that is, they commute with kernels and cokernels. 
        \item 
            Filtered colimits are exact in $\mathrm{SNrm}_{R}^{nA,\le1}$ and $\mathrm{Ban}_{R}^{nA,\le1}$. In $\mathrm{SNrm}_{R}^{nA,\le1}$ filtered colimits are in fact strongly exact - that is, they commute with kernels and cokernels.
            \item 
            In each of the categories $\mathrm{Nrm}_{R}^{\le1}$ and $\mathrm{Nrm}_{R}^{nA,\le1}$, the class of all objects is unionable.
\end{enumerate}

\end{lemma}

\begin{proof}
\begin{enumerate}
    \item 
      First we prove that filtered colimits in $\mathrm{SNrm}_{R}^{\le1}$ commute with all kernels and cokernels. Evidently they commute with cokernels. Let $F,G:\mathcal{I}\rightarrow\mathrm{SNrm}^{\le1}_{R}$ be filtered diagrams, and $\eta:F\rightarrow G$ a natural transformation. Let $K$ be the functor defined by $K(i)=\mathrm{Ker}(\eta(i))$. Let $|-|:\mathrm{SNrm}_{R}^{\le1}\rightarrow{}_{|R|}\mathrm{Mod}$ denote the forgetful functor to the category of $|R|$-modules, where $|R|$ denotes the underlying ring of $R$. This functor commutes with all limits and colimits. In particular, the underlying modules of 
    $$\colim_{i}\mathrm{Ker}(\eta(i))$$
    and
    $$\mathrm{Ker}(\colim_{i}\eta(i))$$
    coincide. It remains to check that the seminorms coincide. This follows from the explicit description of the seminorm on the colimit, and the submodule seminorm. The claim for $\mathrm{Ban}_{R}^{\le1}$ now follows from the claim for $\mathrm{SNrm}_{R}^{\le1}$, and exactness of $\overline{\mathrm{Cpl}}$.
    \item 
    This is identical to the first claim.
    \item 
    Let $\mathcal{C}$ denote either $\mathrm{Nrm}_{R}^{\le1}$ or $\mathrm{Nrm}_{R}^{nA,\le1}$.
    We need to show that, for any seminormed $R$-module $N$, and any transfinite sequence 
    $$F:\lambda\rightarrow\mathcal{C}$$
    equipped with a natural transformation $\eta:F\rightarrow\mathrm{cst}(N)$ to the constant diagram on $N$, such that for each $\alpha$, $\eta(\alpha):F(\alpha)\rightarrow N$ is a strict monomorphism, then the colimit
    $$\colim_{\alpha}F(\alpha)\rightarrow N$$
    is a strict monomorphism. By Part $(1)$, the colimit of seminormed modules
     $$\colim^{\mathrm{SNrm}}_{\alpha}F(\alpha)\rightarrow N$$
     is a strict monomorphism. This means, in particular, that the colimit computed in the category of seminormed modules, $\colim^{\mathrm{SNrm}}_{\alpha}F(\alpha)$, is isometrically isomorphic to a sub-module of $N$. Consequently, it is normed. Hence it is also the colimit of normed modules, and we are done.
\end{enumerate}
\end{proof}

\begin{corollary}\label{cor:exactfiltered}
\begin{enumerate}
\item 
    The category $\mathrm{SNrm}_{R}^{\le1}$ is purely $\aleph_{1}$-locally presentable and has strongly exact functors of filtered colimits. The category $\mathrm{Nrm}_{R}^{\le1}$ is purely $\aleph_{1}$-locally presentable, and the class of all objects is unionable. The category $\mathrm{Ban}_{R}^{\le1}$ is purely $\aleph_{1}$-locally presentable and has exact functors of filtered colimits. 
    \item 
    The category $\mathrm{SNrm}_{R}^{nA,\le1}$ is purely $\aleph_{1}$-locally presentable and has strongly exact functors of filtered colimits. The category $\mathrm{Nrm}_{R}^{nA,\le1}$ is purely $\aleph_{1}$-locally presentable, and the class of all objects is unionable. The category $\mathrm{Ban}_{R}^{nA,\le1}$ is purely $\aleph_{1}$-locally presentable and has exact functors of filtered colimits. 
    In particular, all such categories have enough injectives. 
\end{enumerate}
\end{corollary}


\subsection{The completely non-additive version: seminormed and normed sets.}
Here we briefly discuss pointed seminormed sets.
Following \cite{durov2017homotopy}, let $\Delta$ be a partially ordered set. For convenience, we shall assume that every set in $\Delta$ has an infimum. In particular, $\Delta$ has a minimal element, that we call $0$. 
\begin{definition}
\begin{enumerate}
    \item 
    The category $\mathrm{SNrm}_{\Delta}^{\le1}$ of $\Delta$-\textit{seminormed sets} is a pair $(X,\rho)$ where $X$ is a set, and $\rho:X\rightarrow\Delta$ is a function. A \textit{morphism} $f:(X,\rho_{X})\rightarrow(Y,\rho_{Y})$ of $\Delta$-seminormed sets is a function $f:X\rightarrow Y$ such that $\rho_{Y}(f(x))\le\rho_{X}(x)$ for all $x\in X$. We denote the category of $\Delta$-seminormed sets and their morphisms by $\mathrm{SNrm}_{\Delta}^{\le1}$.
    \item
    A $\Delta$-seminormed set $(X,\rho)$ is said to be \textit{normed} if $\rho(x)\neq0$ for all $x\in X$. The full subcategory of $\mathrm{SNrm}_{\Delta}^{\le1}$ consisting of normed sets is denoted $\mathrm{Nrm}_{\Delta}^{\le1}$.
    \item 
    We denote by $\mathrm{SNrm}^{\le1}_{\Delta_{*}}$ the category of pointed objects in $\mathrm{SNrm}_{\Delta}^{\le1}$. We may think of an object of this category as a pair $(X,\rho)$ where $X$ is a pointed set, and $\rho:X\rightarrow\Delta$ is a function such that $\rho(0_{X})=0$.
    \item 
    We denote  $\mathrm{Nrm}^{\le1}_{\Delta_{*}}\subset \mathrm{SNrm}^{\le1}_{\Delta_{*}}$ the full subcategory consisting of those objects $(X,\rho)$ such that $\rho(x)\neq0$ for all $x\neq 0_{X}$.
    \end{enumerate}
\end{definition}
Consider the category $\mathrm{Fun}(\Delta,\mathrm{Set})$. There is a fully faithful functor $i:\mathrm{SNrm}_{\Delta}^{\le1}\rightarrow\mathrm{Fun}(\Delta,\mathrm{Set})$ sending $(X,\rho)$ to $\delta\mapsto\{x\in X:\rho(x)\le\delta\}$. By \cite{durov2017homotopy}*{Theorem 2.2.24}, this functor has a left adjoint $Q$. Thw functor $Q$ sends $F\in\mathrm{Fun}(\Delta,\mathrm{Set})$ to $(\colim_{\delta\in\Delta}(F(\delta)),\rho_{F})$, where $\rho_{F}([x])=\mathrm{inf}\{\delta\in\Delta:[x]\in\mathrm{Im}(F(\delta)\rightarrow\colim_{\delta'\in\Delta}(F(\delta')))\}$. In particular, $\mathrm{SNrm}^{\le1}_{\Delta}$ is complete and cocomplete.
Now, consider the functor 
$$\mathrm{Sep}:\mathrm{SNrm}^{\le1}_{\Delta_{*}}\rightarrow\mathrm{Nrm}^{\le1}_{\Delta_{*}}$$
defined as follows. For $(X,\rho)\in\mathrm{SNrm}^{\le1}_{\Delta_{*}}$, define an equivalence relation $\sim_{0}$ on $X$ by $x_{1}\sim x_{2}$ if $\rho(x_{1})=\rho(x_{2})=0$. The function $\rho$ is constant on equivalence classes, so defines a function $\overline{\rho}$ on $\overline{X}\defeq X\big\slash\sim_{0}$. The functor $(X,\rho)\mapsto(\overline{X},\overline{\rho})$ is left adjoint to the inclusion $\mathrm{Nrm}^{\le1}_{\Delta_{*}}\rightarrow \mathrm{SNrm}^{\le1}_{\Delta_{*}}$. In particular, $\mathrm{SNrm}^{\le1}_{\Delta_{*}}$ is complete and cocomplete.
\begin{theorem}\label{thm:snrmnrmQeffective}
  The categories and $\mathrm{SNrm}^{\le1}_{\Delta_{*}}$ and  $\mathrm{Nrm}^{\le1}_{\Delta_{*}}$ are left totally parabelian categories with $\mathcal{Q}$-effective unions.
\end{theorem}
\begin{proof}
Let us first describe the strict monomorphisms as norm-preserving monomorphisms. Let $g:Y\rightarrow Z$ be a morphism of normed sets. Then $\mathrm{Ker}(g)$ has, as underlying pointed set, the usual kernel of the underlying map of pointed sets. The norm on $\mathrm{Ker}(g)$ is the restriction of the norm on $Y$. In particular, the morphism $\mathrm{Ker}(g)\rightarrow Y$ is seminorm-preserving. Conversely, suppose that $f:X\rightarrow Y$ is a seminorm-preserving inclusion. Then we may write $Y\cong X\vee C$ where $C$ is the complement of the image of $X$ in $Y$, with a point adjoined. Note that if $Y$ is normed then so is $C$. With this representation, the cokernel of $X\rightarrow X\vee C$ is $C$, and the map $X\vee C\rightarrow C$ is the projection, which sends everything in $X$ to $0_{C}$. Clearly, $X$ is the kernel of this map. 
Now, let
\begin{displaymath}
    \xymatrix{
    X\ar[d]^{g}\ar[r]^{f} & Y\ar[d]^{g'}\\
    Z\ar[r]^{f'} & P
    }
\end{displaymath}
be a pushout diagram with $f$ a strict monomorphism, and $g$ arbitrary. We may write $Y\cong X\vee C$, where $C$ is the complement of the image $X$ in $Y$, with a zero $0_{C}$ adjoined. Then $P\cong Z\vee C$. The inclusion $Z\rightarrow Z\vee C$ is then a strict monomorphism. If $g:X\rightarrow Z$ is a strict epimorphism, then $Y\cong X\vee C\rightarrow Z\vee C\cong P$ is as well. Now, let
\begin{displaymath}
    \xymatrix{
    P\ar[d]^{f'}\ar[r]^{g'} & X\ar[d]^{f}\\
    Y\ar[r]^{g} & Z
    }
\end{displaymath}
be a pullback diagram with $f$ being a strict monomorphism, and $g$ being a strict epimorphism. We need to show that $g'$ is a strict epimorphism. By the above, $Y\cong Z\vee\mathrm{Ker}(g)$. A direct computation gives that $P\cong X\vee\mathrm{Ker}(g)$, and $g':X\vee\mathrm{Ker}(g)\rightarrow X$ is just the projection, and this is strict.
Finally, for $\mathcal{Q}$-effective unions, observe that if $X,Y\rightarrow Z$ are strict monomorphisms, then we may regard them as subobjects with the restricted seminorms. Then $X\cap_{Z}Y$ is just the usual intersection as sets, with the restricted seminorm, and $X\cup_{Z}Y$ is the usual union of underlying sets, with the seminorm restricted from $Z$. Clearly then, all the required morphisms are admissible monomorphisms.
\end{proof}
\section{Injectives and derived functors in $\mathrm{XNrm}_{R}$}\label{sec:injban}
Here we consider the question of the existence of enough injectives in categories of seminormed, normed, and Banach modules.
\subsection{Basic properties}
 Before proceeding, let us first discuss some basic properties and examples of injective modules.
\subsubsection{Divisibility}
  In this section, we let $R$ be a commutative seminormed ring and $s$ an element of $R$.
\begin{definition}
   We say that $R$ is $s$\textit{-adic} if
    \begin{enumerate}
        \item 
        $R$ has no $s$-torsion, and
        \item 
          for every $r\in R$, $|sr|=|s| |r|$.
    \end{enumerate}
\end{definition}
Let $R$ be $s$-adic. Consider the ring $R[\frac{1}{s}]$ obtained by inverting $s$. We may equip this with the seminorm given by $|\frac{r}{s^{k}}|=\frac{|r|}{|s|^{k}}$. The natural morphism $R\rightarrow R[\frac{1}{s}]$ is then an isometric embedding. 
\begin{proposition}
    Let $R$ be an $s$-adic seminormed ring, and let $I$ be an injective object of $(\mathrm{SNrm}_{R}^{\le1},Q^{\le1}_{\kappa})$, where $\kappa\ge \mathrm{max}\{\aleph_{0},|R|\}$. Then for all $i\in I$ there exists a $j\in i$ such that $sj=i$ and $\rho_{I}(i)=|s|\rho_{I}(j)$.
\end{proposition}
\begin{proof}
       Let $i\in I$ Consider the isometric embedding $R_{\rho_{I}(i)}\rightarrow (R[\frac{1}{s}])_{\rho_{I}(i)}$. The morphism $R_{\rho_{I}(i)}\rightarrow I$, $r\mapsto ri$ is a non-expanding morphism. By assumption, this extends to a morphism $(R[\frac{1}{s}])_{\rho_{I}(i)}\rightarrow I$. Consider the image $j$ of $\frac{1}{s}\in (R[\frac{1}{s}])_{\rho_{I}(i)}$ in $I$. We have $sj=i$. Moreover, we have 
   $$\rho_{I}(i)=\rho_{I}(sj)\le |s|\rho_{I}(j)\le |s|\frac{1}{|s|}\rho_{I}(i)=\rho_{I}(i),$$
   where we have used that $\rho_{I}(j)\le\rho_{(R[\frac{1}{s}])_{\rho_{I}(i)}}(\frac{1}{s})=\frac{1}{|s|}\rho_{I}(i).$
Therefore, $|s|\rho_{I}(j)=\rho_{I}(i)=\rho_{I}(sj)$. 
\end{proof}
\begin{example}
    Consider the seminormed (in fact, Banach) ring $\mathbb{Z}_{an}$ consisting of the integers equipped with the Euclidean norm. Then $\mathbb{Z}_{an}$ is $n$-adic for all integers $n$. Consequently, for all injective seminormed abelian groups (i.e., seminormed $\mathbb{Z}_{an}$-modules) $I$, all $i\in I$, and all positive integers $n$, there is $j\in I$ such that $nj=i$ and $\rho_{I}(i)=|n|\rho_{I}(j)$. We might call such a seminormed abelian group \textit{divisible}.
\end{example}

\begin{example}
    Consider the seminormed abelian group $\mathbb{Q}_{an}$ of rational numbers equipped with the Euclidean norm. Then $\mathbb{Z}_{an}\rightarrow\mathbb{Q}_{an}$ is an admissible monomorphism. We consider the quotient $\mathbb{Q}_{an}\big\slash\mathbb{Z}_{an}$. This is a divisible group in the sense of the previous example. Indeed, any element may be represented by $\lambda$ with $\lambda\in\mathbb{Q}\cap[-\frac{1}{2},\frac{1}{2})$. This element has norm $|\lambda|$. For $n\in\mathbb{Z}_{>0}$, we have $n(\frac{1}{n}\lambda)=\lambda$, and $||\frac{1}{n}\lambda||=\frac{1}{n}|\lambda|$. Therefore, $\mathbb{Q}_{an}\big\slash\mathbb{Z}_{an}$ is divisible. Similarly, $\mathbb{R}_{an}\big\slash\mathbb{Z}_{an}$ is divisible. We do not know whether this object is injective. 
\end{example}
\subsubsection{The case of $\mathbb{R}$ or $\mathbb{C}$.}
In this subsection, we briefly specialise to the case that $R$ is either $\mathbb{R}$ or $\mathbb{C}$. This is essentially a rephrasing of very classical results.
\begin{proposition}
    Let $R$ be either $\mathbb{R}$ or $\mathbb{C}$. Then $R$ is injective as an object of $\mathrm{SNrm}_{R}^{\le1}$ and $\mathrm{Ban}_{R}^{\le1}$.
\end{proposition}
\begin{proof}
Since the completion functor $\overline{\mathrm{Cpl}}$ is exact, it suffices to prove the claim for $\mathrm{Ban}_{R}^{\le1}$. It also suffices to prove that if $M\rightarrow N$ is an isometric embedding, then $\mathrm{Hom}_{\mathrm{Ban}_{R}^{\le1}}(N,R)\rightarrow \mathrm{Hom}_{\mathrm{Ban}_{R}^{\le1}}(M,R)$ is surjective. This is precisely the Hahn-Banach Theorem.
\end{proof}
Now, since the class of injectives is closed under products by Corollary \ref{cor:prodright}, we have that, for any set $\mathcal{I}$, $l^{\infty}(\mathcal{I})=\prod_{i\in\mathcal{I}}^{\le1}R$ is injective. Let $M$ be an arbitrary Banach space. Consider the admissible epimorphism $\pi_{M}:\coprod^{\le1}_{m\in M\setminus\{0\}}R\rightarrow M$. Now, the dual map $M^{\vee}\rightarrow l^{\infty}(M\setminus\{0\})$ is an isometric embedding with closed image. For arbitrary Banach $N$, we have an isometric embedding $N\rightarrow N^{\vee\vee}$. Thus, we have an isometric embedding
$$N\rightarrow l^{\infty}(N^{\vee}\setminus\{0\}).$$
In other words, $\{l^{\infty}(\mathcal{I})\}$ is a collection of injective cogenerators in $\mathrm{Ban}_{R}^{\le1}$. 
Let $V\in\mathrm{SNrm}_{R}^{\le1}$. Consider the map
$$V\rightarrow \overline{\mathrm{Cpl}}(V)\times \mathrm{triv}(|V|),v\mapsto([v],v),$$
where $[v]$ denotes the constant Cauchy sequence on $v$. It is straightforward to see that this is an isometric embedding. Now, we can further embed $\overline{\mathrm{Cpl}}(V)\rightarrow l^{\infty}(\mathcal{I})$, for some set $\mathcal{I}$. Consequently $\{l^{\infty}(\mathcal{I})\times\mathrm{triv}(M)\}$ is a collection of injective cogenerators in $\mathrm{SNrm}_{R}^{\le1}$.
In the full categories $\mathrm{SNrm}_{R}$ and $\mathrm{Ban}_{R}$, this is of course well known - for a version for locally convex spaces, see \cite{dcfapp} Proposition 1.1.11. Note that this does not imply that $\mathrm{Nrm}_{R}$ has enough injectives. Indeed, for normed $N$, the isometric embedding $N\rightarrow N^{\vee\vee}$ will not have closed image unless $N$ is Banach. Hence, in general, this is not an admissible monomorphism.
\begin{remark}\label{rem:sphericalcompletion}
    The content of this section is also true for $\mathrm{Ban}_{k}^{nA}$ whenever $k$ is a non-Archimedean, non-trivially valued, and spherically complete field, as here the Hahn-Banach Theorem still holds (see e.g., \cite{perez2010locally} Chapter 4). Ben-Bassat and Kremnizer extended this to any non-Archimedean valued field $k$ in \cite{koren} Lemma A.42. The proof works as in Example \ref{example:goingdown}, by passing to a spherical completion $k\rightarrow K$, and using that every object of $\mathrm{Ban}_{k}^{\mathrm{nA}}$ is flat. 
\end{remark}

\subsection{Non-Archimedean injectives}
Now, let us establish the existence of enough injectives in the non-Archimedean case.
\begin{theorem}\label{thm:nainj}
    Let $R$ be a non-Archimedean, non-trivially valued seminormed ring. Then $(\mathrm{XNrm}_{R}^{nA},\mathcal{Q})$ has enough injectives.
\end{theorem}
\begin{proof}
    The category $\mathrm{XNrm}_{R}^{\le1}$ is quasi-abelian, and the quasi-abelian exact structure is, in particular, tame. Therefore, $(\mathrm{XNrm}_{R}^{\le1,nA},\mathcal{Q}_{max}^{\le1})$ has enough injectives by Corollary \ref{cor:injectivemodel}. To conclude, it suffices to prove that if $I$ is an injective object of $(\mathrm{XNrm}_{R}^{\le1,nA},\mathcal{Q}_{max}^{\le1})$, then it is injective when regarded as an object of $(\mathrm{XNrm}_{R}^{nA},\mathcal{Q}_{max})$. To this end, let $f:X\rightarrow Y$ be an admissible monomorphism in $(\mathrm{XNrm}_{R}^{nA},\mathcal{Q}_{max})$, and let $h:X\rightarrow I$ be a morphism in $\mathrm{XNrm}_{R}^{nA}$. Let $\tilde{X}$ denote the set-theoretic image of $f$, with the sub-module seminorm, as a submodule of $Y$. Then $X\rightarrow\tilde{X}$ is an isomorphism in $\mathrm{XNrm}_{R}^{nA}$, since $f$ is an admissible monomorphism. Therefore, we get an induced map $\tilde{h}:\tilde{X}\rightarrow I$. We denote by $\tilde{f}:\tilde{X}\rightarrow Y$ the inclusion. Then $\tilde{f}$ is an admissible monomorphism in $(\mathrm{XNrm}_{R}^{\le1,nA},\mathcal{Q}_{max}^{\le1})$. We need to extend $\tilde{h}$ to $Y$. Now, by rescaling $\tilde{X}$ and $\tilde{Y}$, we may further assume that $\tilde{h}$ is a morphism in $\mathrm{XNrm}_{R}^{\le1,nA}$. Finally, by injectivity of $I$, we get our extension of $\tilde{h}$ to $Y$, and this concludes the proof. 
\end{proof}
\subsection{The case of pointed normed sets}
Now we turn our attention to the next-simplest case: pointed seminormed and normed sets. Combining Corollary \ref{cor:cor:sheavesparabelian} and Theorem \ref{thm:snrmnrmQeffective} gives the following.
     \begin{corollary}
         Let $\mathcal{T}$ be a small site. Then the categories $\mathrm{Shv}(\mathcal{T},\mathrm{SNrm^{\le1}_{\Delta_{*}}})$ and $\mathrm{Shv}(\mathcal{T},\mathrm{Nrm^{\le1}_{\Delta_{*}}})$ are parabelian categories with enough injectives for their respective parabelian proto-exact structures. If $R$ is an associative monoid in $\mathrm{Shv}(\mathcal{T},\mathrm{SNrm}^{\le1}_{\Delta_{*}})$, then ${}_{R}\mathrm{Mod}(\mathrm{Shv}(\mathcal{T},\mathrm{SNrm^{\le1}_{\Delta_{*}}}))$ and ${}_{R}\mathrm{Mod}(\mathrm{Shv}(\mathcal{T},\mathrm{Nrm^{\le1}_{\Delta_{*}}}))$ are also parabelian categories with enough injectives for their respective parabelian proto-exact structures..
     \end{corollary}
\begin{theorem}\label{thm:injpointedsets}
The parabelian category $\mathrm{XNrm}_{\Delta_{*}}$ has enough injectives.
\end{theorem}
\begin{proof}
Let $I$ be injective in $\mathrm{XNrm}^{\le1}_{\Delta_{*}}$. It suffices to prove that $I$ is injective in $\mathrm{XNrm}_{\Delta_{*}}$. Let $f:X\rightarrow Y$ be a strict monomorphism in $\mathrm{XNrm}^{\le1}_{\Delta_{*}}$, and let $h:X\rightarrow I$ be a morphism. We consider the image $\tilde{X}$ of $f$ in $Y$, together with the morphism $\tilde{f}:\tilde{X}\rightarrow Y$ given by the inclusion. This morphism is a strict monomorphism in $\mathrm{XNrm}^{\le1}_{\Delta_{*}}$. Moreover $X\rightarrow\tilde{X}$ is an isomorphism. By composing with the inverse, we get a morphism $\tilde{h}:\tilde{X}\rightarrow X\rightarrow I$. By rescaling $\tilde{X}$ and $Y$, we may assume that $\tilde{h}$ is non-expanding. Therefore there exists a lift $\tilde{g}:Y\rightarrow I$. This is also a lift of $h$.
\end{proof}
\subsection{The Archimedean case}
The Archimedean version is far more complicated. 
\subsubsection{The $\kappa$-restricted Archimedean version}
We begin by consider exact structures determined by cardinals. For each cardinal $\kappa$, consider the class of objects $I_{\kappa}$ from Construction \ref{construction:smallapprox}.  Consider the induced inflation proto-exact structure $\mathcal{Q}^{\le1}_{\kappa}$ on $\mathrm{XNrm}^{\le1}_{R}$, and the induced exact structure $\mathcal{Q}_{\kappa}$ on $\mathrm{XNrm}_{R}$.
\begin{proposition}
    Let 
    $$0\rightarrow X\rightarrow Y\rightarrow Z\rightarrow 0$$
    be an exact sequence in $(\mathrm{XNrm}_{R},\mathcal{Q}_{max})$ (resp. in $(\mathrm{XNrm}^{\le1}_{R},\mathcal{Q}^{\le1}_{max})$), with $X$, $Y$, and $Z$ all being $\kappa$-small. Then this sequence is exact in $(\mathrm{XNrm}_{R},\mathcal{Q}_{\kappa})$ (resp. in $(\mathrm{XNrm}^{\le1}_{R},\mathcal{Q}^{\le1}_{\kappa})$). 
\end{proposition}
\begin{proof}
   In proving the bounded claim, we will in fact prove the non-expanding claim. By replacing $X$ with the kernel (i.e., the submodule) of $Y\rightarrow Z$ and $Z$ with the quotient of $Y$ by this kernel, we may assume the sequence is even exact in $(\mathrm{XNrm}^{\le1}_{R},\mathcal{Q}^{\le1}_{max})$. For each rescaling by $\delta$, the sequence
     $$0\rightarrow X_{\delta}\rightarrow Y_{\delta}\rightarrow Z_{\delta}\rightarrow 0$$
     is also exact in $(\mathrm{XNrm}^{\le1}_{R},\mathcal{Q}^{\le1}_{max})$. Since $X$, $Y$, and $Z$ are all $\kappa$-small, this sequence is then exact in $(\mathrm{XNrm}^{\le1}_{R},\mathcal{Q}^{\le1}_{\kappa})$. But then 
       $$0\rightarrow \mathrm{Hom}_{\mathrm{XNrm}^{\le1}_{R}}(Z_{\delta},I)\rightarrow \mathrm{Hom}_{\mathrm{XNrm}^{\le1}_{R}}(Y_{\delta},I)\rightarrow \mathrm{Hom}_{\mathrm{XNrm}^{\le1}_{R}}(X_{\delta},I)\rightarrow 0$$
       is left semi-exact for any object $I$ with the right lifting property against strict monomorphisms between $\kappa$-small objects. This sequence is just
       $$0\rightarrow \mathrm{Hom}^{\le\delta}_{\mathrm{XNrm}_{R}}(Z,I)\rightarrow \mathrm{Hom}^{\le\delta}_{\mathrm{XNrm}_{R}}(Y,I)\rightarrow \mathrm{Hom}^{\le\delta}_{\mathrm{XNrm}_{R}}(X,I)\rightarrow 0$$
       where $\mathrm{Hom}^{\le\delta}_{\mathrm{XNrm}_{R}}$ denotes the set of linear maps of seminorm at most $\delta$. Taking the colimit over $\delta\in\mathbb{R}_{>0}$ give the exact sequence
             $$0\rightarrow \mathrm{Hom}_{\mathrm{XNrm}_{R}}(Z,I)\rightarrow \mathrm{Hom}_{\mathrm{XNrm}_{R}}(Y,I)\rightarrow \mathrm{Hom}_{\mathrm{XNrm}_{R}}(X,I)\rightarrow 0.$$
       By definition then, the sequence in question is exact in $(\mathrm{XNrm}_{R},\mathcal{Q}_{\kappa})$.
\end{proof}
\begin{corollary}\label{cor:qisksmall}
    Let $f:X\rightarrow Y$ be a morphism in $\mathrm{Ch}(\mathrm{XNrm}_{R})$ that is a quasi-isomorphism relative to $(\mathrm{XNrm}_{R},\mathcal{Q}_{max})$. If all $X_{n}$ and all $Y_{n}$ are $\kappa$-small, then $f$ is a quasi-isomorphism relative to $(\mathrm{XNrm}_{R},\mathcal{Q}_{\kappa})$.
\end{corollary}
\begin{proposition}
    The categories $$(\mathrm{XNrm}_{R}^{\le1},\mathcal{Q}^{\le1}_{\kappa})$$
    and
    $$(\mathrm{XNrm}_{R},\mathcal{Q}_{\kappa})$$
    have enough injectives. 
\end{proposition}
\begin{proof}
 The fact that $(\mathrm{XNrm}_{R}^{\le1},\mathcal{Q}^{\le1}_{\kappa})$ has enough injectives follows from Proposition \ref{prop:smallob}. For the existence of enough injectives in $(\mathrm{XNrm}_{R},\mathcal{Q}_{\kappa})$, it suffices to prove that an admissible monomorphism in $(\mathrm{XNrm}_{R}^{\le1},\mathcal{Q}^{\le1}_{\kappa})$ is also an admissible monomorphism in $(\mathrm{XNrm}_{R},\mathcal{Q}_{\kappa})$. Now, an admissible monomorphism $(\mathrm{XNrm}_{R}^{\le1},\mathcal{Q}^{\le1}_{\kappa})$ is just a retract of a transfinite composition of maps in $\mathbf{AdMon}_{\kappa}$. Any rescaling of such a morphism is also a retract of a transfinite composition of maps in $\mathbf{AdMon}_{\kappa}$. Let $f:X\rightarrow Y$ be an admissible monomorphism in $(\mathrm{XNrm}_{R}^{\le1},\mathcal{Q}^{\le1}_{\kappa})$, let $I\in\ I_{\kappa}$, and let $h:X\rightarrow I$ be a morphism. By rescaling, we may assume that $h:X\rightarrow I$ has norm at most $1$. Then in $\mathrm{XNrm}_{R}^{\le1}$ we may find an extension of $h$ to $\tilde{h}:Y\rightarrow I$. This shows that $f$ is an admissible monomorphism in $(\mathrm{XNrm}_{R},\mathcal{Q}_{\kappa})$, as required.
\end{proof}
From Proposition \ref{prop:Injchangequill}, we immediately get the following.
\begin{corollary}
    Let $\gamma\ge\kappa$. Then we get a reflective adjunction
    $$\adj{\lambda_{\kappa,\gamma}}{\mathrm{D}_{-}(\mathrm{XNrm}_{R},\mathcal{Q}_{\kappa})}{\mathrm{D}_{-}(\mathrm{XNrm}_{R},\mathcal{Q}_{\gamma})}{\mathbb{R}i_{\kappa,\gamma}}.$$
\end{corollary}

\subsubsection{Injective resolutions and derived functors}

Let $R$ and $S$ be seminormed rings, and let 
$$\adj{F}{\mathrm{XNrm}^{\le1}_{R}}{\mathrm{XNrm}^{\le1}_{S}}{G}$$
be an adjunction, with $F$ being exact.

\begin{definition}
    Let $M\in\mathrm{XNrm}^{\le1}_{S}$ and let $\kappa$ be a cardinal. We define
    $$\mathbb{R}_{\kappa}^{i}G(M)$$
    to be $H_{-i}(R(I_{\bullet}))$, where $M\rightarrow I_{\bullet}$ is an injective resolution in $(\mathrm{XNrm}^{\le1}_{S},Q_{\kappa}^{\le1})$.
\end{definition}

Let $M\rightarrow I_{\bullet}$ and $M\rightarrow J_{\bullet}$ be two injective resolutions of an object $M$ in $(\mathrm{XNrm}^{\le1}_{S},Q_{\kappa}^{\le1})$. By Proposition \ref{prop:injectivelift}, we get commutative diagrams of complexes
\begin{displaymath}
    \xymatrix{
    M\ar@{=}[d]\ar[r] & I_{\bullet}\ar[d]^{u}\\
    M\ar[r] & J_{\bullet}
}
\end{displaymath}

and

\begin{displaymath}
    \xymatrix{
    M\ar@{=}[d]\ar[r] & J_{\bullet}\ar[d]^{v}\\
    M\ar[r] & I_{\bullet}.
    }
\end{displaymath}

Now, consider the \textcolor{blue} category $(\mathrm{XNrm}_{S},Q_{\kappa})$. The diagrams above are diagrams of injective resolutions in $(\mathrm{XNrm}_{S},Q_{\kappa})$. By the dual of \cite{Buehler} Theorem 12.4, $u\circ v$ and $v\circ u$ are homotopy equivalent to the identity in $\mathrm{XNrm}_{S}$. Applying $G$, we get commutative diagrams
\begin{displaymath}
    \xymatrix{
    G(M)\ar@{=}[d]\ar[r] & G(I_{\bullet})\ar[d]^{G(u)}\\
    G(M)\ar[r] & G(J_{\bullet})
}
\end{displaymath}

and

\begin{displaymath}
    \xymatrix{
    G(M)\ar@{=}[d]\ar[r] & G(J_{\bullet})\ar[d]^{G(v)}\\
    G(M)\ar[r] & G(I_{\bullet}),
}
\end{displaymath}

with $G(u)\circ G(v)$ and $G(v)\circ G(u)$ also being homotopy equivalent to the identity. In particular, on the level of homology,  $G(u)\circ G(v)$ and $G(v)\circ G(u)$ are both equal to the identity. Since both $G(u)$ and $G(v)$ are non-expanding, so are the induced maps on homology. Consequently, we get isometries
$$H_{-i}(G(I_{\bullet}))\cong H_{-i}(G(J_{\bullet}))$$
for each non-negative integer $i$.

In particular, up to isometry, $\mathbb{R}_{\kappa}^{i}G(M)$ is independent of the choice of injective resolution.

\begin{remark}
   Let $\gamma\ge\kappa$. Let $M\rightarrow I_{\bullet}$ be an injective resolution in $(\mathrm{XNrm}_{S}^{\le1},Q_{\kappa}^{\le1})$, and let $M\rightarrow J_{\bullet}$ be an injective resolution in $(\mathrm{XNrm}_{S}^{\le1},Q_{\gamma}^{\le1})$. Then, by Remark \ref{rem:decreasing} and Proposition \ref{prop:injectivelift}, there is a map of complexes $I_{\bullet}\rightarrow J_{\bullet}$ lifting the identity on $M$. Applying the functor $G$ to these complexes gives non-expanding maps 
   $$\mathbb{R}^{i}_{\kappa}G(M)\rightarrow\mathbb{R}^{i}_{\gamma}G(M)$$
   for each $i$. 
\end{remark}
\subsubsection{Shrinking quasi-isomorphisms}
Next we will analyse more closely how $\mathcal{Q}_{\kappa}$ relates to $\mathcal{Q}_{max}$. 
Throughout this subsection we fix an uncountable regular cardinal $\kappa\ge |R|$ that satisfies $X^{\aleph_{0}}<\kappa$ whenever $X<\kappa$. 
\begin{proposition}\label{prop:subadepi}
    Let $f:X\rightarrow Y$ be a strict epimorphism in $(\mathrm{XNrm}_{R},\mathcal{Q}_{max})$. Let $C>0$ be such that for all $x\in X$, $\mathrm{inf}_{k\in\mathrm{Ker}(f)}\rho_{X}(x+k)\le C\rho_{Y}(f(x)),$ and let $\tilde{Y}\subset Y$ be a submodule (closed in the normed and Banach cases) that is $\kappa$-small. Let $X'\subset X$ be a $\kappa$-small submodule such that $f(X')\subset\tilde{Y}$. Then there is a $\kappa$-small submodule $\tilde{X}\subset f^{-1}(\tilde{Y})$ containing $X'$ (closed in the normed and Banach cases), such that the restriction $f:\tilde{X}\rightarrow \tilde{Y}$ is surjective, and $\mathrm{inf}_{k\in\mathrm{Ker}(f)\cap\tilde{X}}\rho_{X}(x+k)\le C\rho_{Y}(f(x)).$
\end{proposition}
\begin{proof}
   Clearly, by pulling back, we may just replace $Y$ be $\tilde{Y}$ at the outset. Let $y\in Y$. Pick a sequence $(x^{y}_{n})_{n=1}^{\infty}$ of elements of $f^{-1}(\{y\})$ such that $(\rho_{X}(x_{n}^{y}))_{n=1}^{\infty}$ converges to $\mathrm{inf}_{k\in\mathrm{Ker}(f)}\rho_{X}(x+k)$. Let $\tilde{X}$ be the submodule spanned by $X'$ and $\mathrm{Span}\{x^{y}_{n}:y\in Y,1\le n<\infty\}$ (or its closure in the normed and Banach cases). Then $\tilde{X}$ is $\kappa$-small, and satisfies the required condition by construction.
\end{proof}
\begin{proposition}\label{prop:replacesmallac}
Let $\mathrm{XNrm}_{R}$ to be either $\mathrm{SNrm}_{R}$ or $\mathrm{Ban}_{R}$.
   Let $Y\in\mathrm{Ch}_{\le0}(\mathrm{XNrm}_{R})$ be acyclic for the quasi-abelian exact structure. For each $Y_{n}$ let $X_{n}$ be a $\kappa$-small subobject. Then there is an acyclic subcomplex $\tilde{Y}\subset Y$, such that each $\tilde{Y}_{n}$ is $\kappa$-small and contains $X_{n}$.
\end{proposition}
\begin{proof}
We first consider the seminormed case.
   We will construct by induction a diagram of subcomplexes $0=\tilde{Y}^{0}\subset\ldots\tilde{Y}^{m}\subset\tilde{Y}^{m+1}\subset\ldots\subset  Y$ 
   with each $\tilde{Y}^{m}$ being acyclic, each $\tilde{Y}^{m}_{n}$ being $\kappa$-small, each $\tilde{Y}^{m}_{n}$ containing $X_{n}$ for $n>-m$, and each $\tilde{Y}^{m}_{n}=0$ for $n<-m$. This diagram will be constructed so that the colimit $\tilde{Y}$ (or, degree-wise, the union) is still acyclic. We take $\tilde{Y}^{0}=0$. Suppose $\tilde{Y}^{m}$ has been defined.  Define $\tilde{Y}_{-m}^{m+1}$ using Proposition \ref{prop:subadepi} applied to the admissible epimorphism $d_{-m}:Y_{-m}\rightarrow Z_{-m-1}Y$, the submodule of $Y_{-m}$ generated by $X_{-m}$ and $\tilde{Y}^{m}_{-m}$ and the submodule of $Z_{-m-1}Y$ given the image of $X_{-m}$ under $d_{-m}$ (or the closure of the image in the normed and Banach settings). Note that the image of $X_{-m}$ under $d_{-m}$ coincides with the image of the submodule generated by $X_{-m}$ and $\tilde{Y}^{m}_{-m}$, since $d_{-m}|_{\tilde{Y}^{m}_{-m}}=0$. We define $\tilde{Y}^{m+1}_{-m-1}$ to be the image of $X_{-m}$ under $d_{-m}$, and we define $\tilde{Y}^{m+1}_{n}=0$ for $n<-m-1$. We define $\tilde{Y}^{m+1}_{n}$ containing $\tilde{Y}^{m}_{n}$ for $0\ge n>-m$ inductively. Indeed, suppose $\tilde{Y}^{m+1}_{n}$ has been defined for some $n\ge -m$. Define $\tilde{Y}^{m+1}_{n+1}$ by applying Proposition \ref{prop:subadepi} to the morphism $d_{n+1}:Y_{n+1}\rightarrow Z_{n}Y$, the submodule of $Y_{n+1}$ given by $\tilde{Y}^{m}_{n+1}$, and the submodule of $Z_{n}Y$ given by $\tilde{Y}^{m+1}_{n}\cap Z_{n}Y=\mathrm{Ker}(d_{n}^{Y}|_{\tilde{Y}^{m+1}_{n}})$. Note that, by construction, this complex is acyclic, as $\tilde{Y}^{m+1}_{n+1}\rightarrow\mathrm{Ker}(d_{n}^{Y}|_{\tilde{Y}^{m+1}_{n}})$ is constructed to be a strict epimorphism for all $n$. Now, the degree-wise union $\tilde{Y}$ is still acyclic, as we have ensured that the strictness bound $C$ is the same at all stages. In the Banach case, we just take the closure of $\tilde{Y}$ termwise in $Y$. This is still acyclic, as taking the closure coincides with completion, and completion is an exact functor.
\end{proof}
\begin{corollary}\label{cor:replacesmall}
Let $\mathrm{XNrm}_{R}$ be either $\mathrm{SNrm}_{R}$ or $\mathrm{Ban}_{R}$.
    Let $f:K\rightarrow X$ be a quasi-isomorphism with $K$ and $X$ being non-positively graded, and each $X_{n}$ is $\kappa$-small. Then there is a subcomplex $\tilde{K}\subset K$, that may be chosen to be closed in the normed and Banach cases, such that each $\tilde{K}_{n}$ is $\kappa$-small, and $\tilde{K}\rightarrow X$ is still a quasi-isomorphism.
\end{corollary}
\begin{proof}
Apply the previous proposition to $Y=\mathrm{cone}(f)$, and the subobjects $\{(x,0):x\in X_{n}\}\subset\mathrm{cone}(f)_{n}$. This gives an acyclic complex $\widetilde{\mathrm{cone}(f)}\subset\mathrm{cone}(f)$. Consider $\widetilde{\mathrm{cone}(f)}_{n}\subset X_{n}\oplus K_{n-1}$. Let $\tilde{K}_{n-1}=\{k\in K_{n-1}:(0,k)\in\widetilde{\mathrm{cone}(f)}_{n}\}$. Note  This is evidently a (closed) submodule of $K_{n-1}$ that is $\kappa$-small. Since $\{(x,0):x\in X_{n}\}\subset\widetilde{\mathrm{cone}}(f)_{n}$, we have $\widetilde{\mathrm{cone}(f)}_{n}=X_{n}\oplus\tilde{K}_{n-1}$, and $\tilde{K}_{n-1}=\{k\in K_{n-1}:\exists x\in X_{n}:(x,k)\in\widetilde{\mathrm{cone}(f)}_{n}\}$. By definition of the differential on $\mathrm{cone}(f)$, we then have that $\tilde{K}\subset K$ is a subcomplex. Let $\tilde{f}:\tilde{K}\rightarrow X$ denote the restriction of $f$ to $\tilde{K}$. This is a morphism of complexes, and $\widetilde{\mathrm{cone}}(f)=\mathrm{cone}(\tilde{f})$. In particular, $\tilde{f}$ is a quasi-isomorphism.  
\end{proof}
\begin{corollary}\label{corcor:kapparestrictderived}
Let $\mathrm{XNrm}_{R}$ be either $\mathrm{SNrm}_{R}$ or $\mathrm{Ban}_{R}$.
    Let $X$ be a bounded above complex with each $X_{n}$ being $\kappa$-small. Let $Y$ be an arbitrary bounded above complex. Then the natural map
    $$\mathrm{Hom}_{\mathrm{D}(\mathrm{XNrm}_{R},\mathcal{Q}_{\kappa})}(X,Y)\rightarrow\mathrm{Hom}_{\mathrm{D}(\mathrm{XNrm}_{R},\mathcal{Q}_{max})}(X,Y)$$
    is an equivalence.
    In particular, if $Y\rightarrow I$ is a resolution by a complex of injective objects in $(\mathrm{XNrm}_{R},\mathcal{Q}_{\kappa})$, then we have 
    $$\mathrm{Hom}_{\mathrm{D}(\mathrm{XNrm}_{R},\mathcal{Q}_{\kappa})}(X,I)\cong \mathrm{Hom}_{\mathrm{K}(\mathrm{XNrm}_{R})}(X,I)$$
\end{corollary}
\begin{proof}
Consider the equivalence class of a span 
\begin{displaymath}
    \xymatrix{
    X & K\ar[l]^{\alpha}\ar[r]^{\beta} & Y
    }
\end{displaymath}
in $\mathrm{Ch}_{-}(\mathrm{XNrm}_{R})$, with $\alpha$ being a quasi-isomorphism relative to the quasi-abelian exact structure. By Proposition \ref{cor:replacesmall}, there is a commutative diagram
\begin{displaymath}
    \xymatrix{
    & \tilde{K}\ar[dl]^{\tilde{\alpha}}\ar[d]^{i}\ar[dr] & \\
    X & K\ar[l]^{\alpha}\ar[r]^{\beta} & Y
    }
\end{displaymath}
with each $\tilde{K}_{n}$ being $\kappa$-small, and $i$ being a quasi-isomorphism. Then $\tilde{\alpha}:\tilde{K}\rightarrow X$ is a quasi-isomorphism in $(\mathrm{XNrm}_{R},\mathcal{Q}_{max})$. Since all $X_{n}$ and $\tilde{K}_{n}$ are $\kappa$-small, $\tilde{\alpha}$ is in fact a quasi-isomorphism in $(\mathrm{XNrm}_{R},\mathcal{Q}_{\kappa})$ by Corollary \ref{cor:qisksmall}. This shows that  $$\mathrm{Hom}_{\mathrm{D}(\mathrm{XNrm}_{R},\mathcal{Q}_{\kappa})}(X,Y)\rightarrow \mathrm{Hom}_{\mathrm{D}(\mathrm{XNrm}_{R},\mathcal{Q}_{max})}(X,Y)$$
is surjective. It remains to prove that it is injective. Suppose that a roof
\begin{displaymath}
    \xymatrix{
    X & K\ar[l]^{\alpha}\ar[r]^{\beta} & Y
    }
    \end{displaymath}
in $\mathrm{D}_{-}(\mathrm{XNrm}_{R},\mathcal{Q}_{\kappa})$ becomes zero in $\mathrm{D}_{-}(\mathrm{XNrm}_{R},\mathcal{Q}_{max})$. Then there is a quasi-isomorphism $K'\rightarrow K$ such that the composite $K'\rightarrow K\rightarrow Y$ is null-homotopic. Another application of Corollary \ref{cor:replacesmall} to the composition $K'\rightarrow K\rightarrow X$ means we can assume that each $K'_{n}$ is $\kappa$-small. Then the roof above is already $0$ in $\mathrm{D}_{-}(\mathrm{XNrm}_{R},\mathcal{Q}_{\kappa})$. 
\end{proof}
\begin{theorem}\label{thm:bancolimcard}
The category $\mathrm{D}_{-}(\mathrm{XNrm}_{R},\mathcal{Q}_{\max})$ admits a presentation
    $$\colim_{\kappa}\mathrm{D}_{-}(\mathrm{XNrm}_{R},\mathcal{Q}_{\kappa})\cong\mathrm{D}_{-}(\mathrm{XNrm}_{R},\mathcal{Q}_{\max}).$$
\end{theorem}
\begin{proof}
Note that the colimit here is a strict colimit in the $1$-category of categories.
Using Remark \ref{rem:cofinalcofinal}, we may instead only consider the colimit over cardinals $\kappa$ that are regular, and satisfy $X^{\aleph_{0}}<\kappa$ whenever $X<\kappa$. For each $\kappa$ we have the exact identity functor $\mathrm{Id}_{\kappa}:(\mathrm{XNrm}_{R},\mathcal{Q}_{\kappa})\rightarrow(\mathrm{XNrm}_{R},\mathcal{Q}_{\max})$, and therefore a functor $\lambda_{\kappa}:\mathrm{D}_{-}(\mathrm{XNrm}_{R},\mathcal{Q}_{\kappa})\rightarrow\mathrm{D}_{-}(\mathrm{XNrm}_{R},\mathcal{Q}_{\max})$. These functors are compatible with the transition functors $\lambda_{\kappa,\gamma}$ for $\gamma\ge\kappa$.
Let $F_{\kappa}:\mathrm{D}_{-}(\mathrm{XNrm}_{R},\mathcal{Q}_{\kappa})\rightarrow\mathrm{C}$ be a compatible family of functors. Define $\overline{F}:\mathrm{D}_{-}(\mathrm{XNrm}_{R},\mathcal{Q}_{max})\rightarrow\mathrm{C}$ as follows. Let $X_{\bullet}$ be a bounded above complex, and let $\kappa>\mathrm{sup}_{n\in\mathbb{Z}}\{|X_{n}|\}$. Define $\overline{F}(X_{\bullet})=F_{\kappa}(X_{\bullet})$. By the compatibility assumptions, this is well-defined. Next, let $f:X\rightarrow Y$ be a morphism in $\mathrm{D}_{-}(\mathrm{XNrm}_{R},\mathcal{Q}_{max})$. We may assume that it is the image of a (unique by Corollary \ref{corcor:kapparestrictderived}) morphism $f:X\rightarrow Y$ in $\mathrm{D}_{-}(\mathrm{XNrm}_{R},\mathcal{Q}_{\kappa})$. Define $\overline{F}(f)=F_{\kappa}(f)$. By compatibility of the functors $F_{\kappa}$, this is independent of the choice of $\kappa$. Functoriality follows from the fact that, for any two composable morphisms, we may choose $\kappa$ large enough so that all domains and codomains are $\kappa$-small. Hence, $\overline{F}$ is a functor. By construction, we have $\overline{F}\circ\lambda_{\kappa}=F_{\kappa}$. The functor $\overline{F}$ is clearly unique amongst functors $G:\mathrm{D}_{-}(\mathrm{XNrm}_{R},\mathcal{Q}_{max})\rightarrow\mathrm{C}$ such that $G\circ\lambda_{\kappa}=F_{\kappa}$.
\end{proof}
\subsubsection{Projectives}
It is well known (\cite{dang} Proposition 5.1.10) that the category $(\mathrm{XNrm}_{R},\mathcal{Q}_{max})$ has enough projectives. Indeed, any non-expanding coproduct $\coprod^{\le1}_{i\in\mathcal{I}}R_{\delta_{i}}$ is projective in $(\mathrm{XNrm}_{R},\mathcal{Q}_{max})$. There exists a strict epimorphism $\coprod^{\le1}_{m\in M\setminus\{0\}}R_{\rho(m)}\rightarrow M$ sending $(r_{m})_{m\in M\setminus\{0\}}$ to $\sum_{m\in M\setminus\{0\}}r_{m}m$. 
Since the identity $(\mathrm{XNrm}_{R},\mathcal{Q}_{\kappa})\rightarrow(\mathrm{XNrm}_{R},\mathcal{Q}_{max})$ is exact, any projective in $(\mathrm{XNrm}_{R},\mathcal{Q}_{max})$ is projective in $(\mathrm{XNrm}_{R},\mathcal{Q}_{\kappa})$. 
\begin{proposition}\label{prop:projresbig}
Let $\kappa\ge\mathrm{max}\{|R|,\aleph_{1}\}$ be a regular cardinal such that $X^{\aleph_{0}}$ is $\kappa$-small whenever $X$ is $\kappa$-small.
  Let $X_{\bullet}$ be a bounded below complex, $\kappa\ge |X_{n}|$ for all $n$. Then there exists a projective resolution $P_{\bullet}\rightarrow X_{\bullet}$ in $(\mathrm{XNrm}_{R},\mathcal{Q}_{max})$ with each $P_{n}$ being $\kappa$-small. In particular, $P_{\bullet}\rightarrow X_{\bullet}$ is a projective resolution in $(\mathrm{XNrm}_{R},\mathcal{Q}_{\kappa})$. 
\end{proposition}
\begin{proof}
    One constructs the projective resolution in the usual way, using the canonical projective covers $\coprod^{\le1}_{m\in M\setminus\{0\}}R_{\rho(m)}\rightarrow M$ at each stage. The size assumption and Corollary \ref{cor:adepik} guarantee that we always remain $\kappa$-small.
\end{proof}
\subsubsection{Derived functors}
Let 
$$\adj{F}{\mathrm{XNrm}_{R}}{\mathrm{XNrm}_{S}}{G}$$
ben adjunction, where $F$ is exact for the quasi-abelian exact structures. 
Suppose further that there is $\lambda$ such that, for all $\kappa\ge\lambda$, $F(X)$ is $\kappa$-small whenever $X$ is $\kappa$-small. Then, for all $\kappa\ge\lambda$, $F:(\mathrm{XNrm}_{R},\mathcal{Q}_{\kappa})\rightarrow(\mathrm{XNrm}_{S},\mathcal{Q}_{\kappa})$
is exact. 
Then 
$$\adj{F}{\mathrm{Ch}_{-}(\mathrm{XNrm}_{R},\mathcal{Q}_{\kappa})}{\mathrm{Ch}_{-}(\mathrm{XNrm}_{S},\mathcal{Q}_{\kappa})}{G}$$
is a Quillen adjunction for the injective model structures, so that we get an induced adjunction
$$\adj{\mathbb{L}_{\kappa}F}{\mathrm{D}_{-}(\mathrm{XNrm}_{R},\mathcal{Q}_{\kappa})}{\mathrm{D}_{-}(\mathrm{XNrm}_{S},\mathcal{Q}_{\kappa})}{\mathbb{R}_{\kappa}G}.$$
Since the functor $F:(\mathrm{XNrm}_{R},\mathcal{Q}_{max})\rightarrow(\mathrm{XNrm}_{S},\mathcal{Q}_{max})$ is exact, we also get an induced functor
$$\mathbb{L}F:\mathrm{D}_{-}(\mathrm{XNrm}_{R},\mathcal{Q}_{max})\rightarrow\mathrm{D}_{-}(\mathrm{XNrm}_{S},\mathcal{Q}_{max}).$$
Our next goal is to give examples where this functor has a right adjoint.
\begin{proposition}\label{prop:sestrunc}
    Let $\mathpzc{E}$ be a finitely complete and finitely cocomplete exact category with enough injectives. Let $L$ be an object of $\mathrm{Ch}_{-}(\mathpzc{E})$. Then, for any object $X$ of $\mathrm{D}_{-}(\mathpzc{E})$, we have a short exact sequence
$$0\rightarrow\mathbb{R}^{1}\lim\mathrm{Hom}_{\mathrm{D}_{-}(\mathpzc{E})}(\tau_{\ge n}L,X[-1])\rightarrow \mathrm{Hom}_{\mathrm{D}_{-}(\mathpzc{E})}(L,X)\rightarrow \lim\mathrm{Hom}_{\mathrm{D}_{-}(\mathpzc{E})}(\tau_{\ge n}L,X)\rightarrow 0$$
natural in $L$ and $X$.
\end{proposition}
\begin{proof}
 Without loss of generality, we may assume that $X$ is a complex of injectives. Since it is bounded above, $X$ is then $K$-injective. Then, in all terms in the desired exact sequence, we may replace $\mathrm{D}_{-}(\mathpzc{E})$ by $\mathrm{K}_{-}(\mathpzc{E})$, where $\mathrm{K}_{-}(\mathpzc{E})$ denotes the homotopy category of complexes. Consider the category $(\mathrm{Ind}(\mathpzc{E}),\mathcal{Q}_{split})$ equipped with the split-exact structure. This category has exact direct sums. Thus, for any object $M$ of this category we have, by \cite{krause2005derived}*{Section 4.4, Example (2)} and \cite{beligiannis2000relative}*{Lemma 5.8 (2)}, an exact sequence 
 \begin{align*}
     0&\rightarrow\mathbb{R}^{1}\lim\mathrm{Hom}_{\mathrm{D}_{-}(\mathrm{Ind}(\mathpzc{E}),\mathcal{Q}_{split})}(\tau_{\ge n}M,X[-1])\\
     &\rightarrow \mathrm{Hom}_{\mathrm{D}_{-}(\mathrm{Ind}(\mathpzc{E}),\mathcal{Q}_{split})}(M,X)\\
     &\rightarrow \lim\mathrm{Hom}_{\mathrm{D}_{-}(\mathrm{Ind}(\mathpzc{E}),\mathcal{Q}_{split})}(\tau_{\ge n}M,X)\\
     &\rightarrow 0.
 \end{align*}
For the split exact structure, this gives an exact sequence 
\begin{align*}
0&\rightarrow\mathbb{R}^{1}\lim\mathrm{Hom}_{\mathrm{K}_{-}(\mathrm{Ind}(\mathpzc{E}))}(\tau_{\ge n}M,X[-1])\\
    &\rightarrow \mathrm{Hom}_{\mathrm{K}_{-}(\mathrm{Ind}(\mathpzc{E}))}(M,X)\\
    &\rightarrow \lim\mathrm{Hom}_{\mathrm{K}_{-}(\mathrm{Ind}(\mathpzc{E}))}(\tau_{\ge n}M,X)\\
    &\rightarrow 0.
\end{align*}
We just conclude by observing that $K_{-}(\mathpzc{E})\rightarrow K_{-}(\mathrm{Ind}(\mathpzc{E}))$ is fully faithful.
\end{proof}
The following is now clear.
\begin{proposition}\label{prop:testtrunc}
    Let $(\mathpzc{E},\mathcal{Q})$ be an exact category, such that for any objects $L,X\in\mathrm{D}_{-}(\mathpzc{E},\mathcal{Q})$ there exists an exact sequence
$$0\rightarrow\mathbb{R}^{1}\lim\mathrm{Hom}_{\mathrm{D}_{-}(\mathpzc{E})}(\tau_{\ge n}L,X[-1])\rightarrow \mathrm{Hom}_{\mathrm{D}_{-}(\mathpzc{E})}(L,X)\rightarrow \lim\mathrm{Hom}_{\mathrm{D}_{-}(\mathpzc{E})}(\tau_{\ge n}L,X)\rightarrow 0$$
natural in $L$ and $X$. Let $f:X\rightarrow X'$ be a morphism. Then the morphism 
$$\mathrm{Hom}_{\mathrm{D}_{-}(\mathpzc{E})}(L,X)\rightarrow \mathrm{Hom}_{\mathrm{D}_{-}(\mathpzc{E})}(L,X')$$
is an isomorphism if and only if
$$\mathrm{Hom}_{\mathrm{D}_{-}(\mathpzc{E})}(\tau_{\ge n}L,X)\rightarrow \mathrm{Hom}_{\mathrm{D}_{-}(\mathpzc{E})}(\tau_{\ge n}L,X')$$
and
$$\mathrm{Hom}_{\mathrm{D}_{-}(\mathpzc{E})}(\tau_{\ge n}L,X[-1])\rightarrow \mathrm{Hom}_{\mathrm{D}_{-}(\mathpzc{E})}(\tau_{\ge n}L,X'[-1])$$
are isomorphisms for each $n$.
\end{proposition}
\begin{corollary}\label{cor:natseqban}
    Let $L,X$ be objects of $\mathrm{Ch}_{-}(\mathrm{XNrm}_{R})$. Then there are short exact sequences
\begin{align*}
0\rightarrow&\mathbb{R}^{1}\lim\mathrm{Hom}_{\mathrm{D}_{-}(\mathrm{XNrm_{R}},\mathcal{Q}_{\kappa})}(\tau_{\ge n}L,X[-1])\\
&\rightarrow \mathrm{Hom}_{\mathrm{D}_{-}(\mathrm{XNrm_{R}},\mathcal{Q}_{\kappa})}(L,X)\\
 &\rightarrow\lim\mathrm{Hom}_{\mathrm{D}_{-}(\mathrm{XNrm_{R}},\mathcal{Q}_{\kappa})}(\tau_{\ge n}L,X)\\
 &\rightarrow 0
 \end{align*}
 for all $\kappa$,
    and 
\begin{align*}
0\rightarrow&\mathbb{R}^{1}\lim\mathrm{Hom}_{\mathrm{D}_{-}(\mathrm{XNrm_{R}},\mathcal{Q}_{max})}(\tau_{\ge n}L,X[-1])\\
&\rightarrow \mathrm{Hom}_{\mathrm{D}_{-}(\mathrm{XNrm_{R}},\mathcal{Q}_{max})}(L,X)\\
 &\rightarrow\lim\mathrm{Hom}_{\mathrm{D}_{-}(\mathrm{XNrm_{R}},\mathcal{Q}_{max})}(\tau_{\ge n}L,X)\\
 &\rightarrow 0.
 \end{align*}
\end{corollary}
\begin{proof}
  For the $\mathcal{Q}_{\kappa}$ case, this follows immediately from Proposition \ref{prop:sestrunc}. For the $\mathcal{Q}_{max}$ case, the claim follows from the $\mathcal{Q}_{\kappa}$ case and Corollary \ref{corcor:kapparestrictderived}, by choosing $\kappa$ sufficiently large. 
\end{proof}
\begin{proposition}
    Let $R\rightarrow S$ be a map of Banach rings, with $R$ commutative, and the image contained in the centre of $S$. Let $\kappa\ge\mathrm{max}\{|R|,|S|\}$ be an uncountable regular cardinal such that $X^{\aleph_{0}}$ is $\kappa$-small whenever $X$ is $\kappa$-small. Let $f:F_{\bullet}\rightarrow G_{\bullet}$ be a quasi-isomorphism between bounded below complexes of $S$-modules such that $F_{n}$ and $G_{n}$ are all $\kappa$-small. Let $L_{\bullet}$ be a bounded above complex of $R$-modules such that $L_{n}$ is $\kappa$-small for all $n$. Let $I_{\bullet}$ be a bounded above complex of $\mathcal{Q}_{\kappa}$-injectives. Then the natural map
    $$\mathrm{Hom}_{\mathrm{D}_{-}(\mathrm{XNrm}_{R},\mathcal{Q}_{\kappa})}(L,\underline{\mathrm{Hom}}_{S}(G_{\bullet},I)))\rightarrow \mathrm{Hom}_{\mathrm{D}_{-}(\mathrm{XNrm}_{R},\mathcal{Q}_{\kappa})}(L,\underline{\mathrm{Hom}}_{S}(F_{\bullet},I)))$$
    is an isomorphism. In particular, $\underline{\mathrm{Hom}}_{S}(G_{\bullet},I)\rightarrow \underline{\mathrm{Hom}}_{S}(F_{\bullet},I)$ is an isomorphism in $\mathrm{D}_{-}(\mathrm{XNrm}_{R},\mathcal{Q}_{\kappa})$.
\end{proposition}
\begin{proof}
   By Proposition \ref{prop:testtrunc} and Corollary \ref{cor:natseqban} we may assume that $L$ is in fact bounded. We may then find a projective resolution in $(\mathrm{XNrm}_{R},\mathcal{Q}_{max})$, $P_{\bullet}\rightarrow L$, with each $P_{n}$ being $\kappa$-small. Then this is also a projective resolution in $(\mathrm{XNrm}_{R},\mathcal{Q}_{\kappa})$. Let $\widetilde{\mathrm{Hom}}_{R}(-,-)$ and $\widetilde{\mathrm{Hom}}_{S}(-,-)$ denote the respective hom complex functors. We have 
    \begin{align*}
        \mathrm{Hom}_{\mathrm{D}_{-}(\mathrm{XNrm}_{R},\mathcal{Q}_{\kappa})}(L,\underline{\mathrm{Hom}}_{S}(G_{\bullet},I)))&\cong\mathrm{Hom}_{\mathrm{D}(\mathrm{XNrm}_{R},\mathcal{Q}_{\kappa})}(P_{\bullet},\underline{\mathrm{Hom}}_{S}(G_{\bullet},I)))\\
         &\cong\pi_{0}\widetilde{\mathrm{Hom}}_{R}(P_{\bullet},\underline{\mathrm{Hom}}_{S}(G_{\bullet},I)))\\
        &\cong\pi_{0}\widetilde{\mathrm{Hom}}_{S}(P_{\bullet}\otimes_{R}G_{\bullet},I)\\
        &\cong\pi_{0}\widetilde{\mathrm{Hom}}_{S}(P_{\bullet}\otimes_{R}F_{\bullet},I)\\
        &\cong\mathrm{Hom}_{\mathrm{D}_{-}(\mathrm{XNrm}_{R},\mathcal{Q}_{\kappa})}(L,\underline{\mathrm{Hom}}_{S}(F_{\bullet},I))).
          \end{align*}
          For the $4$th isomorphism above, we argue as follows. Since $P_{\bullet}$ is a bounded below complex of flat objects, we have that $P_{\bullet}\otimes F_{\bullet}\rightarrow P_{\bullet}\otimes G_{\bullet}$ is a quasi-isomorphism. Moreover, these complexes both consist of $\kappa$-small objects. In particular, this is a quasi-isomorphism in $(\mathrm{XNrm}_{R},\mathcal{Q}_{\kappa})$. Since $I$ is a bounded above complex of injectives, 
$$\widetilde{\mathrm{Hom}}(P_{\bullet}\otimes G_{\bullet},I)\rightarrow\widetilde{\mathrm{Hom}}(P_{\bullet}\otimes F_{\bullet},I)$$
is an equivalence.
\end{proof}
\begin{example}
    Let $R\rightarrow S$ be a morphism of seminormed rings such that $R$ is a commutative, and the image of $R$ is contained in the centre of $S$. Let $\kappa\ge\mathrm{max}\{|R|,|S|\}$ be an uncountable regular cardinal such that $X^{\aleph_{0}}<\kappa$ whenever $X<\kappa$. Let $B$ be a right $S$-module, that we may think of as an $(R,S)$-bimodule. Suppose that $B$ is flat as an $R$-module, and admits a projective resolution $P_{\bullet}\rightarrow B$ in $(\mathrm{XNrm}_{R},\mathcal{Q}_{max})$ with each $P_{n}$ being $\kappa$-small. Note that $P_{\bullet}\rightarrow B$ is also a projective resolution in $(\mathrm{XNrm}_{R},\mathcal{Q}_{\kappa})$. We consider the exact functor
    $$F\defeq(-)\otimes_{R}B:(\mathrm{XNrm}_{R},\mathcal{Q}_{max})\rightarrow(\mathrm{XNrm}_{S},\mathcal{Q}_{max})$$
    and the induced functor
        $$\mathbb{L}F:\mathrm{D}_{-}(\mathrm{XNrm}_{R},\mathcal{Q}_{max})\rightarrow \mathrm{D}_{-}(\mathrm{XNrm}_{S},\mathcal{Q}_{max}).$$
        We also define
        $$\mathbb{R}G:\mathrm{D}_{-}(\mathrm{XNrm}_{S},\mathcal{Q}_{max})\rightarrow\mathrm{D}_{-}(\mathrm{XNrm}_{R},\mathcal{Q}_{max})$$
by 
$$M\mapsto\underline{\mathrm{Hom}}_{S}(P_{\bullet},M).$$
Note that $B\otimes_{R} M$ is $\kappa$-small whenever $M$ is $\kappa$-small.
Let $M\in\mathrm{XNrm}_{R}$ be $\kappa$-small. Let also $M\rightarrow I_{\bullet}$ be an injective resolution in  $(\mathrm{XNrm}_{S},\mathcal{Q}_{\kappa})$. We then have quasi-isomorphisms
$$\underline{\mathrm{Hom}}_{S}(P_{\bullet},M)\cong\underline{\mathrm{Hom}}_{S}(P_{\bullet},I_{\bullet})\cong\underline{\mathrm{Hom}}_{S}(B,I_{\bullet}),$$
 where the first quasi-isomorphism is in $(\mathrm{XNrm}_{R},\mathcal{Q}_{max})$, and the second is in $(\mathrm{XNrm}_{R},\mathcal{Q}_{\kappa})$. Let $L\in\mathrm{D}_{-}(\mathrm{XNrm},\mathcal{Q}_{\kappa})$ with each $L_{n}$ $\kappa$-small. We then compute
\begin{align*}
    \mathrm{Hom}_{\mathrm{D}_{-}(\mathrm{XNrm}_{R},\mathcal{Q}_{max})}(L,\mathbb{R}G(M))&\cong\mathrm{Hom}_{\mathrm{D}_{-}(\mathrm{XNrm}_{R},\mathcal{Q}_{max})}(L,\underline{\mathrm{Hom}}_{S}(P_{\bullet},M))\\
    &\cong\mathrm{Hom}_{\mathrm{D}_{-}(\mathrm{XNrm}_{R},\mathcal{Q}_{max})}(L,\underline{\mathrm{Hom}}_{S}(P_{\bullet},I_{\bullet}))\\
    &\cong\mathrm{Hom}_{\mathrm{D}_{-}(\mathrm{XNrm}_{R},\mathcal{Q}_{\kappa})}(L,\underline{\mathrm{Hom}}_{S}(P_{\bullet},I_{\bullet}))\\
    &\cong\mathrm{Hom}_{\mathrm{D}_{-}(\mathrm{XNrm}_{R},\mathcal{Q}_{\kappa})}(L,\underline{\mathrm{Hom}}_{S}(B,I_{\bullet}))\\
    &\cong\mathrm{Hom}_{\mathrm{D}_{-}(\mathrm{XNrm}_{R},\mathcal{Q}_{\kappa})}(L,\mathbb{R}G_{\kappa}M)\\
    &\cong\mathrm{Hom}_{\mathrm{D}_{-}(\mathrm{XNrm}_{S},\mathcal{Q}_{\kappa})}(\mathbb{L}F_{\kappa}L,M)\\
    &\cong\mathrm{Hom}_{\mathrm{D}_{-}(\mathrm{XNrm}_{S},\mathcal{Q}_{\kappa})}(F(L),M)\\
    &\cong\mathrm{Hom}_{\mathrm{D}_{-}(\mathrm{XNrm}_{S},\mathcal{Q}_{max})}(F(L),M)\\
     &\cong\mathrm{Hom}_{\mathrm{D}_{-}(\mathrm{XNrm}_{S},\mathcal{Q}_{max})}(\mathbb{L}F(L),M). 
\end{align*}
Since this works for all sufficiently large $\kappa$, we get an adjunction
$$\adj{\mathbb{L}F}{\mathrm{D}_{-}(\mathrm{XNrm}_{R},\mathcal{Q}_{max})}{\mathrm{D}_{-}(\mathrm{XNrm}_{S},\mathcal{Q}_{max})}{\mathbb{R}G}.$$
\end{example}
\begin{example}
     Let $(G,\rho)$ be a seminormed group. By this we mean that $G$ is a group, and $\rho:G\rightarrow\mathbb{R}_{\ge0}$ is a function such that $\rho(g_{1}g_{2})\le\rho(g_{1})\rho(g_{2})$ for all $g_{1},g_{2}\in G$, and $\rho(e_{G})=1$. Let $R$ be a commutative seminormed ring. We consider the non-commutative algebra object of $\mathrm{XNrm}_{R}$, $S=R[G]$, given by $\coprod^{\le1}_{g\in G}S_{\rho(g)}$. Multiplication is induced by the multiplication maps 
    $$R_{\rho(g_{1})}\otimes R_{\rho(g_{2})}\cong R_{\rho(g_{1})\rho(g_{2})}\rightarrow R_{\rho(g_{1}g_{2})}.$$
    Note that there is an augmentation $R[G]\rightarrow R_{\rho(e_{G})}=R$, so that we may regard $R$ as an $S$-module. Taking $B=R$ gives a definition of bounded cohomology for groups with coefficients in arbitrary seminormed rings. 
\end{example}
\bibliographystyle{amsalpha}
\bibliography{Article2_no_color}

\end{document}